\documentclass[a4paper,12pt]{amsart}
\usepackage[utf8]{inputenc}
\usepackage{amsthm}
\usepackage{amsmath}
\usepackage{bm}
\usepackage{enumitem}
\usepackage{amsfonts}
\usepackage{amssymb}
\usepackage{mathtools}
\usepackage{appendix}
\usepackage{mathrsfs}
\usepackage{setspace}
\usepackage{comment}
\usepackage{xcolor}
\usepackage{todonotes}
\usepackage{thmtools} 
\usepackage[]{amsaddr} 
\usepackage{appendix}
\RequirePackage[a4paper,top=2.54cm,bottom=2.54cm,left=1.90cm,right=1.90cm,%
                headsep=1em,includehead,includefoot]{geometry}
\usepackage[pdfdisplaydoctitle,colorlinks,breaklinks,urlcolor=blue,linkcolor=blue,citecolor=blue]{hyperref} 
\usepackage[nameinlink,capitalise]{cleveref}

\newcommand{\E}{\ensuremath{\mathbb{E}}}
\providecommand{\U}[1]{\protect\rule{.1in}{.1in}}
\newtheorem{theorem}{Theorem}

\newtheorem{lemma}[theorem]{Lemma}

\newtheorem{proposition}[theorem]{Proposition}
\newtheorem{remark}[theorem]{Remark}

\newcommand{\T}{\mathbb{T}}
\newcommand{\R}{\mathbb{R}}
\newenvironment{acknowledgements}{%
  \begin{abstract}
}{%
  \end{abstract}
}
\def\yas#1{\textcolor{blue!98!black}{#1} }

\newcommand{\Div}{\text{div}}
\title[ Limit of Fokker-Planck equation    for   polymers  in turbulent flow]
{Small-scale turbulence limit of Fokker-Planck equation    for   polymers  in turbulent flow}
\author[Yassine Tahraoui]{Yassine Tahraoui}
\address{Scuola Normale Superiore, Piazza dei Cavalieri, 7, 56126 Pisa, Italy}
\email{\href{mailto:yassine.tahraoui at sns.it}{yassine.tahraoui at sns.it}}
\date\today

\begin{document}
	\begin{abstract}
We study the singular limit of  Fokker-Planck equation of   polymers density as the dominant time-scale of small scale component of turbulent flow goes to zero. Here, we  complete the study  of Flandoli-Tahraoui   \href{https://www.sciencedirect.com/science/article/pii/S0022039625008162?via%3Dihub}{[J. Differ. Equ. (2026)]} about scaling limit as the space-scale of small scale component of turbulent flow goes to zero by using stochastic modeling of turbulence. Depending on certain parameters, related to turbulence modeling, the limit density has  generalized Cauchy distribution  for the end-to-end vector.  We discuss also the limit when we don't have a probability density limit.  Our approach is based on the 
derivation of  an appropriate estimates on $L^2$ with appropriate weight  and investigate the convergence.
	\end{abstract}
	\maketitle
\section{Introduction}
Polymers are complex molecular systems, that can be vaguely thought as a chain of springs. Understanding polymer dynamics is important, both theoretically and in practice. One application is that the presence of low concentrations of polymers can lead to a significant change in hydrodynamics, and one of the most important effects is the drag reduction in turbulence.
In a turbulent fluid,  polymers are usually found in two states called coil and stretched. The coil state is like a spherical or ellipsoidal rolled chain, which may be more or less elongated, but still in the roll position. The stretched state is when the chain is elongated, more similar to a straight line than a sphere. Strongly stretching turbulence may lead to the stretched state; when the polymer passes from one state to the other we speak of coil-stretch transition.\\

One of the aims of this contribution is to rigorously justify  the 
power law  distribution  of the end-to-end vector of polymers.
Consider in dimension $2$ or $3$ the following \textit{Hookean} model:
\begin{align}\label{Intro1}
	\begin{cases}
	dR_t&=\nabla u(X_t,t) R_tdt-\dfrac{1}{\beta}R_tdt+\sqrt{2}\sigma d\mathcal{W}_t, \\
	dX_t&=u(X_t,t)dt,
	\end{cases}
	\end{align}
where $X_t$ is the polymer position (the    center  of  mass) and $R_t$ is  the end-to-end vector, representing the orientation and elongation of the chain,  see e.g.  \cite[Section  4.2]{BriTL09}. The polymer is embedded into a fluid having velocity $u(t,x)$, which stretches $R_t$ by $\nabla u(x,t)$. The equation for $R_t$ contains also a damping (restoring) term with relaxation time $\beta$ and Brownian fluctuations 
$\sqrt{2}\sigma d\mathcal{W}_t$ where, to simplify the notation we have denoted by $\sigma^2$ the product $\frac{kT}{\beta}$, $k$ being Boltzmann constant and  $T$ being the temperature. \\

We consider a dilute (non-interacting) family of polymers subject to equations
\eqref{Intro1}, thus described by the kinetic equation for the density $f^{N,\tau}:=f^{N,\tau}\left(
x,r,t\right)  $ of polymers with position $x$ and length $r$ at time $t$%
\begin{align}\label{Intro2}
	\begin{cases}&\partial_t f^{N,\tau}+\Div_x(u^{N,\tau}f^{N,\tau})+\Div_r((\nabla u^{N,\tau}r-\dfrac{1}{\beta}r)f^{N,\tau})=	\sigma^2\Delta_r f^{N,\tau},\\	
    &u^{N,\tau}:=u^{N,\tau}\left(  x,t\right)  =u_{L}\left(  x,t\right)+u_{s}\left(  x,t\right)=u_{L}\left(  x,t\right)  +\circ\sum_{k\in K}%
\sigma_{k}^{N,\tau}\left(  x\right)  \partial_{t}W_{t}^{k},\\
		&f^{N,\tau}|_{t=0}=f_0^\tau.
			\end{cases}\end{align}

Here $u^{N,\tau}\left(x,t\right)  $ is the fluid velocity and assumed that
$u^{N,\tau}\left(  x,t\right)  $ is made of two components, a deterministic
large-scale one $u_{L}\left(  x,t\right)  $ and a stochastic one, modeling
small-scale turbulence, of the form $\sum_{k\in K}\sigma_{k}^{N,\tau}\left(
x\right)  \partial_{t}W_{t}^{k}$, acting in Stratonovich form.
We  stress the fact that the coefficients $\sigma_{k}^{N,\tau}\left(
x\right)  $ of the turbulent part depend on a parameter $N$ so that, when $N$
increases, they represent smaller and smaller space scales, precisely Fourier
frequencies $N\leq\left\vert k\right\vert \leq2N$, providing a separation of
scale regime. The noise acts on $f^{N,\tau}\left(
x,r,t\right)  $ in transport form, but incorporating also the stretching
action by the term $\nabla u^{N,\tau}\left(  x,t\right)  r$. The coefficients $(\sigma_k^{N,\tau})_k$ depends as well   on the dominant time-scale
$\tau$ of $u_s$, see  \autoref{subsection-time-scale}.
\par
In \cite{FlaTah24},  under suitable intensity assumption, such that the stretching term has a finite limit covariance, we proved that $f^{N,\tau}$ weakly converges to the solution $f_{\tau}:=f_{\tau}(x,r,t)$ of a deterministic equation,  with a new diffusion term in the $r$-variable, of the form
	\begin{align}\label{Intro3}
	\begin{cases}&\partial_t f_\tau+\Div_x(u_Lf_\tau)+\Div_r(\nabla u_Lr f_\tau))=	
 \Div_r(\dfrac{1}{\beta}r f_\tau)+\sigma^2\Delta_r f_\tau + 
 \dfrac{k_T}{2}\Div_r(A_b(r) \nabla_r f_\tau)        \\	
		&f_\tau|_{t=0}=f_0^\tau,
			\end{cases}\end{align}
   where $A_b(r)= ((b+1)\left\vert r\right\vert^2I-br\otimes r)$ ($b=2$ in 2D and $b=1$ in 3D), 
$k_T=C_da_\tau^2$ ($C_d=\frac{\pi \log(2)}{8}$ in 2D and  $C_d=\frac{8\pi \log(2)}{15}$ in 3D) and $a_\tau$ is an intensity parameter of the noise depending on $\tau$. We refer to \autoref{section-prel}, where we recall the assumptions and main results from \cite{FlaTah24} about the rigorous derivation of \eqref{Intro3},  as scaling limit of \eqref{Intro2}.
\par In the regime where the relaxation time of polymers $\beta$ and the dominant time-scale  $\tau$  satisfy $\beta=\zeta\tau,\quad \zeta>0$,  \eqref{Intro3} becomes
\begin{align}\label{sys-I}
\begin{cases}&\partial_tf_\tau+\Div_x(u_Lf_\tau)+\Div_r(\nabla u_Lr f_\tau)=	\dfrac{1}{\zeta\alpha\tau}\Div_r\Big(\alpha rf_\tau+\nabla_rf_\tau+\dfrac{1}{2} A_b(r)\nabla_r f_\tau\Big)\\
&f_\tau|_{t=0}=f_0^\tau, \quad (x,r)\in \mathbb{T}^d\times \mathbb{R}^d, d=2,3,
\end{cases}\end{align}
where $\alpha>0$ and related to the intensity of turbulence, see \autoref{subsection-time-scale} how to get \eqref{sys-I} from \eqref{Intro3} and \eqref{constant-meaning} for the interpretation of $\alpha$.\\

The aim of this contribution is to consider the singular limit as $\tau\to 0$ of \eqref{sys-I}. This completes the study done in  \cite{FlaTah24}, and rigorously justifies the physical predictions about the  power tail of the probability density function  of polymer end-to-end vector $R$  (see \textit{e.g.} \cite{Balk}). In \cite{FlaTah24},  we proved a first scaling limit (in space) by using a stochastic modeling of small-scale turbulent flow. In other words, we proved  $\lim_{N\to +\infty} f^{N,\tau}=f^{\tau}$ exists  and $f^{\tau}$ satisfies  \eqref{Intro3}. Here, we prove that $\lim_{\tau\to 0} f^{\tau}=\lim_{\tau\to 0}\lim_{N\to +\infty} f^{N,\tau}$ exists in appropriate sense and we identify the limit. To the best of author's knowledge, this is the first time that  scaling and singular limits\footnote{We use \textit{scaling limit} to denote the limit as $N\to +\infty$, which corresponds to a spatial scaling limit. On the other hand, the term \textit{singular limit} denotes the limit $\tau\to 0,$ which corresponds to the temporal singular parameter $\tau$. We use the word "singular"  because the limit problem (equilibrium) does not have the same nature as the initial one \eqref{sys-I} (kinetic).}, in space and time, have been considered using stochastic modeling of small-scales of turbulent flow in the presence of stretching and transport terms.
The main novelty of our contribution lies in the rigorous analysis based on analytical tools derived from the stochastic diffusion limit, homogenization theory and kinetic equations.

\subsubsection*{Sketch of the main results}
For the convenience of the reader, let us explain heuristically the main result of this work then precise results (see \autoref{MAin-thm-1} and \autoref{Main-thm-2}) and rigorous proofs are provided in the main body of the paper. We distinguish three cases: 
\begin{itemize}
    \item The case $\alpha>\frac{d}{2}$: there exists $f_0\in L^2(\mathbb{T}^d\times \mathbb{R}^d;(1+\frac{\vert r\vert^2}{2})^{\alpha}drdx)$  such that $$ \displaystyle\lim_{\tau\to 0}\int_{\mathbb{T}^d\times \mathbb{R}^d} \vert f_0^\tau-f_0\vert^2(1+\frac{\vert r\vert^2}{2})^{\alpha}drdx =0,$$ then
       $f_\tau \text{ converges to } \rho\otimes\dfrac{1}{Z} (1+\frac{\vert r\vert^2}{2})^{-\alpha} \text{ in }  L^2\big(0,T;L^2(\mathbb{T}^d\times \mathbb{R}^d;(1+\frac{\vert r\vert^2}{2})^{\alpha}drdx)\big),$
    where   $Z=\int_{\mathbb{R}^d}(1+\frac{\vert r\vert^2}{2})^{-\alpha}dr$ and 
  $\rho$ is the unique $L^2$-valued solution of
    \begin{align}
\begin{cases}&\partial_t \rho+u_L\cdot \nabla_x \rho=0\\
&\rho(x)|_{t=0}=\rho_0=\int_{\mathbb{R}^d}f_0(x,r)dr.
\end{cases}\end{align}
 
 \item The case $0<\alpha \leq \frac{d}{2}$ and 
    $(\int_{\mathbb{T}^d\times \mathbb{R}^d} \vert f_0^\tau\vert^2(1+\frac{\vert r\vert^2}{2})^{\alpha}drdx)_\tau$ is bounded:  there exists a subsequence of $(f_\tau)_\tau$ denoted by  $(f_{\tau_k})_{\tau_k}$ such that
    \begin{align*}
       f_{\tau_k} \text{ converges weakly-* to } 0 \text{ in }  L^\infty\big(0,T;L^2(\mathbb{T}^d\times \mathbb{R}^d;(1+\frac{\vert r\vert^2}{2})^{\alpha}drdx)\big).
    \end{align*}
            \item  The case $0<\alpha \leq \frac{d}{2}$ and  $(\int_{\mathbb{T}^d\times \mathbb{R}^d} \vert f_0^\tau\vert drdx)_\tau$ is bounded:  there exists a subsequence of $(f_\tau)_\tau$ denoted by  $(f_{\tau_k})_{\tau_k}$ such that
      $f_{\tau_k} \text{ converges weakly-* to } \nu \text{ in }  L^\infty\big(0,T;\mathcal{M}(\mathbb{T}^d\times \mathbb{R}^d)\big),$
      where $\mathcal{M}$ denotes the space of Radon measures and $\nu$ satisfies in weak sense the following
        \begin{align}\label{case3}
            \Div_r\big(\alpha r\nu+\nabla_r \nu+\dfrac{1}{2}A_b(r)\nabla_r \nu\big)=0.
        \end{align}  
        The solution   to \eqref{case3} is  $\nu=g\otimes \mu_\alpha,$ where $g$ is an element of $L^\infty(0,T;\mathcal{M}(\mathbb{T}^d))$ and $\mu_\alpha \in \mathcal{M}(\mathbb{R}^d)$  defined as follows
\begin{align*}
     \langle  \mu_\alpha,\varphi\rangle_{\mathcal{M},C_c}:=\int_{\mathbb{R}^d}  (1+\frac{\vert r\vert^2}{2})^{-\alpha}\varphi(r)dr, \quad \forall \varphi\in C_c(\mathbb{R}^d).
\end{align*}
\end{itemize}
We refer to \autoref{Sec-main-results} for  a discussion about the physical interpretation of the above results.\\

Stochastic diffusion limit results are related to homogenization theory in very broad sense. Regarding this latter and without trying to be exhaustive, we refer   to   \cite{Bensoussan1978asymptotic} for deterministic  techniques, \cite{Pavliotis2008multiscale} for PDE homogenization and  stochastic modeling (see also \cite{MK,Jikov2012homogenization}). However, it is important to notice that the strategy in stochastic diffusion limits are different from classical homogenization and especially they are very efficient for generalizations to nonlinear problems.\\

The origin of stochastic diffusion limits, based on the It\^o-Stratonovich corrector, is the paper \cite{Galeati} that threw a new light on passive scalars subject to turbulent models. The result of \cite{Galeati} has been generalized in several directions, showing in particular its strength for
nonlinear scalar problems (see \cite{FlaElibook} for a review). 
In particular, we refer to \cite{Flandoli2021scaling,Flandoli2024quantitative} for the results in this direction about  the  2D Euler  and  2D Navier-Stokes equations.   Moreover, this approach leads to   a  delayed  blow-up   and   dissipation enhancement results, see \textit{e.g.} \cite{Flandoli2021delayed,luo2025enhanced}. We refer  also to \cite[Section 1.2]{FlaTah24} for extended literature review on the stochastic diffusion limits and transport noise. It is worth underlying that the mentioned results concern scalar problems with transport noise. However,   when the turbulent fluid acts on
vector fields,  in addition to   transport there is    
\textit{stretching}: namely the differential $D\phi_{t}$ of the Lagrangian flow $\phi_{t}$ of a fluid   produces a modification of the length of vectors $v$,  which leads to  an increase of length for many possible directions $v$. The result of this work can be classified in this direction. \\

More difficult has been adapting the ideas to advected vector fields, because of lack of control on the stochastic stretching, when the diffusive scaling limit is performed. Nevertheless, positive results for particular models have been obtained, adapting the classical scaling of \cite{Galeati}, see  \textit{e.g.} \cite{ButFlaLuo,FL1,FlandoliLuo2024,Papini}. Recently, after considering a new scaling such as   the noise covariance goes to zero but a suitable covariance built on derivatives of the noise converges to a non zero limit, new results have been obtained in \cite{BFLT2024,FlaTah24} in the case of   stochastic transport-stretching. In \cite{BFLT2024}, we  introduced a background stochastic Vlasov equation behind stochastic transport and advection equations which gives additional information on the fluctuations and oscillations of solutions  based on Young measures. We first developed the theory for the stochastic transport of a passive scalar.  Then we developed the theory for a passive vector field, where stretching also acts in addition to  transport. In the case of a passive vector field the background Vlasov equation adds completely new statistical information to the stochastic advection equation and  the theory developed may help to recognize the existence of large values of the length of the vectors, where   the physical phenomenon of magnetic dynamo is an example. In \cite{FlaTah24}, we considered the new scaling by considering the Fokker-Planck equation \eqref{Intro2} of polymers density and obtained the result described in the beginning of this introduction. Namely, \eqref{Intro3} as a limit of \eqref{Intro2}.\\

On the other hand, kinetic equations in general, which correspond to the mesoscopic scale,  appear in the study of many problems in physics, biology, and other fields.  Different types of these equations have been studied extensively, both in terms of the existence of solutions and their properties, as well as the asymptotic behavior with respect to the parameters appearing in the equations. In general, the main interest when studying asymptotic is to derive equations of macroscopic quantities such as density and probability distribution with respect to mesoscopic variable such as particle velocity and the polymers end-to-end vector in our case. Without trying to be exhaustive,  the author studied  in \cite{Poupaud1992} the  asymptotic of electron distribution function in semiconductor kinetic theory where the equation is  given by Boltzmann transport equation for high electric fields. He proved that the limit  distribution function is given by the tensor product of the density (satisfying a  linear transport equation) and a probability distribution function of the velocity variable  satisfying an  appropriate PDE. In plasma, people are interested in studying the  asymptotic of the  Vlasov-Poisson-Fokker-Planck System for example.  There are two  important scalings, the low field limit and the high field limit.  This leads to different scalings of the kinetic equations and requires a separate analysis, we refer \textit{e.g.} to \cite{Masmoudi-Vlasov,Nieto2001high}. In \cite{Masmoudi-Vlasov,Poupaud2000}, the limit  distribution function is given by the tensor product of the density (satisfying a  drift-diffusion equation) and a normalized Maxwellian with zero mean  for  the microscopic  velocity. In \cite{fetecau2015}, the authors  investigate a zero inertia limit of some kinetic equation  where the limit is given by the tensor product of the density (satisfying a   non linear transport equation) and   monokinetic distribution  of the velocity variable. We refer \textit{e.g.} to \cite{Ghattassi,goudon2005,Herda2018,jabin2000,Nieto2001high} for more results about  the asymptotics of the kinetic equations.\\

Finally, it is worth mentioning that due to the stretching term in \eqref{Intro1} and the appropriate stochastic scaling in \cite{FlaTah24}, the matrix $A_b$ is generated after the first scaling limit. This matrix $A_b$ is a fundamental key giving rise to the Cauchy distribution for the end-to-end vector.

\subsubsection*{Structure of the paper} The manuscript is organized as follows: in \autoref{section-prel}, we present some preliminaries and  formulate the problem. Then, we  collect the main results   and the physical interpretation of  our work in \autoref{Sec-main-results}. \autoref{section-estimate} is devoted to the proof of a uniform estimate with respect to $\tau$ and some results about the  continuity equation associated with \eqref{sys-I}.  In \autoref{section-cv}, we prove the convergence results.  We conclude the paper with  \autoref{Appendix} containing a proof of \autoref{lemma-appendix}.

\section{Preliminaries and formulation of the  problem}\label{section-prel}
The aim of this section is to recall the main ingredients of the rigorous derivation of \eqref{sys-I} as the scaling limit of \eqref{Intro2}. We
consider a dilute  family of polymers  described by the kinetic equation \eqref{Intro2} for the density $f^{N,\tau}.$
In order to make the manuscript  self-contained and  for the convenience of the reader, let us recall the structure of $(\sigma_{k}^{N,\tau})_k$. For more details we refer to  \cite{FlaTah24}.
	\subsection{Structure of the small scales}\label{assumption_noise-2D}
    In order to clarify the stochastic modeling of small scales, $\tau$ denotes the  dominant time-scale  of 
     $u_s$   and $\ell\sim N^{-1}$ its space scale. 
    \subsubsection*{The 2D case}
Consider $\mathbb{Z}_{0}^{2}:=\mathbb{Z}^2-\{(0,0)\}$ divided into its four quadrants (write
$k=\left(  k_{1},k_{2}\right)  $)%
\begin{align*}
K_{++}  & =\left\{  k\in\mathbb{Z}_{0}^{2}:k_{1}\geq0,k_{2}>0\right\};\quad
K_{-+}   =\left\{  k\in\mathbb{Z}_{0}^{2}:k_{1}<0,k_{2}\geq0\right\}  \\
K_{--}  & =\left\{  k\in\mathbb{Z}_{0}^{2}:k_{1}\leq0,k_{2}<0\right\} ; \quad
K_{+-}   =\left\{  k\in\mathbb{Z}_{0}^{2}:k_{1}>0,k_{2}\leq0\right\}
\end{align*}
and set%
\begin{align*}
K_{+}  & =K_{++}\cup K_{+-};  \quad 
K_{-}   =K_{-+}\cup K_{--}  \text{  and }   K=K_{+}\cup K_{-}.
\end{align*}
Let $(\Omega,\mathcal{F},(\mathcal{F}_t)_t,P)$  be  a   complete   filtered probability space.  
    Define \footnote{For $y=(y_1,y_2)\in   \R^2,$  $y^\perp$   stands  for $(-y_2,y_1).$}
\begin{align*}
\sigma_{k}^{N,\tau}\left(  x\right)     =\theta_{k}^{N,\tau}\frac{k^{\perp}}{\left\vert
k\right\vert }\cos k\cdot x,\qquad k\in K_{+},  \quad 
\sigma_{k}^{N,\tau}\left(  x\right)     =\theta_{k}^{N,\tau}\frac{k^{\perp}}{\left\vert
k\right\vert }\sin k\cdot x,\qquad k\in K_{-}%
\end{align*}
where
\begin{align}\label{scale-space}
\theta^{N,\tau}_{k}&=\dfrac{a_\tau}{\left\vert k\right\vert^2 },  \qquad   N \leq \left\vert k\right\vert \leq  2N,\qquad N\in \mathbb{N}^* ; \quad
\theta^{N,\tau}_{k}=0 \qquad \text{elsewhere.}
\end{align}
where $a_\tau$ is a positive constant measuring the intensity, see \autoref{subsection-time-scale}. Let us also consider a family $(W_t^k)_t^{k\in    \mathbb{Z}^2_0}$ of independent Brownian motions on the probability space $(\Omega,\mathcal{F},P)$.  Then, set $u_s(x,t)=\sum_{k\in K}%
\sigma_{k}^{N,\tau}\left(  x\right)  \partial_{t}W_{t}^{k}.$

\subsubsection*{The 3D case}
We  introduce the partition $\mathbb{Z}^3_0=\Gamma_{3,+}\cup \Gamma_{3,-}$\footnote{$\mathbb{Z}^3_0=\mathbb{Z}^3-\{(0,0,0)\}$} such that $\Gamma_{3,+}=-\Gamma_{3,-} $
and  we consider a family  of  real valued independent Brownian motions $(B_t^{k,j})_t^{k\in    \mathbb{Z}^3_0}, j\in \{1,2\}$   defined  on the complete   filtered probability space    $(\Omega,\mathcal{F},(\mathcal{F}_t)_t,P)$, that is 
$ \mathbb{E}(B_t^{k,j}B_s^{l,m})= \min(t,s)\delta_{k,l}\delta_{j,m}.$
Then,   we introduce a sequence of complex-valued Brownian motions adapted to $(\mathcal{F}_t)_t$  defined as follows
\begin{align*}
    W^{k,j}_t&=\begin{cases}
        B^{k,j}_t+iB^{-k,j}_t & \textit{if } k\in \Gamma_{3,+} \\
   B^{-k,j}_t-iB^{k,j}_t & \textit{if } k\in \Gamma_{3,-}.
    \end{cases}
    \end{align*}
  Let $N\in \mathbb{N}^*$,    define
  $
   \theta^{N,\tau}_{k,j} =\dfrac{a_\tau}{\lvert k\rvert^{5/2}}\mathbf{1}_{\{N\leq \lvert k\rvert\leq 2N\} },\  
$ for   a positive constant $a$.
Then, for each $ k\in \mathbb{Z}^{3}_0,\ j\in\{1,2\}$ we denote by $\sigma_{k,j}^{N,\tau}(x)=\theta_{k,j}^{N,\tau}  a_{k,j}e^{ik\cdot x}$, where $\{\frac{k}{\lvert k\rvert}, a_{k,1}, a_{k,2}\}$ is an orthonormal system of $\R^3$ for $k\in \Gamma_{3,+}$ and $a_{k,j}=a_{-k,j}$ if $k\in \Gamma_{3,-}$. Set
\begin{align}\label{noise-3D}
  \mathbf{W}^{N,\tau}(t,x)  =\sum_{{\substack{ k\in \mathbb{Z}^3_0, j\in \{1,2 \} }}}\theta_{k,j}^{N,\tau}  a_{k,j}e^{ik\cdot x}W^{k,j}(t),
\end{align}
 and consider 
$u_s= \partial_t \mathbf{W}^{N,\tau}$ in the 3D case. \\

Seeking not to overload the notation, we simply recall the result of the scaling limit and refer to the main theorems in \cite{FlaTah24} about the precise mathematical framework of  well-posedness and the various estimates necessary for the analysis. Thus, the following subsection concerns the rigorous derivation of \eqref{Intro3} as a scaling limit of \eqref{Intro2}.
\subsection{Scaling limit results} Let us present the following scaling limit result based on \cite{FlaTah24}.
 \subsubsection{Notation    and functional  setting }
 Let $d\in \{2,3\},$
we      consider the periodic boundary conditions with respect to  the spacial variable    $x$,    namely  $x$ belongs to  the  d-dimensional torus $\T^d= (\mathbb{R}/2\pi\mathbb{Z})^d.$  On  the other   hand,   the end-to-end  vector  variable  $r$   belongs to  $\R^d.$   Let	$m\in	\mathbb{N}^*$  and  introduce the following Lebesgue and Sobolev spaces with polynomial weight, namely
\begin{align*}
&L^2_{r,m}(\mathbb{T}^d\times\mathbb{R}^d):=\{f:\mathbb{T}^d\times\mathbb{R}^d\to \mathbb{R}:  \int_{\mathbb{T}^d\times\mathbb{R}^d}\hspace{-0.7cm}\vert f(x,r)\vert^2(1+\vert r \vert^2)^{\frac{m}{2}}dxdr:= \Vert f\Vert_{L^2_{r,m}}^2 < \infty \}.\end{align*}	
Note that $L^2_{r,0}(\mathbb{T}^d\times\mathbb{R}^d)$ coincide with the  Lebesgue space $L^2(\mathbb{T}^d\times\mathbb{R}^d)$.
We	will	use	the	following	notation	$L^2_{r,2}(\mathbb{T}^d\times\mathbb{R}^d):=H.$	
We	recall	the definition of	 inner products	defined	on	the	spaces $H$ and $L^2(\mathbb{T}^d\times\mathbb{R}^d)$.
\begin{align*}
(h,g)_H&:= \int_{\mathbb{T}^d\times\mathbb{R}^d} h(x,r) g(x,r)(1+\vert r \vert^2)dxdr,\quad	\forall	g,h \in H;\\
(h,g)&:= \int_{\mathbb{T}^d\times\mathbb{R}^d} h(x,r) g(x,r)dxdr,\quad	\forall	g,h \in L^2(\mathbb{T}^d\times\mathbb{R}^d).
\end{align*}

\subsubsection{Scaling limit as $N\to +\infty$}
Let $T>0$ and $u_L=u_L(x,t)$ be a given  large scale component  such that
 $u_L\in C\big([0,T], C^2(\mathbb{T}^d;\mathbb{R}^d)\big)$ and $\Div_x (u_L)=0.$\\
 
 Since	$L^{\infty}(0,T;H)$	is	not	separable,	
it is	convenient	to	recall	the	following	space:
$$L^2_{w-*}(\Omega;L^\infty(0,T;H))=\{ 	u:\Omega\to	L^\infty(0,T;H)	\text{	is	 weakly-* measurable		and	}	\E\Vert	u\Vert_{L^\infty(0,T;H)}^2<\infty\},$$
where	weakly-* measurable	stands	for	the	measurability	when	$L^\infty(0,T;H)$	is	endowed	with	the	$\sigma$-algebra	generated	by	the	Borel	sets	of	weak-*	topology, we recall that (see 	\cite[Thm. 8.20.3]{Edwards}) $$ L^2_{w-*}(\Omega ;L^\infty([0, T];H))\simeq \left(L^2(\Omega ;L^1([0, T];H^\prime)) \right)^\prime .$$	

\par Assume that $f_0^\tau \in H$ and $f^{N,\tau}$ be the unique  quasi-regular weak solution  of \eqref{Intro2} (see   \cite[Def. 5, Thm. 7 and Thm. 8]{FlaTah24}).  Let us recall the main ideas to get \eqref{Intro3} from \eqref{Intro2}. \\

By using  the stochastic modeling of small-scale, we get  a stochastic Fokker-Planck equation in Stratonovich form:
\begin{align}\label{Stra-FP-second-limit}
	\begin{cases}&\partial_tf^{N,\tau}+\Div_x(u_L^{N,\tau}f^{N,\tau})+\Div_r((\nabla u_L^{N,\tau}r-\dfrac{1}{\beta}r)f^{N,\tau})=	\sigma^2\Delta_rf^{N,\tau}\\	&\quad -\sum_{k\in K}\sigma_k^N.\nabla_xf^{N,\tau}\circ\partial_tW^k -\sum_{k\in K}(\nabla \sigma_k^Nr).\nabla_rf^{N,\tau}\circ \partial_tW^k\\	
		&f^{N,\tau}|_{t=0}=f_0^\tau.
\end{cases}\end{align}
We write \eqref{Stra-FP-second-limit} in  It\^o form, where we need to compute  It\^o-Stratonovich corrector, see \cite[Section 3]{FlaTah24}. By using the space-homogeneity and mirror symmetry of the covariance  operator $Q(x,y):=\sum_{k\in K}\sigma_k^N(x)\otimes\sigma_k^N(y),$ we get 
the following  It\^o form  of the  stochastic   Fokker Planck equation    	
	\begin{align}\label{Ito-FP}
		\begin{cases}&df^{N,\tau}+\Div_x(u_L^{N,\tau}f^{N,\tau})dt+\Div_r((\nabla u_L^{N,\tau}r-\dfrac{1}{\beta}r)f^{N,\tau}) dt\\	&=	\sigma^2\Delta_rf^{N,\tau}dt-\sum_{k\in K}\sigma_k^{N,\tau}.\nabla_xf^{N,\tau}dW^k -\sum_{k\in K}(\nabla \sigma_k^{N,\tau}r).\nabla_rf^{N,\tau} dW^k\\
			&\quad +\alpha_N\Delta_x f^{N,\tau}dt+ \dfrac{1}{2}\Div_r(\sum_{k\in K}\left((\nabla \sigma_k^{N,\tau}r) \otimes (\nabla \sigma_k^{N,\tau}r)\right) \nabla_rf^{N,\tau})dt\\
			&f^{N,\tau}|_{t=0}=f_0^\tau.
	\end{cases}\end{align}
	where 
	$$A_N(r)=\sum_{k\in K}\left((\nabla \sigma_k^Nr) \otimes (\nabla \sigma_k^Nr)\right) = k_TA_b(r)+O(\dfrac{1}{N})P(r), \quad P \text{ is a polynomial of second degree. }$$
	By using the scaling \eqref{scale-space}, we see that  $A_N(r) \to k_TA_b(r) $ and  $\alpha_N \to 0$.  Thus, we can interpret the emergence of the part including $A_b(r)$ as a consequence of the stretching effect  and the choice of scaling \eqref{scale-space}. More precisely, we get (see  \cite[Thm. 9 and Subsection 7.4.]{FlaTah24}) the following result.


\begin{theorem}\label{Thm-scaling-N}	 Under the assumption that the small scales have the given structure in \autoref{assumption_noise-2D}. There  exists  a   new probability space,	denoted by  the same    way  (for simplicity)    $(\Omega,\mathcal{F},P)$,  $f_\tau\in L^2_{w-*}(\Omega ;L^\infty([0, T];H))),  \nabla_rf_\tau\in L^2(\Omega ;L^2([0, T];H)))   $    such    that  the following   convergence holds   (up to  a   sub-sequence)\begin{align*}
    f^{N,\tau} &\rightharpoonup f_\tau \text{		in	}  L^2_{w-*}(\Omega ;L^\infty([0, T];H))),\quad
         \nabla_rf^{N,\tau} \rightharpoonup \nabla_r f_\tau \text{		in	}  L^2(\Omega ;L^2([0, T];H)).\end{align*} Moreover   $f_\tau$ is   the unique  solution    of  the following   problem:    	$P$-a.s. for any $t\in [0,T]$: \begin{align}&\int_{\mathbb{T}^d}\int_{\mathbb{R}^d}f_\tau(x,r,t)\phi(x)\psi(r) drdx-\int_{\mathbb{T}^d}\int_{\mathbb{R}^d}f_0^\tau(x,r)\phi(x)\psi(r) drdx\notag\\&= \int_0^t\int_{\mathbb{T}^d}\int_{\mathbb{R}^d}f_\tau(x,r,s)\left(u_L(x,s)\cdot \nabla_x \phi(x)\psi(r)+ (\nabla u_L(s,x)r-\dfrac{1}{\beta}r) \cdot\nabla_r\psi(r)\phi(x)\right) dr dxds\notag\\	&-	\int_0^t\int_{\mathbb{T}^d}\int_{\mathbb{R}^d}\sigma^2\nabla_rf_\tau(x,r,s) \cdot\nabla_r\psi(r)\phi(x)+\dfrac{k_T}{2}A_b(r) \nabla_r f_\tau(x,r,s)\cdot\nabla_r\psi(r)\phi(x) dr dxds,\notag\end{align}
for any $\phi  \in C^\infty(\mathbb{T}^d)$ and  $\psi  \in C^\infty_c(\mathbb{R}^d),$      $A_b(r)= ((b+1)\left\vert r\right\vert^2I-br\otimes r)$ ($b=2$ in 2D and $b=1$ in 3D), 
$k_T=C_da_\tau^2$ ($C_d=\frac{\pi \log(2)}{8}$ in 2D and  $C_d=\frac{8\pi \log(2)}{15}$ in 3D) and $a_\tau$ is an intensity parameter of the noise. \end{theorem} 

 \subsection{The derivation of  \eqref{sys-I}}\label{subsection-time-scale}
       We recall that $\sigma^2$ is the  product $\frac{kT}{\beta}$, $k$ being Boltzmann constant and  $T$ being the temperature. Hence
 $   \sigma^2=\frac{C_2}{\beta},$ where $ C_2=kT>0.$
In order to formulate the main equation, let us present the following.
\subsubsection{Noise specification and the parameter $a_\tau$}
We focus on the 2D case and 3D  case follows in the same way.
Denote by $u_s$ the  small scale of the turbulent velocity, its  vortex structures  have a dominant time-scale
$\tau$ and space scale $\ell$, and the associated turbulent kinetic energy
$k_{T}$ (of order $\frac{\ell^{2}}{\tau^{2}}$, in other words, $k_T=\kappa\frac{\ell^{2}}{\tau^{2}}$ for a positive constant $\kappa$). In a Gaussian approximation of
the model, a reasonable choice is %
\[
u_s\left(  x,t\right)  :=\frac{\sqrt{k_{T}}}{\ell}\sum_{N\leq\left\vert
k\right\vert \leq2N}\frac{1}{\left\vert k\right\vert ^{2}}\frac{k^{\perp}%
}{\left\vert k\right\vert }e_{k}\left(  x\right)  Z_{t}^{k},%
\]
where the functions $e_{k}\left(  x\right)  $, sine and cosine, $N\sim\ell^{-1}$, $Z_{t}^{k}$ are stationary
Ornstein-Uhlenbeck processes solution of
\[
dZ_{t}^{k}=-\frac{1}{\tau}Z_{t}^{k}dt+\sqrt{\frac{2}{\tau}}dW_{t}^{k},%
\]
where $(W_{t}^{k})_k$ are independent Brownian motions. The average kinetic energy is given by
\[
\mathbb{E}\left[  \left\vert u_s\left(  x,t\right)  \right\vert ^{2}\right]
=\frac{k_{T}}{\ell^{2}}\sum_{N\leq\left\vert k\right\vert \leq2N}\frac
{1}{\left\vert k\right\vert ^{4}}e_{k}^{2}\left(  x\right),
\]
and $\frac{1}{\ell^{2}}\sum_{N\leq\left\vert k\right\vert \leq2N}\frac
{1}{\left\vert k\right\vert ^{4}}e_{k}^{2}\left(  x\right)$
has a unitary variance  size. The intensity of the turbulent flow is given by $\sqrt{k_T}$.
Therefore, after using   the
approximation%
\[
\lim_{\tau\rightarrow0}\mathbb{E}\left[  \left(  \frac{1}{\sqrt{\tau}}\int%
_{0}^{t}Z_{s}^{k}ds-W_{t}^{k}\right)  ^{2}\right]  =0,
\]
we write
\[
u_s\left(  x,t\right)  dt:=\sqrt{\tau k_{T}}\sum_{N\leq\left\vert
k\right\vert \leq2N} \frac{1}{\ell}\frac{1}{\left\vert k\right\vert ^{2}}\frac{k^{\perp}%
}{\left\vert k\right\vert }e_{k}\left(  x\right)  dW_{t}^{k}.
\]
Thus, $a_\tau$  is of the order  $\dfrac{\ell}{\sqrt{\tau}}C(\ell)$ with $C(\ell)=\dfrac{1}{\ell}.$ More precisly,  $a_\tau=\sqrt{\kappa}\dfrac{\ell}{\sqrt{\tau}}C(\ell)$.  Thus we infer that $a_\tau=\dfrac{C_3}{\sqrt{\tau}}$ with $C_3=\sqrt{\kappa}>0$.
Let us formulate the main equation. It is clear 
 that we can write the limit equation given by \autoref{Thm-scaling-N}  as follows
	\begin{align}\label{Limit-FP-eqn}
	\begin{cases}&\partial_t f_\tau+\Div_x(u_Lf_\tau)+\Div_r(\nabla u_Lr f_\tau))=	
 \Div_r(\dfrac{1}{\beta}r f_\tau)+\sigma^2\Delta_r f_\tau + 
 \dfrac{C_da_\tau^2}{2}\Div_r(A_b(r) \nabla_r f_\tau)        \\	
		&f_\tau|_{t=0}=f_0^\tau.
			\end{cases}\end{align}
In the regime where the relaxation time of polymers and the dominant time-scale of the small scale turbulent flow satisfies
$\beta=\zeta\tau, \zeta >0$, we obtain  
\begin{align}\label{Limit-FP-eqn-2}
\begin{cases}&\partial_tf_\tau+\Div_x(u_Lf_\tau)+\Div_r((\nabla u_Lr)f_\tau)=	 \Div_r\Big(\dfrac{1}{\zeta\tau}rf_\tau+\dfrac{ C_2}{\zeta\tau}\nabla_rf_\tau+\dfrac{C_dC_3^2}{2\tau}A_b(r)\nabla_r f_\tau\Big)\\
&f_\tau|_{t=0}=f_0^\tau.
\end{cases}\end{align}
By setting  
\begin{align}\label{constant-meaning}
    \alpha=\dfrac{1}{\zeta C_dC_3^2}>0 \text{ and  }  \gamma=\dfrac{C_2}{\zeta C_dC_3^2}>0,
\end{align}
it is worth making some comments on the parameter $\alpha.$ Notice that form \eqref{constant-meaning}, $\alpha$  become smaller    either when $\zeta$ become larger (\textit{i.e.} the relaxation time of polymer $\beta$ becomes larger than the dominant time-scale $\tau$)  or $C_3$ becomes larger (\textit{i.e.} the intensity of the turbulent flow becomes larger). In both cases, this corresponds to a stronger turbulent flow for the polymers. 
\\

 Therefore, we obtain the following equation for $f_\tau=f_\tau(x,r,t)$:
\begin{align*}
\begin{cases}\partial_tf_\tau&+\Div_x(u_Lf_\tau)+\Div_r(\nabla u_Lrf_\tau)=	\dfrac{1}{\zeta\alpha\tau}\Div_r\Big(\alpha rf_\tau+\gamma\nabla_rf_\tau+\dfrac{1}{2} A_b(r)\nabla_r f_\tau\Big)\\
f_\tau|_{t=0}&=f_0^\tau.
\end{cases}\end{align*}
Without loss of generality set $\gamma=1$. Our aim is  investigate the limit behavior  of $(f_\tau)_\tau$ as $\tau\to 0$ in \eqref{Limit-FP-eqn-2}.
\begin{align}\label{Limit-FP-eqn-2}
\begin{cases}\partial_tf_\tau&+\Div_x(u_Lf_\tau)+\Div_r(\nabla u_Lrf_\tau)=	\dfrac{1}{\zeta\alpha\tau}\Div_r\Big(\alpha rf_\tau+\nabla_rf_\tau+\dfrac{1}{2} A_b(r)\nabla_r f_\tau\Big)\\
f_\tau|_{t=0}&=f_0^\tau.
\end{cases}\end{align}

For clarity of the presentation, it worth  summarizing the various key parameters (see \autoref{table1}) of the manuscript and those related to it from \cite{FlaTah24}. First, we recall that   $C_3>0$  is a dimensionless constant measuring the intensity of the turbulent flow, $C_d=\frac{\pi \log(2)}{8}$ in 2D and  $C_d=\frac{8\pi \log(2)}{15}$ in 3D occurring after taking the scaling limit and depending on the choice of  coefficients   $(\sigma_{k}^{N,\tau})_N$.\\

\begin{table}[h!]
\centering
\begin{tabular}{ |p{8.6cm}||p{3cm}|p{2.5cm}|  }
 \hline
  Parameter & Scaling limit \cite{FlaTah24} ($N\to +\infty$) & Singular limit ($\tau \to 0)$\\
 \hline
 Boltzmann constant $\times$  Temperature  & $kT$    & $kT$\\
 Polymer relaxation time & $\beta$ & $\beta$\\
 Dominant space-scale of small scale   turbulence 
 & $\ell \sim N^{-1}$  & $\ell \sim N^{-1}$ \\
 Dominant time-scale of small scale  turbulence
&   /  & $\tau$\\
 Intensity parameter (turbulence)   & $a$ & $a_\tau=\dfrac{C_3}{\sqrt{\tau}}$  \\
   Constant $\zeta >0$  &   /  & $\beta=\zeta\tau$ \\
  Exponent $\alpha$  (see also \eqref{constant-meaning})&  / & $\alpha=\dfrac{1}{\zeta C_dC_3^2}$\\
 \hline
\end{tabular}
\caption{Key parameters.}
\label{table1}
\end{table}

\section{Main results}\label{Sec-main-results}
Let us introduce the following space
\begin{align*}
      H_\alpha&=\{ g:\mathbb{T}^d\times\mathbb{R}^d\to \mathbb{R}: \quad \int_{\mathbb{T}^d}\int_{\mathbb{R}^d}g^2(x,r)(1+\frac{\vert r\vert^2}{2})^\alpha drdx<\infty\},\quad \alpha > 0.
\end{align*}
\begin{remark}\label{rmq-less1}
   Notice that  $H_\alpha\hookrightarrow H\cap L^1(\mathbb{T}^d\times \mathbb{R}^d)$ if $\alpha>\frac{d}{2}$.
\end{remark}
In the following, we always consider  $f_0^\tau\in H$ for any $\tau >0.$
\subsection{Statement of the main result}
Consider the following assumptions.
\begin{itemize}
    \item[(H$_1$)] There exists $\mathbf{\Lambda}>0$ independent of $\tau$  such that $\displaystyle\sup_{\tau>0}\Vert f^\tau_0 \Vert_{H_\alpha}^2\leq  \mathbf{\Lambda}. $
    \item[(H$_2$)] There exists $f_0\in H_\alpha$ such that      $       \displaystyle\lim_{\tau\to 0}\Vert f_0^\tau-f_0\Vert_{H_\alpha} =0.  $
\end{itemize}
The first result concerns the case $\alpha > \frac{d}{2}.$
     \begin{theorem}\label{MAin-thm-1} Let $ \tau>0$ and $\alpha > \frac{d}{2}, $  assume that  assumption $H_1$   is satisfied. Let $f_\tau$ be the  unique quasi-regular weak solution
 to \eqref{equ-final-I}, then 
 \begin{align}
    \sup_{t\in [0, T]} \Vert f_\tau(t)\Vert_{H_\alpha}^2
    &+\dfrac{1}{\zeta\alpha\tau}\int_0^T\int_{\mathbb{T}^d}\int_{\mathbb{R}^d}\vert \nabla_r \big(f_\tau(s)(1+\frac{\vert r\vert^2}{2})^{\alpha}\big)\vert^2(1+\frac{\vert r\vert^2}{2})^{-\alpha+1} dr dxds\leq  e^{\mathbf{K}T}\mathbf{\Lambda}, 
\end{align}
 where $\mathbf{K}=2\alpha\Vert \nabla_xu_L \Vert_\infty$.  If  moreover  $H_2$  is satisfied,   the following convergence holds as $\tau \to 0$
\begin{align}\label{cv-L2-equilibruim}
   f_\tau \to_\tau \rho \otimes\dfrac{1}{Z}(1+\frac{\vert r\vert^2}{2})^{-\alpha} \text{ in } L^2(0,T;H_\alpha),
\end{align}
 where $\rho$ is the unique solution to \eqref{continuity-eqn-limit} and
 $Z:=\int_{\mathbb{R}^d}(1+\frac{\vert r\vert^2}{2})^{-\alpha}dr=C(d,\alpha)$ is a normalizing factor. In addition, the following rate of convergence holds
\begin{align*}
    &  \int_0^T\hspace{-0.2cm}\int_{\mathbb{T}^d}\hspace{-0.1cm}\int_{\mathbb{R}^d}(f_\tau-\rho \otimes p_\alpha)^2(s)(1+\frac{\vert r\vert^2}{2})^\alpha dr dxds
\leq \mathbf{C}_\alpha(\tau+\dfrac{T}{Z}   \Vert f_0^\tau-f_0\Vert_{H_\alpha}^2),\end{align*}
where   $\mathbf{C}_\alpha$ is a positive constant.

      \end{theorem}
      \begin{proof}
  It is a consequence of    \autoref{propo-estimates},  \autoref{prop-cv-1} and \autoref{Lemma-uniq-cont}.
      \end{proof}
     
          Note that  \eqref{cv-L2-equilibruim}, in \autoref{MAin-thm-1}, asserts convergence  in the space  $L^2(0,T;H_\alpha)$ with rate $O(\tau)$ if  $\Vert f_0^\tau-f_0\Vert_{H_\alpha}^2=O(\tau).$ This convergence is the consequence of the combination of the inequality \eqref{cv-inequality} and  the weighted Poincaré inequality \eqref{ineq-B-2}.  Additionally, from \eqref{cv-inequality} it follows
\begin{align}
    &\sup_{t\in [0, T]}\int_{\mathbb{T}^d}\int_{\mathbb{R}^d}\vert f_\tau(t)-\rho(t)\otimes p_\alpha\vert^2(1+\frac{\vert r\vert^2}{2})^\alpha  drdx\leq  \int_{\mathbb{T}^d}\int_{\mathbb{R}^d}\vert f_0^\tau-\rho_0\otimes p_\alpha\vert ^2(1+\frac{\vert r\vert^2}{2})^\alpha  drdx+\widetilde{\mathbf{C}}\sqrt{\tau},\notag
\end{align}
which  asserts convergence  in the space  $L^\infty(0,T;H_\alpha)$ with rate $O(\sqrt{\tau})$ if $\Vert f_0^\tau-\rho_0\otimes p_\alpha\Vert_{H_\alpha}^2=O(\sqrt{\tau}).$  However, it worth noticing that the convergence in $L^\infty(0,T;H_\alpha)$ requires a stronger assumption of convergence of the initial data, namely  it requires $f_0$ to have the form $\rho_0\otimes p_\alpha.$
      \begin{remark}
         It is worth mentioning that the same conclusion of \autoref{MAin-thm-1} holds if $H_2$ is replaced by  there exists $f_0\in L^2(\T^d;L^1(\R^d))$ such that      $       \displaystyle\lim_{\tau\to 0}\Vert f_0^\tau-f_0\Vert_{L^2_xL^1_r} =0.  $
      \end{remark}
     
 In the following, $\mathcal{M}(\mathbb{T}^d\times \mathbb{R}^d)$ denotes the space of Radon measures. 
      Concerning the case  $0<\alpha \leq  \frac{d}{2}, $  we have
 \begin{theorem}\label{Main-thm-2}
      Let $ \tau>0$ and $0<\alpha \leq  \frac{d}{2}.$ We  distinguish two cases:
      \begin{enumerate}
          \item[i)]   Assume that   $H_1$   is satisfied and  let $f_\tau$ be the  unique quasi-regular weak solution to \eqref{equ-final-I}. Then 
 \begin{align}
    \sup_{t\in [0, T]} \Vert f_\tau(t)\Vert_{H_\alpha}^2
    &+\dfrac{1}{\zeta\alpha\tau}\int_0^T\int_{\mathbb{T}^d}\int_{\mathbb{R}^d}\vert \nabla_r \big(f_\tau(s)(1+\frac{\vert r\vert^2}{2})^{\alpha}\big)\vert^2(1+\frac{\vert r\vert^2}{2})^{-\alpha+1} dr dxds\leq  e^{\mathbf{K}T}\mathbf{\Lambda}. 
\end{align}
Moreover,  there exists a subsequence $(f_{\eta_k })_k$ of $(f_\tau)_\tau$ such that the following convergence holds   
\begin{align}
    f_{\eta_k } \stackrel{\ast}{\rightharpoonup} 0 \text{ in } L^\infty(0,T;H_\alpha) \text{ as } k \to +\infty.
\end{align}
\item[ii)] Let  $f^\tau_0 \in L^1(\mathbb{T}^d\times \mathbb{R}^d) \cap H$ and $f_\tau$ be the  unique quasi-regular weak solution
 to \eqref{equ-final-I}. Assume the existence of   $\mathbf{\Lambda}_2>0$ independent of $\tau$  such that $\displaystyle\sup_{\tau>0}\Vert f^\tau_0 \Vert_{L^1}\leq  \mathbf{\Lambda}_2. $ Then
\begin{align}\label{bound-1-sec1}
    \sup_{t\in[0,T]}  \int_{\mathbb{T}^d}\int_{\mathbb{R}^d}\vert f_\tau(t)\vert  drdx \leq \int_{\mathbb{T}^d}\int_{\mathbb{R}^d}\vert f_0^\tau\vert drdx \leq \mathbf{\Lambda}_2.
\end{align}
Moreover, there exists a subsequence  $(f_{\tau_k } )_k$ of $(f_\tau)_\tau$   and  $\nu\in L^\infty(0,T;\mathcal{M}(\mathbb{T}^d\times \mathbb{R}^d))$ such that 
\begin{align}
     f_{\tau_k } \stackrel{\ast}{\rightharpoonup} \nu  \text{ in } L^\infty(0,T;\mathcal{M}(\mathbb{T}^d\times \mathbb{R}^d)) \text{ as } k \to +\infty,
\end{align}
and $\nu$ solves, in the sense of distributions \footnote{It is the dual space of $C^\infty_c(]0,T[\times \T^d\times \R^d)$, denoted by $\mathcal{D}^\prime$., the following
\begin{align}\label{eqn-radon-thm}
    \Div_r\big(\alpha r\nu+\nabla_r \nu+\dfrac{1}{2}A_b(r)\nabla_r \nu\big)=0 \text{ in } \mathcal{D}^\prime.
\end{align}}
 \end{enumerate}
    
 \end{theorem} 
  \begin{proof}
  It is a consequence of    \autoref{propo-estimates},  \autoref{prop-cv-to-0} and \autoref{prop-est-L1}.   \end{proof}
      \begin{remark}\label{Rmq-radon-stretching}
          Since   the differential operator acts only with respect to $r$-variable, a particular solution are of the form $\nu=\mathbf{g}_x(\cdot)\otimes \mu_r,$ where $\mathbf{g}_x$ is an element of $L^\infty(0,T;\mathcal{M}(\mathbb{T}^d))$ and $\mu_r \in \mathcal{M}(\mathbb{R}^d)$ solving
    \begin{align}\label{1stsolution-measure}
    \Div_r\big(\alpha r\mu_r+\nabla_r \mu_r+\dfrac{1}{2}A_b(r)\nabla_r \mu_r \big)=0 \text{ in } (C_c^\infty(\mathbb{R}^d))^\prime.
\end{align}
On the other hand, $P_\alpha(r)=(1+\frac{\vert r\vert^2}{2})^{-\alpha} \in C^\infty_b(\mathbb{R}^d)$ and  $\alpha rP_\alpha(r)+\nabla_r P_\alpha(r)+\dfrac{1}{2}A_b(r)P_\alpha(r)=0.$
Define the following Radon measure $\mu_\alpha$ 
\begin{align}\label{explicit-Radon-sol}
     \langle  \mu_\alpha,\varphi\rangle_{\mathcal{M},C_c}:=\int_{\mathbb{R}^d}  P_\alpha(r)\varphi(r)dr, \quad \forall \varphi\in C_c(\mathbb{R}^d).
\end{align}
Indeed, 
let $\varphi\in C_c(\mathbb{R}^d)$, one gets
\begin{align*}
   \vert  \langle P_\alpha,\varphi\rangle_{\mathcal{M},C_c}\vert =\vert \int_{\mathbb{R}^d}  P_\alpha(r)\varphi(r)dr \vert \leq  \Vert \varphi\Vert_{C_c}\int_{\text{supp}(\varphi)}P_\alpha(r)dr\leq C\Vert \varphi\Vert_{C_c}
\end{align*}
for some positive constant $C.$ Consequently, $\mathbf{g}_x(\cdot)\otimes \mu_\alpha \in L^\infty(0,T;\mathcal{M}(\mathbb{T}^d\times \mathbb{R}^d)) $ solves \eqref{eqn-radon-thm}.
      \end{remark}
     \subsubsection{Non-uniqueness of measure-valued solutions  and selection principle ($0<\alpha\leq \dfrac{d}{2}$)\label{Subsection-discussion-uniq}}
            Since   the differential operator in \eqref{eqn-radon-thm} acts only with respect to $r$-variable,
            the solutions has the form  
            \begin{align}\label{solution-form.gle}
                \nu=\mathfrak{g}(t,x)\otimes \varrho_r
            \end{align}
             where $\mathfrak{g}$ is any element of $L^\infty(0,T;\mathcal{M}(\mathbb{T}^d))$ and 
   $\varrho_r \in \mathcal{M}(\mathbb{R}^d)$  is the solution of 
\begin{align}\label{eqn-radon-thm-**}
    \Div_r\big(\alpha r\varrho_r+\nabla_r \varrho_r+\dfrac{1}{2}A_b(r)\nabla_r \varrho_r\big)=0 \text{ in } \mathcal{D}^\prime (\mathbb{R}^d).
\end{align}

$\bullet$ \textbf{(Non-uniqueness of $\nu$)}
Notice that at this level, the uniqueness fails by choosing $\mathfrak{g}_1,\mathfrak{g}_2\in L^\infty(0,T;\mathcal{M}(\mathbb{T}^d))$ such that $\mathfrak{g}_1\neq \mathfrak{g}_2$ and  notice from \autoref{Rmq-radon-stretching} that  $\mathfrak{g}_1(t,x)\otimes \mu_r$ and $\mathfrak{g}_2(t,x)\otimes \mu_r$ are two distinct solutions to \eqref{eqn-radon-thm}.\\


$\bullet$  \textbf{(The uniqueness of $\varrho_r$)} Let us discuss the uniqueness for the solution  to  \eqref{eqn-radon-thm-**}. Namely, $\varrho_r$ satisfies (see \eqref{passge-limit})
\begin{align}\label{sense-RAdon--}
    \langle \varrho_r,  -\alpha r\cdot\nabla_r\varphi+ \Delta_r\varphi  +\dfrac{1}{2} \Div_r(A_b(r)\nabla_r\varphi)\rangle_{\mathcal{M},C_c} ds=0 \text{ for any } \varphi \in C^2_c(\mathbb{R}^d).
\end{align}
By using    Lebesgue decomposition theorem of Radon measures (cf. \textit{e.g.} \cite[Theorem 6.10]{Rudin}),  there exists a unique pair $(\varrho_{ac}^r,\varrho_s^r)$   of  measures such that 
$   \varrho_r=\varrho_s^r+\varrho_{ac}^r$
where
\begin{itemize}
    \item  $\varrho_{ac}^r<< \lambda_d$
 (absolutely continuous with respect to Lebesgue measure on $\mathbb{R}^d$). Moreover, $\varrho_{ac}=h(r)dr$ with $h\in L^1_{loc}(\mathbb{R}^d).$
 \item  $\varrho_s^r  \perp \lambda_d$  (singular with respect to Lebesgue measure) and may include Dirac measures. In other words, concentrations are authorized in the space $\mathcal{M}(\mathbb{R}^d).$
\end{itemize}
As we already indicated  in \autoref{Rmq-radon-stretching} that $\mu_r(dr)= (1+\frac{\vert r\vert^2}{2})^{-\alpha} dr \in \mathcal{M}(\mathbb{R}^d)$ solves \eqref{1stsolution-measure}.
In accordance with  Lebesgue decomposition theorem,  $\mu_r$ belongs to $\varrho_{ac}^r$ \textit{i.e.} 
\begin{align}\label{decompostion-our-case}
  \varrho_{ac}^r=\mu_r+C\mu_{r,2}, C\geq 0 \text{ where } \mu_{r,2} \text{ is another  absolutely continuous  component}.  
\end{align}
In general, it is difficult to explicitly compute all the solution to \eqref{eqn-radon-thm-**} but if uniqueness holds, then it is necessary to get $\varrho_r=\mu_r$,  since \eqref{eqn-radon-thm-**} is a linear equation. This is the goal of the next proposition. 
\begin{proposition}\label{prop-non-uniq-mesure}
    There exists a unique Radon measure solution to \eqref{eqn-radon-thm-**}, in the sense \eqref{sense-RAdon--}.
\end{proposition}
\begin{proof}
Recall that \eqref{sense-RAdon--} is a linear PDE. Hence the difference of two solutions, still denoted by $\varrho_r$,    satisfies  \eqref{sense-RAdon--}.
In order to prove uniqueness, we use the standard  duality argument and  Riesz representation theorem (see \textit{e.g.} \cite[Theorem 2.14]{Rudin}) for Radon measures. Namely, the uniqueness follows if we show
\begin{align}\label{uniq.cretirea-Radon}
    \langle \varrho_r, \psi\rangle_{\mathcal{M},C_c^\infty}=0 \text{ for any } \psi\in C^\infty_c(\mathbb{R}^d).
\end{align}
Let us see if \eqref{uniq.cretirea-Radon} holds. To abbreviate the notation, define
\[
\mathcal{L}^* \varphi = -\alpha r \cdot \nabla_r \varphi + \Delta_r \varphi + \dfrac{1}{2}\operatorname{div}_r(A_b(r)\nabla_r \varphi).
\]
From \eqref{sense-RAdon--}, we have $
\langle \varrho_r, \mathcal{L}^* \varphi \rangle = 0$ for any $ \varphi \in C_c^2(\mathbb{R}^d).
$
Let us analyze the adjoint equation. Namely,
let $\psi\in  C^\infty_c(\mathbb{R}^d) $ and denote by $B_M$ the ball with radius $M$ such that $\mathrm{supp}(\psi)\subset B_M$. Then, consider  the following  problem in $B_M$ with Dirichlet boundary condition:
\begin{align}\label{equ-adjoint-uniq}
\begin{cases}
\mathcal{L}^* \varphi = \psi & \text{in } B_M, \\
\varphi = 0 & \text{on } \partial B_M.
\end{cases}    
\end{align}
The existence  of solution in $H^1_0(B_M)$ to \eqref{equ-adjoint-uniq} follows by using \textit{"Fredholm alternative for linear PDEs"}  (see \textit{e.g.} \cite[Theorem 4 page 303]{evans2022partial}). Since $\partial B_M$ is $C^\infty$-boundary, 
the classical elliptic theory ( see  \textit{e.g.} \cite[Theorem 6 page 326]{evans2022partial}) ensures the 
 existence of $\varphi \in C^\infty(\overline{B_M}).$ Then, 
extend $\varphi$ by zero outside $B_M$ to get $\varphi \in C_c^2(\mathbb{R}^d)$. Therefore, we obtain
\begin{align}
    \langle \varrho_r, \psi\rangle_{\mathcal{M},C_c^\infty}=    \langle \varrho_r, \mathcal{L}^* \varphi \rangle_{\mathcal{M},C_c}=0.
\end{align}
The  claim  \eqref{uniq.cretirea-Radon} follows from the arbitrariness of $\psi$.
\end{proof}

Finally,  let us discuss  a selection principle for $\mathfrak{g}.$ \\

$\bullet$ \textbf{(Selection principle for $\mathfrak{g}$)} Since    $\displaystyle\sup_{\tau>0}\Vert f^\tau_0 \Vert_{L^1}\leq  \mathbf{\Lambda}_2,$ by using \eqref{bound-1-sec1}  we can set $g_\tau=\int_{\mathbb{R}^d}f_\tau dr$  and integrate  \eqref{Limit-FP-eqn-2} with respect to $r$,  therefore we get
\begin{align}\label{transport-measure}
\begin{cases}&\partial_tg_\tau+\Div_x(u_Lg_\tau)=0, \qquad \tau >0\\
&g_\tau|_{t=0}=g_0^\tau=\int_{\mathbb{R}^d}f_0^\tau dr \in L^1(\mathbb{T}^d).
\end{cases}\end{align}
Note that \eqref{transport-measure} is linear transport equation studied  in the seminal work \cite{diperna1989ordinary}. Therefore,
by writing   \cite[Corollary II.1 and Theorem II.3.]{diperna1989ordinary} in the case of torus, we have
\begin{proposition}\label{prop-transport-equation-1}
    Let $u_L\in C\big([0,T], C^2(\mathbb{T}^d;\mathbb{R}^d)\big)$, $\Div_x (u_L)=0$ and $g_0^\tau \in L^1(\mathbb{T}^d)$.   There exists a unique \emph{weak solution}  $g_\tau\in L^\infty([0,T]; L^1(\mathbb{T}^d)) $ to \eqref{transport-measure} in the following sense
\[
\int_0^T \int_{\mathbb{R}^d} \big( g_\tau(t,x) \, \partial_t \varphi(t,x) + g_\tau(t,x) \, u_L(t,x) \cdot \nabla \varphi(t,x) \big) \, dx \, dt
+ \int_{\mathbb{R}^d} g_0^\tau(x) \, \varphi(0,x) \, dx = 0.
\]
for every test function $\varphi \in C_c^\infty([0,T) \times \mathbb{T}^d).$ Moreover, the solution is renormalized:
    \[
    \partial_t \beta(g_\tau) + u_L \cdot \nabla \beta(g_\tau) = 0
    \]
    for all $\beta \in C^1(\mathbb{R})$ with bounded derivative and vanishing near $0$.
\end{proposition}
\textit{A priori }, at the limit, we will be interested in \eqref{transport-measure} with initial Radon measure data ( a weak-* limit, as $\tau \to 0$, of $(f_0^\tau)_\tau$ in $\mathcal{M}(\mathbb{T}^d)$), which requires more involved analysis so we restrict ourself to the case $L^1$-data  at the limit, in order to explain the idea of selection.
\begin{proposition}
Let $g_0 \in L^1(\mathbb{T}^d)$ and 
    assume moreover that $(g_0^\tau)_\tau$ in \eqref{transport-measure} converges to $g_0$ in  $L^1(\mathbb{T}^d)$ as $\tau \to 0$. Then $g_\tau$ converges to $\widetilde{g}$ in $L^\infty([0,T]; L^1(\mathbb{T}^d))$ as $\tau \to 0$,  where  $\widetilde{g}\in L^\infty([0,T]; L^1(\mathbb{T}^d))$  is  the unique \emph{weak solution}   to 
    \begin{align}\label{transport-measure-limit}
\begin{cases}&\partial_t\widetilde{g}+\Div_x(u_L\widetilde{g})=0,\\
&\widetilde{g}|_{t=0}=g_0.
\end{cases}\end{align}
\end{proposition}
\begin{proof}
The existence  and uniqueness to   \eqref{transport-measure-limit} follows as above, namely by using \cite[Corollary II.1 and Theorem II.3.]{diperna1989ordinary}. On the other hand, using the linearity of \eqref{transport-measure} and \eqref{transport-measure-limit}, the difference $g_\tau-\widetilde{g}$ is the    unique \emph{weak solution}  to   
\begin{align}\label{transport-measure-difference}
\begin{cases}&\partial_t(g_\tau-\widetilde{g})+\Div_x(u_L(g_\tau-\widetilde{g}))=0, \qquad \tau >0\\
&g_\tau-\widetilde{g}|_{t=0}=g_0^\tau-g_0 \in L^1(\mathbb{T}^d).
\end{cases}\end{align}
Recall that the solution is renormalized. By choosing appropriate sequence $\beta_M\in C^1(\mathbb{R})$ with bounded derivative, vanishing near $0$ and converges to the absolute value function ( note that we can also use similar computation as in \autoref{prop-est-L1} as well), we get
\begin{align}
     \sup_{t\in[0,T]}  \int_{\mathbb{T}^d}\vert g_\tau-\widetilde{g}\vert  dx \leq \int_{\mathbb{T}^d}\vert g_0^\tau-g_0\vert dx \to 0 \text{ as } \tau \to 0.
\end{align}\end{proof}
Consequently, a naturel selection to avoid the trivial non-uniqueness, we can select $\mathfrak{g}$ in \eqref{solution-form.gle} to be  $\widetilde{g}$,   the unique \emph{weak solution}   of \eqref{transport-measure-limit}.  
\begin{remark}
    A natural assumption guarantees that $(g_0^\tau)_\tau$ 
converges towards $g_0$ in  $L^1(\mathbb{T}^d)$ is to assume $f_0^\tau$ converges to some $f_0\in L^1(\mathbb{T}^d\times\R^d)$, instead assuming only the boundedness in $L^1(\mathbb{T}^d\times\R^d)$ as in \autoref{Main-thm-2}-ii. 
\end{remark}

\subsection{On the Role of the Unbounded Potential}\label{Sec-Unbounded-potential} Let us show that the unbounded nature of the restoring force leads to the emergence of a  measure valued limit if $0<\alpha\leq \frac{d}{2}.$ Note that the unbounded nature of the force allows the polymer end-to-end vector to take values in $\mathbb{R}^d$, which plays a role in the lack of  $L^1$-compactness in this regime.
\begin{proposition}\label{prop-loss-L1} Under the assumptions of \autoref{Main-thm-2} (ii),  let  $f_\tau \in L^\infty(0,T;L^1(\mathbb{T}^d\times \mathbb{R}^d))$ be  the unique solution  to \eqref{Limit-FP-eqn-2} satisfying \begin{align*}
     \sup_{t\in[0,T]}  \int_{\mathbb{T}^d}\int_{\mathbb{R}^d}\vert f_\tau(t)\vert  drdx \leq  \mathbf{\Lambda}_2,\quad  \forall\tau >0.
\end{align*}
Then $(f_\tau)_\tau$  does not contain any weakly-* convergent subsequence in $L^\infty(0,T;L^1(\mathbb{T}^d\times \mathbb{R}^d))$.
\end{proposition}
\begin{proof} We reason by contradiction.
Assume that $(f_\tau)_\tau$ has a subsequence,  denoted  by $(f_{\tau_k})_{\tau_k}, k\in \mathbb{N}$, which   converges weakly-*  to some $\overline{f}$ in $L^\infty(0,T;L^1(\mathbb{T}^d\times \mathbb{R}^d))$. Namely, the following holds
\begin{align}\label{cv-L1-contra--1}
\int_0^T    \int_{\mathbb{T}^d}\int_{\mathbb{R}^d}f_{\tau_k}(s,x,r)\xi(s)\psi(x)\phi(r) dsdxdr \to  \int_0^T    \int_{\mathbb{T}^d}\int_{\mathbb{R}^d}\overline{f}(s,x,r)\xi(s)\psi(x)\phi(r) dsdxdr, 
\end{align}
for any $ \xi \in L^1([0,T]), \psi\in L^\infty(\mathbb{T}^d), \phi\in L^\infty(\mathbb{R}^d).$ 
On the other hand,  we know that  $f_{\tau_k}$ is the unique solution to \eqref{Limit-FP-eqn-2} for any $k\in \mathbb{N}$ and  it is  bounded in $L^\infty(0,T;L^1(\mathbb{T}^d\times \mathbb{R}^d))$. Therefore, it converges, with respect to the weak-* topology,  in $L^\infty(0,T;\mathcal{M}(\mathbb{T}^d\times \mathbb{R}^d))$  (up to subsequence) to  some $\widetilde{\nu}$ such that
\begin{align}
 &\Div_r\big(\alpha r\widetilde{\nu}+\nabla_r \widetilde{\nu}+\dfrac{1}{2}A_b(r)\nabla_r \widetilde{\nu}\big)=0 \text{ in } \mathcal{D}^\prime,\notag\\
\text{ and  } &\int_0^T    \int_{\mathbb{T}^d}\int_{\mathbb{R}^d}f_{\tau_k}(s,x,r)\xi(s)\psi(x)\phi(r) dsdxdr \to  \int_0^T    \langle  \widetilde{\nu}, \psi\phi\rangle \xi(s) ds, \label{cv-L1-contra--2}
\end{align}
for any $ \xi \in L^1([0,T]),\psi\in C_c(\mathbb{T}^d)$ and $\phi\in C_c(\mathbb{R}^d).$
We proved in \autoref{Subsection-discussion-uniq} that $\widetilde{\nu}$ has necessarily the form  $h \otimes(1+\frac{\vert r\vert^2}{2})^{-\alpha} dr$ for some $h\in L^\infty(0,T;\mathcal{M}(\mathbb{T}^d))$. Now, set  $\psi\in C_c(\mathbb{T}^d)$ and $\phi\in C_c(\mathbb{R}^d)$ in \eqref{cv-L1-contra--1}. Using the uniqueness of the limit, we obtain from  \eqref{cv-L1-contra--2} 
\begin{align*}
    \int_0^T    \int_{\mathbb{T}^d}\int_{\mathbb{R}^d}\overline{f}(s,x,r)\psi(x)\phi(r) \xi(s) dxdrds=  \int_0^T    \langle  h\otimes(1+\frac{\vert r\vert^2}{2})^{-\alpha} dr, \psi \phi\rangle \xi(s)ds.
\end{align*}
The arbitrariness of of $\psi,\phi$ and  $\xi$ gives  $\overline{f}=h\otimes(1+\frac{\vert r\vert^2}{2})^{-\alpha} dr$. Finally, recall that $$\int_{\mathbb{R}^d}(1+\frac{\vert r\vert^2}{2})^{-\alpha} dr=+\infty \text{ if } 0<\alpha\leq \frac{d}{2},$$ which contradicts the assumption  \(\overline{f} \in L^\infty(0,T; L^1(\mathbb{T}\times\mathbb{R}^d))\).
\end{proof}
\begin{remark} We can use a similar argument and   Dunford-Pettis theorem (see \textit{e.g.} \cite[Theorem 4.30]{Brezis}) to show that  
 the sequence \((f_\tau)_\tau\) is not uniformly integrable.
\end{remark}

\subsection{Physical interpretation}
From a physical point of view, 
the statistics of the polymer length $R$ have been investigated by several authors in the physical literature, see for instance  \cite{Balk,Gerashen,PicardoLanceVinc}. In the coil state,  the distribution of polymer
end-to-end vector $R$ is found to be power law
\begin{equation}
f\left(  R\right)  \sim R^{-1-\theta}\qquad\text{for relatively large
}R\label{power law}%
\end{equation}
with the exponent $\theta$ positive (so that $f$ is normalizable). The
exponent $\theta$ depends on the stretching properties of the turbulent flow:
the highest is the stretching intensity, the lowest is $\theta$. At $\theta=0$
one has the coil-stretch transition.\\
To make links with the physical literature on the \textit{coil-stretch transition}, it is worth recalling that the tail exponent $\theta$ is  related to the Lyapunov exponents of the turbulent flow and polymer relaxation time.  More precisely, if $\phi_{t}\left(  x\right)  $ denotes
the Lagrangian flow associated to the turbulent flow, and we set
\[
\mathbf{L}\left(  q\right)  =\lim_{t\rightarrow\infty}\frac{1}{t}%
\log\mathbb{E}\left[  \left\vert D\phi_{t}\left(  x\right)  \right\vert
^{q}\right]
\]
then $\theta$ satisfies 
$
\frac{\theta}{2\beta}=\mathbf{L}\left(  \theta\right).
$
It is important to mention that  the computation of $\mathbf{L}\left( \theta\right)  $ is not
easy  in general and  we would like to predict the exponent $\theta$ based on turbulence features. In our case, we obtain the following probability density
\begin{align*}
 p( r) =  \dfrac{1}{Z} (1+\frac{\vert r\vert^2}{2})^{-\alpha} \sim  \vert r\vert^{-2\alpha} \quad \text{ for relatively large } \vert r\vert,
\end{align*}
where  $\vert r\vert$ denotes the  length of the end-to-end vector  of the polymers and $\alpha$ depends on some interpretable constants, see \eqref{constant-meaning}. Taking into account the space dimension, we get $$p( \vert r\vert )\sim \vert r\vert^{-1+d-2\alpha}$$ in accordance with the prediction of \cite{Balk}.
Now, comparing to \eqref{power law} we get 
\begin{align}\label{tail-math-phys}
 \theta=2\alpha-d  = \dfrac{2}{\zeta C_d C_3^2}-d \text{ where } \zeta=\dfrac{\beta}{\tau}=\dfrac{\text{relaxation time of polymer}}{\text{dominant time-scale of small scale  turbulence}}.
\end{align}
 
Let us discuss the  results of \autoref{MAin-thm-1} and \autoref{Main-thm-2}. We distinguish two main cases.
\begin{enumerate}
    \item[i.]  If $\alpha>\dfrac{d}{2}$\footnote{Note that $\int_{\mathbb{R}^d} (1+\frac{\vert r\vert^2}{2})^{-\alpha} dr=\vert S^{d-1}\vert\int_0^{+\infty} (1+v^2)^{-\alpha}v^{d-1}dv<\infty$ if and only if  $\alpha>\dfrac{d}{2}$ ($\vert S^{d-1}\vert$ denotes the surface area  of the unit sphere in 
$\mathbb{R}^d$).
}, then we have  two sub-cases:
       \begin{itemize}
     \item The case $\alpha>\dfrac{d+1}{2}$ \footnote{Note that first absolute moment $\int_{\mathbb{R}^d}\vert r\vert (1+\frac{\vert r\vert^2}{2})^{-\alpha} dr=\vert S^{d-1}\vert\int_0^{+\infty} (1+v^2)^{-\alpha}v^{d}dv<\infty$ if and only if  $\alpha>\dfrac{d+1}{2}.$}: the average (mean) associated with   the probability density in  \autoref{MAin-thm-1} is finite and the most of polymers have an equilibrium size.  One can consider that polymers are in coil state in this regime.
    \item  The case $\dfrac{d}{2}<\alpha\leq\dfrac{d+1} {2}$: The mean is not finite  and  larger values of   polymer end-to-end vector are
more probable, which means that the most of polymers have large size and the polymers are in stretched state. Moreover, at $\alpha=\dfrac{d+1}{2}$   the average (mean) associated with   the probability density in  \autoref{MAin-thm-1} is not  finite anymore, this can be interpreted as the criterion for the coil-stretch transition. Note also that $\frac{d}{2}$ corresponds to $\theta=0$ from \eqref{tail-math-phys}, which is the coil-stretch transition threshold  following \cite{Balk}.
\end{itemize}

\item[ii.] The case $0<\alpha\leq \dfrac{d}{2}:$ In this case, we no longer have a probability density in \autoref{Main-thm-2}. Moreover, we don't have  \textit{a priori} the uniqueness of the limit as $\tau\to 0$.  Depending on the regularity of the initial data, we obtain different limits, the first one is the trivial limit. More importantly, the second limit $\nu,$ Radon measure valued one in \autoref{Main-thm-2}. The latter seems to be more natural because it corresponds to the $L^1$-framework and we are dealing with Fokker-Planck type equation (see \cite[Rmq. 11]{FlaTah24}). As noticed in  \autoref{Rmq-radon-stretching} and \autoref{Subsection-discussion-uniq}, we can define a family of explicit solutions given by $$ g(t,x)\otimes \mu_\alpha,\quad g\in L^\infty(0,T;\mathcal{M}(\mathbb{T}^d)) \text{ and } \mu_\alpha \text{ given by \eqref{explicit-Radon-sol}}.$$
Notice that $\mu_\alpha$ show a power law decay with the exponent $\alpha$ in the generalized sense of Radon measure.
From \eqref{constant-meaning}, as $\alpha$ decreases, the turbulent flow becomes stronger. Thus, we can say  that the  polymers are strongly stretched in comparison with the second sub-case above, where the pdf must be interpreted as a less regular object, namely as a  Radon measure absolutely continuous with respect to Lebesgue measure on $\mathbb{R}^d$, with density (Radon–Nikodym derivative) showing the power law decay behavior.
In addition, in the stretched state, the precise mathematics depends on the idealizations of
the model.  If we had introduced a superlinear damping instead of the linear
damping $-\frac{1}{\beta}R_t$, this would produce a sort of cut-off at very high lengths ( e.g. \textit{FENE} model, see  \autoref{outlook}), so that the behavior (\ref{power law}) would be true only
in a range
\[
R_{min}<<R<<R_{max}%
\]
and globally the function $f$ would still be a pdf. We refer to \autoref{outlook} for further discussion on \textit{FENE} model. In our idealization of
linear damping, the stretching may overcome the damping and lead to infinite
length in the asymptotic regime, which is the idealized signature of stretch state. More precisely,  note that the mathematical idealization in the Hookean model allows the infinite extensibility mechanism.   
Recall that the term including $A_b(r)$  (see the operator $\mathcal{L}$ in \eqref{operator-L-def}) is a consequence of stretching effect, resulting after taking the first scaling limit \cite{FlaTah24} and  it is responsible for the loss of compactness in $L^1$ if $0<\alpha\leq \frac{d}{2}$. Indeed, in the absence of the term $A_b(r)$ we see that  $\mathcal{L}$ is the  Fokker-Planck operator associated  with Ornstein-Uhlenbeck process and we expect a Gaussian distribution as $\tau\to 0$, even if the finite extensibility constraint is missing. On the other hand, since $A_b(r)r = |r|^2 r$ and  $|r|^2$is the smallest eigenvalue of the matrix $A_b(r)$, the effective diffusion increases with $|r|$ and  enhance the outward motion for large values of  $|r|$. Consequently, it allows the  mass to escape to infinity and, mathematically speaking, it indicates a failure of tightness of the associated probability measures and the loss of  compactness   in $L^1$-space, see \autoref{Sec-Unbounded-potential}.  In other words,  the term including $A_b(r)$ overcomes the drift term containing $\alpha$ (which confines the mass)  if  $0<\alpha\leq \frac{d}{2}$. To conclude, the notion of "infinite length" reflects the possibility of polymers reaching infinite length and  this corresponds to  the loss of compactness from a mathematical point of view. Summarizing, this phenomenon is a consequence of  a combination of the stretching effect of the turbulence and the   infinite extensibility possibility of the Hookean model.  
\end{enumerate} 

In all cases, our main results justified rigorously match  the physical prediction (see \textit{e.g.} \cite{Balk}) concerning the power-law distribution of polymers embedded in turbulent flow.

\subsection{The coil-stretch operator}
A first step in the study of the limit as $\tau \to 0$ in \eqref{Limit-FP-eqn-2} requires understanding the singular term $\dfrac{1}{\zeta\alpha\tau}\Div_r\Big(\alpha rf_\tau+\nabla_rf_\tau+\dfrac{1}{2} A_b(r)\nabla_r f_\tau\Big)$. This will be the objective of this subsection.
    Notice that $P_\alpha(r)=(1+\frac{\vert r\vert^2}{2})^{-\alpha}$\footnote{In the following, we use  $p_\alpha$ for the renormalized  Cauchy-type function and  $P_\alpha$ for the non renormalized  one.} satisfies
\begin{align}\label{ground-state}
    \alpha rP_\alpha(r)+\nabla_rP_\alpha(r)+\dfrac{1}{2} A_b(r)\nabla_r P_\alpha(r)=0,
\end{align}
where we denoted $A_b(r)=(b+1)\lvert r\rvert^2 I-br\otimes r$. 
Thus
\begin{align*}
    \alpha rg+\nabla_rg+\dfrac{1}{2} A_b(r)\nabla_r g
    &=(1+\frac{\vert r\vert^2}{2})^{-\alpha}\big[\nabla_r\{g(1+\frac{\vert r\vert^2}{2})^{\alpha}\}+\dfrac{1}{2} A_b(r)\nabla_r \{g(1+\frac{\vert r\vert^2}{2})^{\alpha}\} \big]\\
    &=(1+\frac{\vert r\vert^2}{2})^{-\alpha}\big[\nabla_rg_\alpha+\dfrac{1}{2} A_b(r)\nabla_r g_\alpha\big],
\end{align*}
where (to simplify the notation) we used $g_\alpha=g(1+\frac{\vert r\vert^2}{2})^{\alpha}$. Thus, we can guess formally  that the penalization $\tau \to 0$ leads to 
$$f_\tau(t,x,r) \simeq \rho(t,x)p_\alpha(r),$$
where $p_\alpha(r)=\dfrac{1}{Z}(1+\frac{\vert r\vert^2}{2})^{-\alpha}$ ( $Z$ is a normalizing
factor) is generalized Cauchy  density with parameter $\alpha>\frac{d}{2}.$
The aim of this subsection is to study some spectral properties of the following operator
\begin{align}\label{operator-L-def}
  \mathcal{L}(f)=\Div_r\Big(\alpha rf+\nabla_rf+\dfrac{1}{2} \left((b+1)\lvert r\rvert^2 I-br\otimes r\right)\nabla_r f\Big), \quad b\in \{1,2\}.
\end{align}
Let $d\in\{2,3\}$ and consider the following  space \begin{align*}
    \mathbf{X}_\alpha=\{ g:\mathbb{R}^d\to \mathbb{R}: \quad \int_{\mathbb{R}^d}g^2(r)(1+\frac{\vert r\vert^2}{2})^\alpha dr<\infty\} \quad \alpha > 0.
\end{align*}
Notice that $\mathbf{X}_\alpha$ is a Hilbert  space endowed with its natural inner product, namely
$$(f,g)_{\mathbf{X}_\alpha}=\int_{\mathbb{R}^d}f(r)g(r)(1+\frac{\vert r\vert^2}{2})^\alpha dr .$$
We introduce the following unbounded operator $\mathcal{L}$ acting on $\mathbf{X}_\alpha$ as follows
\begin{align*}
    \mathcal{L}: \mathbf{X}_\alpha &\to \mathbf{X}_\alpha\\
    g&\mapsto  \mathcal{L}(g)=\Div_r\Big((1+\frac{\vert r\vert^2}{2})^{-\alpha}\big[\nabla_rg_\alpha+\dfrac{1}{2} A_b(r)\nabla_r g_\alpha\big]\Big),
\end{align*}
with  domain $$\mathcal{D}(\mathcal{L})=\{ g\in \mathbf{X}_\alpha, \quad \nabla_r\cdot\Big((1+\frac{\vert r\vert^2}{2})^{-\alpha}\big[(I+\dfrac{1}{2} A_b(r))\nabla_rg_\alpha\big]\Big) \in \mathbf{X}_\alpha\}.$$
 It is clear that $A_b(r)^*=A_b(r)$. Let us introduce the following space $\mathbf{Z}_\alpha$, which serves to prove  the closedness of $-\mathcal{L}.$
 \begin{align*}
     \mathbf{Z}_\alpha=\{g\in \mathbf{X}_\alpha: \quad \int_{\mathbb{R}^d}(1+\frac{\vert r\vert^2}{2})^{-\alpha+1}\vert \nabla_r f(1+\frac{\vert r\vert^2}{2})^{\alpha}\vert^2dr<+\infty \}.
 \end{align*}
It is not difficult to check that $\mathbf{Z}_\alpha$ is a Hilbert space equipped by its natural inner product.
 Let us  present the main spectral properties of the operator $\mathcal{L}.$ 
\begin{proposition} Let $\alpha > \dfrac{d}{2}$,
    the operator $-\mathcal{L}$ is self adjoint on $\mathbf{X}_\alpha$ and satisfies:
    \begin{enumerate}
        \item $-\mathcal{L}$ is positive and $Ker(\mathcal{L})=\{cP_\alpha,\quad  c\in \mathbb{R}\},$
        \item $R(\mathcal{L})=\{ g\in \mathbf{X}_\alpha: \quad \int_{\mathbb{R}^d}g(r) dr=0\},$
        \item For all $g\in R(\mathcal{L})$ there exists $f\in \mathcal{D}(\mathcal{L})$ such that $\mathcal{L}(f)=g.$  The solution is unique under the solvability condition $\int_{\mathbb{R}^d}f(r) dr=0.$
    \end{enumerate}
\end{proposition}
\begin{proof} $-\mathcal{L}$ is symmetric on $\mathbf{X}_\alpha$. Indeed,  let $f,g\in \mathcal{D}(\mathcal{L}).$ Since  $A_b(r)^t=A_b(r)$,  we get 
\begin{align*}
    (\mathcal{L}(f),g)_{\mathbf{X}_\alpha}&=-\int_{\mathbb{R}^d}\Big((1+\frac{\vert r\vert^2}{2})^{-\alpha}\big[\nabla_rf_\alpha+\dfrac{1}{2} A_b(r)\nabla_r f_\alpha\big]\Big)\cdot \nabla_rg_\alpha dr\\
    & =-\int_{\mathbb{R}^d}\nabla_rf_\alpha\cdot\Big((1+\frac{\vert r\vert^2}{2})^{-\alpha}\big[\nabla_rg_\alpha+\dfrac{1}{2} A_b(r)\nabla_r g_\alpha\big]\Big) dr=(f,\mathcal{L}(g))_{\mathbf{X}_\alpha}.
\end{align*}
In particular, since $A_b(r)\nabla_r f_\alpha\cdot \nabla_rf_\alpha \geq \vert r\vert^2\vert \nabla_rf_\alpha\vert^2,$ 
we obtain
    \begin{align}
    (-\mathcal{L}(f),f)_{\mathbf{X}_\alpha}&=\int_{\mathbb{R}^d}\Big((1+\frac{\vert r\vert^2}{2})^{-\alpha}\big[\nabla_rf_\alpha+\dfrac{1}{2} A_b(r)\nabla_r f_\alpha\big]\Big)\cdot \nabla_rf_\alpha dr\notag\\
    &=\int_{\mathbb{R}^d}(1+\frac{\vert r\vert^2}{2})^{-\alpha}\big[\vert \nabla_rf_\alpha\vert^2+\dfrac{1}{2} A_b(r)\nabla_r f_\alpha\cdot \nabla_rf_\alpha\big] dr\notag\\
    &\geq \int_{\mathbb{R}^d}(1+\frac{\vert r\vert^2}{2})^{-\alpha}\big[(1+\frac{\vert r\vert^2}{2})\vert \nabla_rf_\alpha\vert^2\big] dr=\int_{\mathbb{R}^d}(1+\frac{\vert r\vert^2}{2})^{-\alpha+1}\vert \nabla_rf_\alpha\vert^2 dr\geq 0.\label{ker-eqn}
\end{align}
Thus $-\mathcal{L}$ is symmetric and positive. On the other hand,  notice that $\mathcal{D}(\mathcal{L})$ is dense in $\mathbf{X}_\alpha$ since $C^\infty_c(\mathbb{R}^d) \subset \mathcal{D}(\mathcal{L}).$ By using Lax-Milgram theorem, we obtain the existence and uniqueness of $h\in \mathbf{Z}_\alpha$ for any $g\in \mathbf{X}_\alpha $ such that 
\begin{align*}
    \int_{\mathbb{R}^d}(1+\frac{\vert r\vert^2}{2})^{-\alpha}\big[\nabla_rh_\alpha+\dfrac{1}{2} A_b(r)\nabla_r h_\alpha\big] \cdot \nabla_r\phi_\alpha dr&+ \int_{\mathbb{R}^d}h\phi(1+\frac{\vert r\vert^2}{2})^{\alpha}dr\\&=\int_{\mathbb{R}^d}g\phi(1+\frac{\vert r\vert^2}{2})^{\alpha}dr,\quad \forall \phi \in \mathbf{Z}_\alpha.
\end{align*}
In other words, $ (I-\mathcal{L})h=g.$ Let us show that $I-\mathcal{L}$ is closed, let $(h_n,(I-\mathcal{L})h_n)$ converges to $(h,g)$ in $\mathbf{X}_\alpha\times \mathbf{X}_\alpha.$ We need to prove that $h\in \mathcal{D}(\mathcal{L})$ and $(I-\mathcal{L})h=g.$ By using Lax-Milgram theorem, there exists a unique $\psi \in \mathbf{Z}_\alpha$ such that $(I-\mathcal{L})\psi=g$ hence $\psi\in \mathcal{D}(\mathcal{L})$. Therefore  $(I-\mathcal{L})(\psi-h_n)$ converges to $0$ in $\mathbf{X}_\alpha$ which ensures that $\psi-h_n$ goes to $0$ in $\mathbf{X}_\alpha$ and $\psi=h \in \mathbf{X}_\alpha.$ We deduce that $-\mathcal{L}$ is also closed and $-1$ belongs to the resolvent of $-\mathcal{L}.$
\begin{enumerate}
    \item From \eqref{ker-eqn}, we get    $Ker(\mathcal{L})=\{cP_\alpha,  c\in \mathbb{R}\}$ since  $f_\alpha=f(1+\frac{\vert r\vert^2}{2})^{\alpha}$.

\item  Since $-\mathcal{L}$ is closed symmetric and $-1$ belongs to the resolvent of $-\mathcal{L}$ then it is self-adjoint thanks to \cite[Cor. p. 137]{ReedSimon75}. By using \cite[Thm. 2.19]{Brezis} we have $$R(\mathcal{L})=Ker(\mathcal{L})^\perp:=\{ g\in \mathbf{X}_\alpha:  (g,P_\alpha)_{\mathbf{X}_\alpha}=0\}=\{ g\in \mathbf{X}_\alpha:  \int_{\mathbb{R}^d}g(r) dr=0\}.$$
\end{enumerate}
The point (3) holds abviously.
\end{proof}
\begin{remark} Notice that in the case
   $0<\alpha\leq d/2$,  we still have 
    $-\mathcal{L}$ is positive. On the other hand, $P_\alpha(r)=(1+\frac{\vert r\vert^2}{2})^{-\alpha} \notin \mathbf{X}_\alpha $ which implies that $Ker(\mathcal{L})=\{0\}.$
\end{remark}

\section{Uniform estimates and continuity equation}\label{section-estimate}
\subsection{Uniform estimates with respect to $\tau$}
Let $\tau>0$ and $\alpha>0,$ the aim of this section is to prove some uniform estimates with respect to $\tau$.  First, we improve the regularity with respect to  the $r-$variable  comparing to the one proved in \cite{FlaTah24} with  a  bound  depending on $\tau$ (see \autoref{lemma-appendix}). Then, we combine \autoref{lemma-appendix} and \eqref{ground-state} to obtain  uniform estimates with respect to $\tau.$ \\

To simplify the notation, let's introduce the following spaces
\begin{align*}
    \mathcal{Y}&=\{\varphi\in H:    \nabla_r\varphi\in H,  \nabla_x\varphi\in  L^2(\T^d\times\R^d)\}.\\
      H_\alpha&=\{ g:\mathbb{T}^d\times\mathbb{R}^d\to \mathbb{R}: \quad \int_{\mathbb{T}^d}\int_{\mathbb{R}^d}g^2(x,r)(1+\frac{\vert r\vert^2}{2})^\alpha drdx<\infty\}\quad \alpha > 0.
\end{align*}
By approximation arguments and \autoref{Thm-scaling-N} we can deduce that   \eqref{Limit-FP-eqn-2} is satisfied  in $\mathcal{Y}^\prime$-sense,  namely there exists a unique  quasi-regular weak solution $f_\tau$ such that 
 \begin{align}\label{regularity-f}
 f_\tau\in L^2_{w-*}(\Omega ;L^\infty([0, T];H))),\quad \nabla_rf_\tau\in L^2(\Omega ;L^2([0, T];H))     
 \end{align}
  and  $f_\tau(t):=f_\tau(t,x,r)$ satisfies P-a.s. for any $t\in [0,T]:$ \begin{align}&\int_{\mathbb{T}^d}\int_{\mathbb{R}^d}f_\tau(t)\phi drdx-\int_{\mathbb{T}^d}\int_{\mathbb{R}^d}f_0^\tau\phi drdx\label{equ-final-I}\\&= \int_0^t\int_{\mathbb{T}^d}\int_{\mathbb{R}^d}f_\tau(s)\left(u_L(s)\cdot \nabla_x \phi+ (\nabla u_L(s)r) \cdot\nabla_r\phi\right) dr dxds\notag\\	&-	\dfrac{1}{\zeta\alpha\tau}\int_0^t\int_{\mathbb{T}^d}\int_{\mathbb{R}^d} \alpha f_\tau(s)r\cdot\nabla_r\phi+\nabla_rf_\tau(s) \cdot\nabla_r\phi  +\dfrac{1}{2}A_b(r) \nabla_rf_\tau(s)\cdot\nabla_r\phi dr dxds, \notag\end{align} 
  for any $\phi \in \mathcal{Y}.$
In order to derive some necessary estimates,
we  introduce  a    regularization  kernel. More  precisely,     let $\delta>0$    and      $\Theta$ be   a smooth  radially symmetric density of a probability measure on
$\R^d$, compactly supported in $B(0, 1)$    and define  the approximation of identity for the convolution on $\R^d$ as  $\Theta_\delta(y)=\dfrac{1}{\delta^d}\Theta(\dfrac{y}{\delta}).$       Since   we  are working on  $\T^d\times \R^d,$  we  recall  that    for any  integrable function $h$    on $\T^d$, $h$ can be extended
periodically to a locally integrable function on the whole $\R^d$   and
convolution $\Theta_\delta * h$ is    meaningful   and  $[h]_\delta:=\Theta_\delta * h$ is still a $C^\infty$-periodic function.  \\

Now, we present the following lemma which serves for the proof of \autoref{propo-estimates}.
\begin{lemma}\label{lemma-appendix}
 There exists   $\mathbf{C}=\dfrac{d/2+10\alpha}{\zeta\tau}+2\alpha\Vert \nabla_xu_L \Vert_\infty$ such that 
\begin{align}
   \sup_{t\in [0, T]} \int_{\mathbb{T}^d}\int_{\mathbb{R}^d}f_\tau(t)^2(1+\frac{\vert r\vert^2}{2})^\alpha  drdx&+\dfrac{1}{2\zeta\tau\alpha}\int_0^T\int_{\mathbb{T}^d}\int_{\mathbb{R}^d}  \vert \nabla_r f_\tau(s)\vert^2(1+\frac{\vert r\vert^2}{2})^{\alpha+1}drdxds\notag\\
&\leq  e^{\mathbf{C}T}\int_{\mathbb{T}^d}\int_{\mathbb{R}^d}(f_0^\tau)^2(1+\frac{\vert r\vert^2}{2})^\alpha  drdx\leq  \mathbf{\Lambda}e^{\mathbf{C}T}. \label{estim-beta-1}\end{align}
\end{lemma}
Although the proof of \autoref{lemma-appendix} follows by  classical  arguments,  we give a proof in \autoref{Appendix} for the convenience of the reader.
\begin{remark}($L^2$-estimate)
 It is worth mentioning that we can prove the $L^2$ estimates using arguments similar to the one used in \autoref{Appendix} but the bound depends on $\dfrac{1}{\tau}.$  For the convenience of the reader, let us give just the formal argument. The estimate is based on the analysis of the following 
    \begin{align*}
        &\dfrac{1}{\zeta\alpha\tau}\langle\Div_r\Big(\alpha rf_\tau+\nabla_rf_\tau+\dfrac{1}{2} A_b(r)\nabla_r f_\tau\Big), f_\tau\rangle\\&=-\dfrac{1}{\zeta\alpha\tau}\int_{\mathbb{T}^d}\int_{\mathbb{R}^d}  \Big(\alpha rf_\tau+\nabla_rf_\tau+\dfrac{1}{2} A_b(r)\nabla_r f_\tau\Big)\cdot \nabla_r f_\tau drdx\\
        &= -\dfrac{1}{\zeta\tau}\int_{\mathbb{T}^d}\int_{\mathbb{R}^d}   rf_\tau\cdot \nabla_r f_\tau drdx-\dfrac{1}{\zeta\alpha\tau}\int_{\mathbb{T}^d}\int_{\mathbb{R}^d}\overbrace{\big(\vert \nabla_rf_\tau\vert^2+\dfrac{1}{2} A_b(r)\nabla_r f_\tau\cdot \nabla_r f_\tau\big)}^{\geq 0} drdx\\
      &=  \dfrac{d}{2\zeta\tau}\int_{\mathbb{T}^d}\int_{\mathbb{R}^d}   f_\tau^2 drdx-\dfrac{1}{\zeta\alpha\tau}\int_{\mathbb{T}^d}\int_{\mathbb{R}^d}\overbrace{\big(\vert \nabla_rf_\tau\vert^2+\dfrac{1}{2} A_b(r)\nabla_r f_\tau\cdot \nabla_r f_\tau\big)}^{\geq 0} drdx.
    \end{align*}
This is why we consider the space $H_\alpha$ to obtain uniform bound with respect to $\tau.$
\end{remark}

\begin{proposition}\label{propo-estimates}
 Let $\alpha> 0$ and assume  that (H$_1$) holds. There exists $\mathbf{K}=2\alpha\Vert \nabla_xu_L \Vert_\infty$  such that
\begin{align}
    \sup_{t\in [0, T]} \Vert f_\tau(t)\Vert_{H_\alpha}^2
    &+\dfrac{1}{\zeta\alpha\tau}\int_0^T\int_{\mathbb{T}^d}\int_{\mathbb{R}^d}\vert \nabla_r (f_\tau(s)P_{-\alpha})\vert^2(1+\frac{\vert r\vert^2}{2})^{1-\alpha} dr dxds\leq  e^{\mathbf{K}T}\mathbf{\Lambda}:= \overline{\mathbf{K}}. \label{key-estimate-beta}
\end{align}
\end{proposition}

    \begin{proof}
Let $t\in [0,T]$ and      $\phi\in \mathcal{Y}$, if  we  denote  $X=(x,r)\in \T^d\times  \R^d$   and $\Theta_\delta(X)=\Theta_\delta(x)$,
then    $\phi_{\delta}:=\Theta_\delta*\phi$   is  an  appropriate test    function    in  \eqref{equ-final-I}, a classical computation yields
	\begin{align}\label{eqn-regulari-test}&\int_{\mathbb{T}^d}\int_{\mathbb{R}^d}[f_\tau(t)]_\delta\phi drdx-\int_{\mathbb{T}^d}\int_{\mathbb{R}^d}[f_0^\tau]_\delta\phi drdx\\&= -\int_0^t\int_{\mathbb{T}^d}\int_{\mathbb{R}^d}[u_L(s)\cdot \nabla_xf_\tau(s)]_\delta \phi+ [(\nabla u_L(s)r) \cdot\nabla_rf_\tau(s)]_\delta\phi dr dxds\notag\\	&+	\dfrac{1}{\zeta\alpha\tau}\int_0^t\int_{\mathbb{T}^d}\int_{\mathbb{R}^d} \alpha [\Div_r( f_\tau(s)r)]_\delta\phi dr dxds\notag \\&+\dfrac{1}{\zeta\alpha\tau}\int_0^t\langle[\Delta_rf_\tau(s) ]_\delta,\phi  \rangle+\dfrac{1}{2}\langle[\Div_r(A_b(r) \nabla_rf_\tau(s))]_\delta,\phi \rangle ds, \quad \forall \phi \in \mathcal{Y},\notag\end{align}
    where $\langle \cdot,\cdot\rangle$ denotes the duality pairing between $\mathcal{Y}^\prime$ and $\mathcal{Y}.$  Let us  introduce a  sequence of cut-off function $\mathcal{T}_M(s)=\max(-M,\min(s,M)), M\in \mathbb{N}$ and define $$T_M(r):=\mathcal{T}_M\big((1+\frac{\vert r\vert^2}{2})^\alpha\big),\quad r\in \mathbb{R}^d,\quad d=2,3.$$
 Recall that  $ \vert T_M(r)\vert \leq (1+\frac{\vert r\vert^2}{2})^\alpha$ and  $T_M(r)\uparrow (1+\frac{\vert r\vert^2}{2})^\alpha $ as $M\uparrow +\infty.$
 Let $\delta>0, M \in \mathbb{N},$
 notice that $\phi_{\delta,M}(x,r)= T_M(r)[f_\tau(\cdot)]_\delta(x,r)$ is an appropriate test function in \eqref{eqn-regulari-test}, therefore  we obtain
   \begin{align}&\int_{\mathbb{T}^d}\int_{\mathbb{R}^d}[f_\tau(t)]_\delta^2 T_M drdx-\int_{\mathbb{T}^d}\int_{\mathbb{R}^d}[f_0^\tau]_\delta^2 T_M drdx\label{est-M-delta}\\&= -\int_0^t\int_{\mathbb{T}^d}\int_{\mathbb{R}^d}[u_L(s)\cdot \nabla_xf_\tau(s)]_\delta  T_M[f_\tau(s)]_\delta+ [(\nabla u_L(s)r) \cdot\nabla_rf_\tau(s)]_\delta T_M[f_\tau(s)]_\delta dr dxds\notag\\	&+	\dfrac{1}{\zeta\alpha\tau}\int_0^t\int_{\mathbb{T}^d}\int_{\mathbb{R}^d} \alpha [\Div_r( f_\tau(s)r)]_\delta T_M[f_\tau(s)]_\delta dr dxds+\dfrac{1}{\zeta\alpha\tau}\int_0^t\langle [\Delta_rf_\tau(s) ]_\delta, T_M[f_\tau(s)]_\delta \rangle ds\notag\\ & +\dfrac{1}{\zeta\alpha\tau}\int_0^t\langle\dfrac{1}{2}[\Div_r(A_b(r) \nabla_rf_\tau(s))]_\delta, T_M[f_\tau(s)]_\delta \rangle ds.\notag\end{align}
Now, we use \eqref{estim-beta-1}  to prove  some uniform estimates independent of $\tau$. For the analysis of the term independent of $\tau$, see the first part of the proof of \autoref{lemma-appendix} in \autoref{Appendix}.  We will discuss only the $\tau$ dependent terms.
 \subsubsection{Uniform estimates with respect to $\tau$}\label{subsection-uniform-beta}
Note that
\begin{align*}
   & \dfrac{1}{\zeta\alpha\tau}\int_0^t\int_{\mathbb{T}^d}\int_{\mathbb{R}^d} \alpha [\Div_r( f_\tau(s)r)]_\delta T_M[f_\tau(s)]_\delta dr dxds+\dfrac{1}{\zeta\alpha\tau}\int_0^t\langle [\Delta_rf_\tau(s) ]_\delta, T_M[f_\tau(s)]_\delta \rangle ds\notag\\ & +\dfrac{1}{\zeta\alpha\tau}\int_0^t\langle\dfrac{1}{2}[\Div_r(A_b(r) \nabla_rf_\tau(s))]_\delta, T_M[f_\tau(s)]_\delta \rangle ds\\
   &=-\dfrac{1}{\zeta\alpha\tau}\int_0^t\int_{\mathbb{T}^d}\int_{\mathbb{R}^d} \big(\alpha r [f_\tau(s)]_\delta +\nabla_r [f_\tau(s)]_\delta+\dfrac{1}{2}A_b(r) \nabla_r[f_\tau(s)]_\delta\big)\cdot\nabla_r(T_M[f_\tau(s)]_\delta) dr dxds.
\end{align*}
By using \eqref{ground-state}, we obtain
\begin{align*}
    &\alpha r [f_\tau(s)]_\delta +\nabla_r [f_\tau(s)]_\delta+\dfrac{1}{2}A_b(r) \nabla_r[f_\tau(s)]_\delta\\&=\nabla_r \Big([f_\tau(s)]_\delta(1+\frac{\vert r\vert^2}{2})^\alpha\Big)(1+\frac{\vert r\vert^2}{2})^{-\alpha}+\dfrac{1}{2}A_b(r) \nabla_r\Big([f_\tau(s)]_\delta(1+\frac{\vert r\vert^2}{2})^\alpha\Big)(1+\frac{\vert r\vert^2}{2})^{-\alpha}.
\end{align*}
Therefore, we have
\begin{align*}
   I_\delta^M:=&-\dfrac{1}{\zeta\alpha\tau}\int_0^t\int_{\mathbb{T}^d}\int_{\mathbb{R}^d} \big(\alpha r [f_\tau(s)]_\delta +\nabla_r [f_\tau(s)]_\delta+\dfrac{1}{2}A_b(r) \nabla_r[f_\tau(s)]_\delta\big)\cdot\nabla_r(T_M[f_\tau(s)]_\delta) dr dxds\\
   &=-\dfrac{1}{\zeta\alpha\tau}\int_0^t\int_{\mathbb{T}^d}\int_{\mathbb{R}^d}\nabla_r ([f_\tau(s)]_\delta(1+\frac{\vert r\vert^2}{2})^\alpha)(1+\frac{\vert r\vert^2}{2})^{-\alpha}\cdot\nabla_r(T_M[f_\tau(s)]_\delta) dr dxds\\
   &-\dfrac{1}{\zeta\alpha\tau}\int_0^t\int_{\mathbb{T}^d}\int_{\mathbb{R}^d}\dfrac{1}{2}A_b(r) \nabla_r([f_\tau(s)]_\delta(1+\frac{\vert r\vert^2}{2})^\alpha)(1+\frac{\vert r\vert^2}{2})^{-\alpha}\cdot\nabla_r(T_M[f_\tau(s)]_\delta) dr dxds.
\end{align*}
By using \eqref{regularity-f}, it is possible  to let $\delta \to 0$  in the last equality. Moreover, notice that
\begin{align*}
    \nabla_r(T_Mf_\tau)=\nabla_r((1+\frac{\vert r\vert^2}{2})^\alpha f_\tau)T_M(1+\frac{\vert r\vert^2}{2})^{-\alpha}+\nabla_r(T_M(1+\frac{\vert r\vert^2}{2})^{-\alpha})(1+\frac{\vert r\vert^2}{2})^\alpha f_\tau
\end{align*}
thus
\begin{align*}
    \lim_{\delta\to 0} I_\delta^M&=-\dfrac{1}{\zeta\alpha\tau}\int_0^t\int_{\mathbb{T}^d}\int_{\mathbb{R}^d}\vert \nabla_r (f_\tau(s)(1+\frac{\vert r\vert^2}{2})^\alpha)\vert^2(1+\frac{\vert r\vert^2}{2})^{-\alpha}T_M(1+\frac{\vert r\vert^2}{2})^{-\alpha} dr dxds\\
    &-\dfrac{1}{\zeta\alpha\tau}\int_0^t\int_{\mathbb{T}^d}\int_{\mathbb{R}^d} (I+\dfrac{1}{2}A_b(r))\nabla_r (f_\tau(s)(1+\frac{\vert r\vert^2}{2})^\alpha)\cdot \nabla_r(T_M(1+\frac{\vert r\vert^2}{2})^{-\alpha})f_\tau(s) dr dxds\\
   &\hspace{-0.5cm}-\dfrac{1}{\zeta\alpha\tau}\int_0^t\int_{\mathbb{T}^d}\int_{\mathbb{R}^d}\dfrac{1}{2}A_b(r) \nabla_r(f_\tau(s)(1+\frac{\vert r\vert^2}{2})^\alpha)\cdot\nabla_r(f_\tau(s)(1+\frac{\vert r\vert^2}{2})^\alpha)  T_M(1+\frac{\vert r\vert^2}{2})^{-2\alpha}dr dxds
\end{align*}
Concerning the first and the third terms, we have 
\begin{align*}
    &J^M:=\dfrac{1}{\zeta\alpha\tau}\int_0^t\int_{\mathbb{T}^d}\int_{\mathbb{R}^d}\vert \nabla_r (f_\tau(s)(1+\frac{\vert r\vert^2}{2})^\alpha)\vert^2(1+\frac{\vert r\vert^2}{2})^{-\alpha}T_M(1+\frac{\vert r\vert^2}{2})^{-\alpha} dr dxds\\
   &+\dfrac{1}{2\zeta\alpha\tau}\int_0^t\int_{\mathbb{T}^d}\int_{\mathbb{R}^d}A_b(r) \nabla_r(f_\tau(s)(1+\frac{\vert r\vert^2}{2})^\alpha)\cdot\nabla_r(f_\tau(s)(1+\frac{\vert r\vert^2}{2})^\alpha)  T_M(1+\frac{\vert r\vert^2}{2})^{-2\alpha}dr dxds\\
   &\geq \dfrac{1}{\zeta\alpha\tau}\int_0^t\int_{\mathbb{T}^d}\int_{\mathbb{R}^d}\vert \nabla_r (f_\tau(s)(1+\frac{\vert r\vert^2}{2})^\alpha)\vert^2(1+\frac{\vert r\vert^2}{2})^{1-\alpha}T_M(1+\frac{\vert r\vert^2}{2})^{-\alpha} dr dxds,
\end{align*}
hence, by using Fatou's lemma (see \textit{e.g.} \cite[Lemma 4.1]{Brezis}) we get 
\begin{align*}
    \liminf_{M\to +\infty}J^M
   &\geq \dfrac{1}{\zeta\alpha\tau}\int_0^t\int_{\mathbb{T}^d}\int_{\mathbb{R}^d}\vert \nabla_r (f_\tau(s)(1+\frac{\vert r\vert^2}{2})^\alpha)\vert^2(1+\frac{\vert r\vert^2}{2})^{1-\alpha} dr dxds.
\end{align*}
Finally, we use \eqref{estim-beta-1} to show that $\displaystyle\lim_{M\to +\infty} \vert R^M\vert=0$ where
\begin{align*}
    R^M=
    &-\dfrac{1}{\zeta\alpha\tau}\int_0^t\int_{\mathbb{T}^d}\int_{\mathbb{R}^d} (I+\dfrac{1}{2}A_b(r))\nabla_r (f_\tau(s)(1+\frac{\vert r\vert^2}{2})^\alpha)\cdot \nabla_r(T_M(1+\frac{\vert r\vert^2}{2})^{-\alpha})f_\tau(s) dr dxds.
\end{align*}
Notice that $T_M(r)(1+\frac{\vert r\vert^2}{2})^{-\alpha}=\widetilde{T}_M(r)\circ (1+\frac{\vert r\vert^2}{2})^{\alpha},
$ where $\widetilde{T}_M(s)=\dfrac{T_M(s)}{s}, s>0$ thus
\begin{align}\label{eqn-cut-vanish}
    \nabla_r(T_M(r)(1+\frac{\vert r\vert^2}{2})^{-\alpha})= -\alpha rM(1+\frac{\vert r\vert^2}{2})^{-\alpha-1}1_{ \{  (1+\frac{\vert r\vert^2}{2})^{\alpha} >M\}}(r)
\end{align}
in a weak sense. On the other hand, we have
\begin{align*}
    \vert -\alpha rM(1+\frac{\vert r\vert^2}{2})^{-\alpha-1}1_{ \{  (1+\frac{\vert r\vert^2}{2})^{\alpha} >M\}}(r)\vert &\leq 2\alpha M(1+\frac{\vert r\vert^2}{2})^{-\alpha-\frac{1}{2}}1_{ \{  (1+\frac{\vert r\vert^2}{2})^{\alpha} >M\}}(r)\\&\leq  2\alpha (1+\frac{\vert r\vert^2}{2})^{-\frac{1}{2}}1_{ \{  (1+\frac{\vert r\vert^2}{2})^{\alpha} >M\}}(r)
\end{align*}
and consequently
\begin{align*}
    \vert R^M\vert \leq
    &\dfrac{2}{\zeta\tau}\int_0^t\int_{\mathbb{T}^d}\int_{\mathbb{R}^d} \hspace{-0.3cm}\vert (I+\dfrac{1}{2}A_b(r))\nabla_r (f_\tau(s)(1+\frac{\vert r\vert^2}{2})^\alpha)\vert f_\tau(s)   (1+\frac{\vert r\vert^2}{2})^{-\frac{1}{2}}1_{ \{ (1+\frac{\vert r\vert^2}{2})^{\alpha} >M\}}(r) dr dxds.
\end{align*}

Since 
\begin{align*}
\lim_{M\to +\infty}&\vert (I+\dfrac{1}{2}A_b(r))\nabla_r (f_\tau(s)(1+\frac{\vert r\vert^2}{2})^\alpha)\vert f_\tau(s)   (1+\frac{\vert r\vert^2}{2})^{-\frac{1}{2}}1_{ \{ (1+\frac{\vert r\vert^2}{2})^{\alpha} >M\}}(r)\\&=0  \quad  \text{ a.e. }(x,r,s)\in \mathbb{T}^d\times \mathbb{R}^d\times (0,T).    
\end{align*}
By using \eqref{estim-beta-1}, we obtain the following dominated function
\begin{align*}
 \mathbf{y}=   \vert (I+\dfrac{1}{2}A_b(r))\nabla_r (f_\tau(s)(1+\frac{\vert r\vert^2}{2})^\alpha)\vert f_\tau(s)(1+\frac{\vert r\vert^2}{2})^{-\frac{1}{2}}\in  L^1(\mathbb{R}^d\times \mathbb{T}^d\times (0,T)),
\end{align*}
Lebesgue dominated convergence theorem (see \textit{e.g.} \cite[Theorem 4.2]{Brezis}) ensures that  $\displaystyle\lim_{M\to +\infty} \vert R^M\vert=0.$ Thus
 \begin{align*}
    &\int_{\mathbb{T}^d}\int_{\mathbb{R}^d}f_\tau(t)^2(1+\frac{\vert r\vert^2}{2})^\alpha  drdx
    +\dfrac{1}{\zeta\alpha\tau}\int_0^t\int_{\mathbb{T}^d}\int_{\mathbb{R}^d}\vert \nabla_r (f_\tau(s)(1+\frac{\vert r\vert^2}{2})^\alpha)\vert^2(1+\frac{\vert r\vert^2}{2})^{1-\alpha} dr dxds\\
&\leq  \int_{\mathbb{T}^d}\int_{\mathbb{R}^d}(f_0^\tau)^2(1+\frac{\vert r\vert^2}{2})^\alpha  drdx+2\alpha\Vert \nabla_xu_L \Vert_\infty \int_0^t\int_{\mathbb{T}^d}\int_{\mathbb{R}^d} f_\tau^2(s)(1+\frac{\vert r\vert^2}{2})^{\alpha} drdxds.
\end{align*}
Set $\mathbf{K}=2\alpha\Vert \nabla_xu_L \Vert_\infty$ and apply Gronwall inequality to deduce 
\begin{align*}
    &\sup_{t\in [0, T]}\int_{\mathbb{T}^d}\int_{\mathbb{R}^d}f_\tau(t)^2(1+\frac{\vert r\vert^2}{2})^\alpha  drdx
    \\&+\dfrac{1}{\zeta\alpha\tau}\int_0^T\int_{\mathbb{T}^d}\int_{\mathbb{R}^d}\vert \nabla_r (f_\tau(s)(1+\frac{\vert r\vert^2}{2})^\alpha)\vert^2(1+\frac{\vert r\vert^2}{2})^{1-\alpha} dr dxds\leq  e^{\mathbf{K}T}\int_{\mathbb{T}^d}\int_{\mathbb{R}^d}(f_0^\tau)^2(1+\frac{\vert r\vert^2}{2})^\alpha  drdx.
\end{align*}
    \end{proof}
     \subsection{The continuity equation}
     Let $\alpha>\dfrac{d}{2}$ and  $\tau>0$, after integration "formally" with respect to the variable $r$ in  \eqref{Limit-FP-eqn-2},  set  $$\rho_\tau=\rho_\tau(t,x):=\int_{\mathbb{R}^d} f_\tau(t,x,r)dr$$ and 
      consider the following continuity equation 
    \begin{align}\label{continuity-eqn}
\begin{cases}&\partial_t \rho_\tau+u_L\cdot \nabla_x \rho_\tau=0\\
&\rho_\tau(x)|_{t=0}=\rho_0^\tau=\int_{\mathbb{R}^d}f_0^\tau(x,r)dr.
\end{cases}\end{align}
 Under the assumption $H_1$, we have  $(\rho_0^\tau)_\tau$ is bounded in $L^2(\mathbb{T}^d).$
 \begin{proposition}\label{prop-exit-cont}
 Assume that  $H_1$ holds,    there exists a unique  solution $\rho_\tau\in C([0,T];L^2(\mathbb{T}^d))$ to  \eqref{continuity-eqn} in the following sense
    \begin{align}\label{sense-cont}
        \int_{\mathbb{T}^d} \rho_\tau(t)\varphi dx-      \int_0^t\int_{\mathbb{T}^d} \rho_\tau(s)u_L\cdot \nabla\varphi dxds=\int_{\mathbb{T}^d}\rho_0^\tau\varphi dx, \qquad \forall \varphi \in H^1(\mathbb{T}^d).
    \end{align}
    Moreover, $(\rho_\tau)_\tau$ is bounded in $L^\infty(0,T;L^2(\mathbb{T}^d))$.
 \end{proposition}
 \begin{proof}
    The existence follows by a classical arguments \textit{e.g.} use the Galerkin approximation and prove that the approximation sequence is a Cauchy sequence in $C([0,T];L^2(\mathbb{T}^d))$(see also \cite[Corollaries II.1 and II.2]{diperna1989ordinary} for a different proof). The uniqueness  and the boundedness are consequence of   \autoref{Lemma-uniq-cont} below.
 \end{proof}

 \begin{lemma}(Stability)\label{Lemma-uniq-cont}
     Let  $\rho_1$ and $\rho_2$ be two solution to \eqref{continuity-eqn} in the sense \eqref{sense-cont}, with initial conditions $\rho_0^1$ and $\rho_0^2$ respectively. Then
       \begin{align}\label{stability-cont}
     \sup_{t\in [0,T]}   \int_{\mathbb{T}^d} \vert (\rho_1-\rho_2)(t)\vert^2 dx\leq \int_{\mathbb{T}^d}\vert\rho_0^1-\rho_0^2\vert^2 dx.
    \end{align}
 \end{lemma}
 \begin{proof}
 Let $\varphi \in H^1(\mathbb{T}^d)$ and $t\in [0,T]$.
     Note that $\rho_1-\rho_2$ satisfies 
    \begin{align*}
        \int_{\mathbb{T}^d} (\rho_1-\rho_2)(t)\varphi dx-      \int_0^t\int_{\mathbb{T}^d} (\rho_1-\rho_2)(s)u_L\cdot \nabla\varphi dxds=\int_{\mathbb{T}^d}(\rho_0^1-\rho_0^2)\varphi dx. 
    \end{align*}
Recall that  $\varphi_{\delta}:=\Theta_{\delta}*\varphi$\footnote{Recall that $\Theta_{\delta}$ is  a symmetric regularization kernel.}   is  an  appropriate test    function    in  last equation, hence
\begin{align*}
        \int_{\mathbb{T}^d} [(\rho_1-\rho_2)]_\delta(t)\varphi dx+      \int_0^t\int_{\mathbb{T}^d} [u_L\cdot \nabla (\rho_1-\rho_2)(s) ]_\delta  \varphi dxds=\int_{\mathbb{T}^d}[(\rho_0^1-\rho_0^2)]_\delta\varphi dx.  \end{align*}
    Set 
  $\varphi:=\Theta_\delta*(\rho_1-\rho_2)=[\rho_1-\rho_2]_\delta,$ which  is  an  appropriate test    function in the last equality. We get
  \begin{align}\label{cont-reg-tested}
        \int_{\mathbb{T}^d} [(\rho_1-\rho_2)]_\delta^2(t)dx+      \int_0^t\int_{\mathbb{T}^d} [u_L\cdot \nabla (\rho_1-\rho_2)(s) ]_\delta  [\rho_1-\rho_2]_\delta dxds=\int_{\mathbb{T}^d}[(\rho_0^1-\rho_0^2)]_\delta^2 dx  
  \end{align}
  Now, we are in position to use commutator estimates (\autoref{Lemma-tech-1}, Appendix C). Namely, notice that
  \begin{align*}
      [u_L\cdot \nabla (\rho_1-\rho_2)]_\delta [\rho_1-\rho_2]_\delta&=      \big( [u_L\cdot \nabla (\rho_1-\rho_2)]_\delta-u_L\cdot \nabla [(\rho_1-\rho_2)]_\delta\big) [(\rho_1-\rho_2)(s)]_\delta\\&\qquad+      u_L\cdot \nabla [(\rho_1-\rho_2)]_\delta [\rho_1-\rho_2]_\delta.
  \end{align*}
  Recall that $\rho_1-\rho_2\in C([0,T];L^2(\mathbb{T}^d))$ and by applying \autoref{Lemma-tech-1}, we get 
  \begin{align*}
       \big( [u_L\cdot \nabla (\rho_1-\rho_2)]_\delta-u_L\cdot \nabla [(\rho_1-\rho_2)]_\delta\big) [\rho_1-\rho_2]_\delta \to 0 \text{ in } L^2(\T^d\times (0,T)) \text{ as } \delta \to 0.
  \end{align*}
  On the other hand,  we have
  \begin{align*}
  \int_0^t\int_{\mathbb{T}^d}    u_L\cdot \nabla [(\rho_1-\rho_2)]_\delta [\rho_1-\rho_2]_\delta dxds=\dfrac{1}{2}\int_0^t\int_{\mathbb{T}^d} \Div\big( u_L [(\rho_1-\rho_2)]_\delta^2\big) dxds=0.
  \end{align*}
  Now, we pass to the limit in \eqref{cont-reg-tested} as $\delta \to 0$ to conclude the proof.
 \end{proof}
\begin{lemma}\label{lem-cont-limit} Assume that $H_1$ and $H_2$ hold. Then, 
there exists a unique solution $\rho$ in the sense of \eqref{sense-cont}
  to    \begin{align}\label{continuity-eqn-limit}
\begin{cases}&\partial_t \rho+u_L\cdot \nabla_x \rho=0\\
&\rho(x)|_{t=0}=\rho_0=\int_{\mathbb{R}^d}f_0(x,r)dr,
\end{cases}\end{align}
    and the  sequence $(\rho_\tau)_\tau$ converges  to $\rho$ in $C([0,T],L^2(\mathbb{T}^d))$  as $\tau \to 0$.
\end{lemma}
\begin{proof}
Recall that $\rho_0=\int_{\mathbb{R}^d}f_0(x,r)dr\in L^2(\T^d)$,
the existence and uniqueness follow like  \autoref{prop-exit-cont}. On the other hand,    we get $\Vert \rho_0^\tau-\rho_0\Vert_{L^2} \to 0.$ Indeed, notice that
\begin{align*}
    \Vert \rho_0^\tau-\rho_0\Vert_{L^2}^2&=\int_{\T^d}\big(\int_{\R^d} (f_0^\tau(x,r)-f_0(x,r) dr\big )^2dx\leq  \int_{\T^d} (\int_{\R^d} \vert f_0^\tau(x,r)-f_0(x,r)\vert dr)^2dx\\
    &\leq  \int_{\T^d} (\int_{\R^d} \vert f_0^\tau(x,r)-f_0(x,r)\vert dr)^2dx\\
    &\leq   \int_{\R^d}(1+\frac{\vert r\vert^2}{2})^{-\alpha} dr\int_{\T^d} \int_{\R^d} \vert f_0^\tau(x,r)-f_0(x,r)\vert^2(1+\frac{\vert r\vert^2}{2})^\alpha drdx \to 0 \text{ as } \delta \to 0,
\end{align*}
thanks to $H_2$ and  $\int_{\R^d}(1+\frac{\vert r\vert^2}{2})^{-\alpha} dr <+\infty$ since $\alpha> \dfrac{d}{2}.$
Thus, the convergence is a consequence of \autoref{Lemma-uniq-cont} by setting $\rho_1=\rho_\tau$ and $\rho_2=\rho$.
    \end{proof}

    \section{Convergence result as $\tau\to 0$ }\label{section-cv}
    \subsection{The case $\alpha>\frac{d}{2}$.}
    Let $\alpha>\frac{d}{2}$,  we prove the following result.
            \begin{proposition}\label{prop-cv-1}
                Let $(f_\tau)_\tau$ be  the unique
quasi-regular weak solution to \eqref{equ-final-I} for any $\tau>0$. Assume that assumptions $H_1$ and $H_2$ are satisfied then $f_\tau$ converges to $\rho \otimes p_\alpha$ in $L^2(0,T;H_\alpha)$ as $\tau \to 0$. More precisely, there exists $\mathbf{C}_\alpha>0$( independent of $\tau$) such that the following inequality holds
\begin{align*}
    &  \int_0^T\hspace{-0.2cm}\int_{\mathbb{T}^d}\hspace{-0.1cm}\int_{\mathbb{R}^d}(f_\tau-\rho \otimes p_\alpha)^2(s)(1+\frac{\vert r\vert^2}{2})^\alpha dr dxds
\leq \mathbf{C}_\alpha\tau+\dfrac{1}{Z}\int_0^T\hspace{-0.2cm}\int_{\mathbb{T}^d} \vert \rho_\tau-\rho\vert^2 dxds,\notag
\end{align*}
where $\rho$ is the unique solution to \eqref{continuity-eqn-limit}  and $p_\alpha(r)=\dfrac{1}{Z}(1+\frac{\vert r\vert^2}{2})^{-\alpha}$  is the Cauchy probability density on $\mathbb{R}^d$. 
            \end{proposition} 

   \subsection*{Proof of \autoref{prop-cv-1}}
   Let $t\in [0,T]$ and    set $f:=f(t,x,r)=\rho(t,x) \otimes p_\alpha(r)$, define $g_\tau=f_\tau-f$ and $g_0^\tau=f_0^\tau-\rho_0\otimes p_\alpha$. By using \eqref{ground-state} and \eqref{continuity-eqn-limit} we get
 \begin{align}&\int_{\mathbb{T}^d}\int_{\mathbb{R}^d}g_\tau(t)\phi drdx-\int_{\mathbb{T}^d}\int_{\mathbb{R}^d}g_0^\tau\phi drdx\label{equ-beta-I}\\&= \int_0^t\int_{\mathbb{T}^d}\int_{\mathbb{R}^d}g_\tau(s)u_L(s)\cdot \nabla_x \phi dr dxds+ \int_0^t\int_{\mathbb{T}^d}\int_{\mathbb{R}^d} f_\tau(s)(\nabla u_L(s)r) \cdot\nabla_r\phi dr dxds\notag\\	&-	\dfrac{1}{\zeta\alpha\tau}\int_0^t\int_{\mathbb{T}^d}\int_{\mathbb{R}^d} \alpha g_\tau(s)r\cdot\nabla_r\phi+\nabla_rg_\tau(s) \cdot\nabla_r\phi  +\dfrac{1}{2}A_b(r) \nabla_rg_\tau(s)\cdot\nabla_r\phi dr dxds,  \forall \phi \in \mathcal{Y}.\notag\end{align}
 
 Let     $\phi\in \mathcal{Y}$ and $\delta>0$, arguments already detailed ensure
	\begin{align}&\int_{\mathbb{T}^d}\int_{\mathbb{R}^d}[g_\tau(t)]_\delta\phi drdx-\int_{\mathbb{T}^d}\int_{\mathbb{R}^d}[g_0^\tau]_\delta\phi drdx\notag\\&= -\int_0^t\int_{\mathbb{T}^d}\int_{\mathbb{R}^d}[u_L(s)\cdot \nabla_xg_\tau(s)]_\delta \phi+ [(\nabla u_L(s)r) \cdot\nabla_rf_\tau(s)]_\delta\phi dr dxds\notag\\	&+	\dfrac{1}{\zeta\alpha\tau}\int_0^t\int_{\mathbb{T}^d}\int_{\mathbb{R}^d} \alpha [\Div_r( g_\tau(s)r)]_\delta\phi dr dxds\label{eqn-cv-delta}\\&+\dfrac{1}{\zeta\alpha\tau}\int_0^t\langle[\Delta_rg_\tau(s) ]_\delta,\phi  \rangle+\dfrac{1}{2}\langle[\Div_r(A_b(r) \nabla_rg_\tau(s))]_\delta,\phi \rangle ds,\notag\end{align}
    where $\langle \cdot,\cdot\rangle$ denotes the duality pairing between $\mathcal{Y}^\prime$ and $\mathcal{Y}.$\\ 
   Consider $\phi_{\delta,M}(x,r)= T_M(r)[g_\tau(\cdot)]_\delta(x,r) \in \mathcal{Y},$ which  is an appropriate test function in \eqref{eqn-cv-delta}. Thus,  we obtain
   \begin{align}&\int_{\mathbb{T}^d}\int_{\mathbb{R}^d}[g_\tau(t)]_\delta^2 T_M drdx-\int_{\mathbb{T}^d}\int_{\mathbb{R}^d}[g_0^\tau]_\delta^2 T_M drdx\notag\\&= -\int_0^t\int_{\mathbb{T}^d}\int_{\mathbb{R}^d}[u_L(s)\cdot \nabla_xg_\tau(s)]_\delta  T_M[g_\tau(s)]_\delta+ [(\nabla u_L(s)r) \cdot\nabla_rf_\tau(s)]_\delta T_M[g_\tau(s)]_\delta dr dxds\notag\\	&+	\dfrac{1}{\zeta\alpha\tau}\int_0^t\int_{\mathbb{T}^d}\int_{\mathbb{R}^d} \alpha [\Div_r( g_\tau(s)r)]_\delta T_M[g_\tau(s)]_\delta dr dxds+\dfrac{1}{\zeta\alpha\tau}\int_0^t\langle [\Delta_rg_\tau(s) ]_\delta, T_M[g_\tau(s)]_\delta \rangle ds\notag\\ & +\dfrac{1}{\zeta\alpha\tau}\int_0^t\langle\dfrac{1}{2}[\Div_r(A_b(r) \nabla_rg_\tau(s))]_\delta, T_M[g_\tau(s)]_\delta \rangle ds\notag\end{align}
   Now,  let   us    pass    to  the limit  as $\delta \to 0$ then $M\to +\infty$  in  the last   equality.
   \begin{itemize}
       \item A standard  argument (see \textit{e.g.} the proof of \autoref{lemma-appendix}) ensures
    \begin{align*}
          \liminf_{M\to +\infty} \int_{\mathbb{T}^d}\int_{\mathbb{R}^d}g_\tau(t)^2 T_M drdx \leq \liminf_{M\to +\infty}\liminf_{\delta\to 0} \int_{\mathbb{T}^d}\int_{\mathbb{R}^d}[g_\tau(t)]_\delta^2 T_M drdx,
       \end{align*}
     and   by monotone convergence theorem (see \textit{e.g.} \cite[Theorem 4.1]{Brezis}), we deduce 
        \begin{align*}
          \liminf_{M\to +\infty} \int_{\mathbb{T}^d}\int_{\mathbb{R}^d}g_\tau(t)^2 T_M drdx =\lim_{M\to +\infty} \int_{\mathbb{T}^d}\int_{\mathbb{R}^d}g_\tau(t)^2 T_M drdx=\int_{\mathbb{T}^d}\int_{\mathbb{R}^d}g_\tau(t)^2(1+\frac{\vert r\vert^2}{2})^\alpha  drdx.
       \end{align*}
       \item  Since $g_0^\tau\in H_\alpha$, we get $\displaystyle\lim_{\delta\to 0}\int_{\mathbb{T}^d}\int_{\mathbb{R}^d}[g_0^\tau]_\delta^2 T_M drdx=\int_{\mathbb{T}^d}\int_{\mathbb{R}^d}(g_0^\tau)^2 T_M drdx,$ again monotone convergence theorem ensures
       \begin{align*}
         \lim_{M\to +\infty}  \int_{\mathbb{T}^d}\int_{\mathbb{R}^d}(g_0^\tau)^2 T_M drdx=\int_{\mathbb{T}^d}\int_{\mathbb{R}^d}(g_0^\tau)^2(1+\frac{\vert r\vert^2}{2})^\alpha  drdx.
       \end{align*}
       \item $\displaystyle\lim_{\delta\to 0}\int_0^t\int_{\mathbb{T}^d}\int_{\mathbb{R}^d}[u_L(s)\cdot \nabla_xg_\tau(s)]_\delta  T_M[g_\tau(s)]_\delta drdxds=0.$ Indeed,  Note that
       \begin{align*}
          &[u_L(s)\cdot \nabla_xg_\tau(s)]_\delta  T_M[g_\tau(s)]_\delta\\&=\big([u_L(s)\cdot \nabla_xg_\tau(s)]_\delta-u_L(s)\cdot \nabla_x[g_\tau(s)]_\delta\big)  T_M[g_\tau(s)]_\delta+u_L(s)\cdot \nabla_x[g_\tau(s)]_\delta  T_M[g_\tau(s)]_\delta.
       \end{align*}
       On the one hand, since $T_M$ depends only on $r$, we have 
       \begin{align*}
           \int_0^t\int_{\mathbb{T}^d}\int_{\mathbb{R}^d}  u_L(s)\cdot \nabla_x[g_\tau(s)]_\delta  T_M[g_\tau(s)]_\delta drdxds&=           \int_0^t\int_{\mathbb{R}^d} T_M \int_{\mathbb{T}^d} u_L(s)\cdot \nabla_x[g_\tau(s)]_\delta  [g_\tau(s)]_\delta dxdrds\\
           &=\dfrac{1}{2}\int_0^t\int_{\mathbb{R}^d} T_M \int_{\mathbb{T}^d} \Div[u_L(s)[g_\tau(s)]_\delta^2] dxdrds=0.
       \end{align*}
       On the other hand,  since $\vert T_M\vert \leq M$ and by using Cauchy–Schwarz inequality we have
       \begin{align*}
           &\vert\int_0^t\int_{\mathbb{T}^d}\int_{\mathbb{R}^d} \big([u_L(s)\cdot \nabla_xg_\tau(s)]_\delta-u_L(s)\cdot \nabla_x[g_\tau(s)]_\delta\big)  T_M[g_\tau(s)]_\delta drdxds \vert\\
           & \leq M\Vert  [u_L(s)\cdot \nabla_xg_\tau(s)]_\delta-u_L(s)\cdot \nabla_x[g_\tau(s)]_\delta \Vert_{L^{2}(\T^d\times\R^d\times (0,T))}\Vert  [g_\tau(s)]_\delta\Vert_{L^{2}(\T^d\times\R^d\times (0,T))}\\
           &\leq M\Vert  [u_L\cdot \nabla_xg_\tau]_\delta-u_L\cdot \nabla_x[g_\tau]_\delta \Vert_{L^{2}(\T^d\times\R^d\times (0,T))}\Vert  g_\tau\Vert_{L^{2}(\T^d\times\R^d\times (0,T))} \to 0,
       \end{align*}
       by using that $g_\tau\in L^2(0,T;H)$ and the commutator estimate, see \autoref{Lemma-tech-1}.  
       
          \item  The term $\int_0^t\int_{\mathbb{T}^d}\int_{\mathbb{R}^d} [(\nabla u_L(s)r) \cdot\nabla_rf_\tau(s)]_\delta T_M[g_\tau(s)]_\delta dr dxds$. This term requires special attention, by using  \eqref{regularity-f} and the properties of convolution one gets 
\begin{align*}
  &\lim_{\delta\to 0}  \int_0^t\int_{\mathbb{T}^d}\int_{\mathbb{R}^d} [(\nabla u_L(s)r) \cdot\nabla_rf_\tau(s)]_\delta T_M[g_\tau(s)]_\delta dr dxds\\&=\int_0^t\int_{\mathbb{T}^d}\int_{\mathbb{R}^d} (\nabla u_L(s)r) \cdot\nabla_rf_\tau(s) T_Mg_\tau(s) dr dxds.
\end{align*}
Notice that 
\begin{align*}
   &\int_0^t\int_{\mathbb{T}^d}\int_{\mathbb{R}^d} (\nabla u_L(s)r) \cdot\nabla_rf_\tau(s) T_Mg_\tau(s) dr dxds=\int_0^t\int_{\mathbb{T}^d}\int_{\mathbb{R}^d} (\nabla u_L(s)r) \cdot\nabla_rf_\tau(s) T_Mf_\tau(s) dr dxds\\&\qquad-\int_0^t\int_{\mathbb{T}^d}\int_{\mathbb{R}^d} (\nabla u_L(s)r) \cdot\nabla_rf_\tau(s) T_M\rho(s)\otimes p_\alpha dr dxds.
\end{align*}
Concerning the second term, we have
\begin{align} \vert (\nabla u_L(s)r) \cdot\nabla_rf_\tau(s) T_M\rho(s)\otimes p_\alpha \vert \leq  \dfrac{1}{Z}\Vert \nabla u_L\Vert_\infty \vert r\vert  \vert \nabla_rf_\tau(s)\vert  \vert\rho(s)\vert.
\end{align}
Thanks to \eqref{estim-beta-1}, we get  $\Vert \nabla u_L\Vert_\infty \vert r\vert  \vert \nabla_rf_\tau(s)\vert  \vert\rho(s)\vert\in L^1(\mathbb{R}^d\times \mathbb{T}^d\times (0,T))$ and Lebesgue's dominated convergence theorem (see \textit{e.g.} \cite[Theorem 4.2]{Brezis}) ensures
\begin{align*}
    &\lim_{M\to +\infty}\int_0^t\int_{\mathbb{T}^d}\int_{\mathbb{R}^d} (\nabla u_L(s)r) \cdot\nabla_rf_\tau(s) T_M\rho(s)\otimes p_\alpha dr dxds\\&=\dfrac{1}{Z}\int_0^t\int_{\mathbb{T}^d} \rho(s)\int_{\mathbb{R}^d} (\nabla u_L(s)r) \cdot\nabla_rf_\tau(s)  dr dxds\\
    &=\int_0^t\int_{\mathbb{T}^d} \rho(s)\int_{\mathbb{R}^d} \Div_r\big(\nabla u_L(s)r f_\tau(s)\big)  dr dxds=0.
\end{align*}
On the other hand, concerning the first term we have
\begin{align*}
    &\int_0^t\int_{\mathbb{T}^d}\int_{\mathbb{R}^d} (\nabla u_L(s)r) \cdot\nabla_rf_\tau(s) T_Mf_\tau(s) dr dxds\\&=    \int_0^t\int_{\mathbb{T}^d}\int_{\mathbb{R}^d} (\nabla u_L(s)r) \cdot\nabla_rf_\tau(s) T_M(1+\frac{\vert r\vert^2}{2})^{-\alpha} f_\tau(s)(1+\frac{\vert r\vert^2}{2})^\alpha dr dxds\\
    &=-\int_0^t\int_{\mathbb{T}^d}\int_{\mathbb{R}^d} f_\tau^2(s)(\nabla u_L(s)r)  \cdot\nabla_r(T_M(1+\frac{\vert r\vert^2}{2})^{-\alpha}) (1+\frac{\vert r\vert^2}{2})^\alpha dr dxds\\
    &-\int_0^t\int_{\mathbb{T}^d}\int_{\mathbb{R}^d}  f_\tau(s) T_M(1+\frac{\vert r\vert^2}{2})^{-\alpha} (\nabla u_L(s)r)\cdot\nabla_r(f_\tau(s)(1+\frac{\vert r\vert^2}{2})^\alpha) dr dxds\\
    &=A_1^M+A_2^M.
\end{align*}
We have $\displaystyle\lim_{M\to +\infty}\vert A_1^M\vert=0.$ Indeed, by using \eqref{eqn-cut-vanish} we get
\begin{align*}
    \vert A_1^M \vert &\leq  \alpha\Vert \nabla u_L\Vert_\infty \int_0^t\int_{\mathbb{T}^d}\int_{\mathbb{R}^d} f_\tau^2(s)
    \vert r \vert^2  M(1+\frac{\vert r\vert^2}{2})^{-\alpha-1}1_{ \{  (1+\frac{\vert r\vert^2}{2})^{\alpha} >M\}}(r) (1+\frac{\vert r\vert^2}{2})^\alpha dr dxds\\
    &\leq   \alpha\Vert \nabla u_L\Vert_\infty \int_0^t\int_{\mathbb{T}^d}\int_{\mathbb{R}^d} f_\tau^2(s)(1+\frac{\vert r\vert^2}{2})^\alpha 1_{ \{  (1+\frac{\vert r\vert^2}{2})^{\alpha} >M\}}(r)  dr dxds.
\end{align*}
\begin{itemize}
    \item[i.]  Since    $0\leq 1_{ \{  (1+\frac{\vert r\vert^2}{2})^{\alpha} >M\}}(r) \leq 1$ for any $r\in \mathbb{R}^d$, we have   \begin{align*}
    \vert \mathbf{f}_M(s,r)\vert :=      f_\tau^2(s)(1+\frac{\vert r\vert^2}{2})^\alpha 1_{ \{  (1+\frac{\vert r\vert^2}{2})^{\alpha} >M\}}(r) \leq         f_\tau^2(s)(1+\frac{\vert r\vert^2}{2})^\alpha:=\mathbf{f}(s,r) ,\quad \forall M\in \mathbb{N}. 
      \end{align*}
for $ (s,r)\in [0,T]\times \mathbb{R}^d$.      Thanks to \autoref{propo-estimates}, we have $\mathbf{f}\in L^1(\mathbb{T}^d\times \mathbb{R}^d\times (0,T)),$ namely 
       $$\displaystyle\int_0^T\int_{\mathbb{T}^d}\int_{\mathbb{R}^d} \mathbf{f}dr dxds= \displaystyle\int_0^T\int_{\mathbb{T}^d}\int_{\mathbb{R}^d}f_\tau^2(s)(1+\frac{\vert r\vert^2}{2})^\alpha   dr dxds<+\infty.$$  
     \item[ii.]  Since  $  \displaystyle\lim_{M\to +\infty} 1_{ \{  (1+\frac{\vert r\vert^2}{2})^{\alpha} >M\}}(r)=0 $  for every $r\in \mathbb{R}^d$ and  $f_\tau^2(s)(1+\frac{\vert r\vert^2}{2})^\alpha$ is independent of $M$ almost everywhere $(x,r,s)\in \mathbb{T}^d\times \mathbb{R}^d\times (0,T)$, it follows that 
       \begin{align*}
     \lim_{M\to +\infty}      f_\tau^2(s)(1+\frac{\vert r\vert^2}{2})^\alpha 1_{ \{  (1+\frac{\vert r\vert^2}{2})^{\alpha} >M\}}(r)  dr dxds=0 \text{ a.e. } (x,r,s)\in \mathbb{T}^d\times \mathbb{R}^d\times (0,T).
      \end{align*}
     \end{itemize}
      Therefore, Lebesgue's dominated convergence theorem (see \textit{e.g.} \cite[Theorem 4.2]{Brezis}) ensures 
      \begin{align*}
     \lim_{M\to +\infty}     \int_0^t\int_{\mathbb{T}^d}\int_{\mathbb{R}^d} f_\tau^2(s)(1+\frac{\vert r\vert^2}{2})^\alpha 1_{ \{  (1+\frac{\vert r\vert^2}{2})^{\alpha} >M\}}(r)  dr dxds=0
      \end{align*}
      and
      $\displaystyle\lim_{M\to +\infty}\vert A_1^M\vert=0.$\\
      
      Concerning $A_2^M$, since $\vert T_M(1+\frac{\vert r\vert^2}{2})^{-\alpha}\vert \leq 1$  we have
     \begin{align*}
      G_M:=   &\vert f_\tau(s) T_M(1+\frac{\vert r\vert^2}{2})^{-\alpha} (\nabla u_L(s)r)\cdot\nabla_r(f_\tau(s)(1+\frac{\vert r\vert^2}{2})^\alpha)  \vert \\& \quad\leq  \vert f_\tau(s)  (\nabla u_L(s)r)\cdot\nabla_r(f_\tau(s)(1+\frac{\vert r\vert^2}{2})^\alpha)\vert  =G.  \end{align*}
     By using \eqref{key-estimate-beta} and Young's inequality, we have $G\in L^1(\mathbb{T}^d\times \mathbb{R}^d\times (0,T)).$ Indeed,
     \begin{align*}
        \int_0^t\int_{\mathbb{T}^d}\int_{\mathbb{R}^d} \vert G\vert  dr dxds&\leq  \int_0^t\int_{\mathbb{T}^d}\int_{\mathbb{R}^d}  f_\tau^2(s)(1+\frac{\vert r\vert^2}{2})^{\alpha}dr dxds\\&+2\Vert \nabla u_L\Vert_\infty^2 \int_0^t\int_{\mathbb{T}^d}\int_{\mathbb{R}^d}\vert \nabla_r (f_\tau(s)(1+\frac{\vert r\vert^2}{2})^\alpha\vert^2(1+\frac{\vert r\vert^2}{2})^{1-\alpha} dr dxds<+\infty.
     \end{align*}
     On the other hand, notice that 
     \begin{align*}
          \lim_{M\to +\infty}&   f_\tau(s) T_M(1+\frac{\vert r\vert^2}{2})^{-\alpha} (\nabla u_L(s)r)\cdot\nabla_r(f_\tau(s)(1+\frac{\vert r\vert^2}{2})^\alpha) \notag\\&=   f_\tau(s)  (\nabla u_L(s)r)\cdot\nabla_r(f_\tau(s)(1+\frac{\vert r\vert^2}{2})^\alpha)  \text{ a.e. } (x,r,s)\in \mathbb{T}^d\times \mathbb{R}^d\times (0,T).
     \end{align*}
     Lebesgue's dominated convergence theorem (see \textit{e.g.} \cite[Theorem 4.2]{Brezis}) gives
      \begin{align}
     \lim_{M\to +\infty}&  \int_0^t\int_{\mathbb{T}^d}\int_{\mathbb{R}^d}  f_\tau(s) T_M(1+\frac{\vert r\vert^2}{2})^{-\alpha} (\nabla u_L(s)r)\cdot\nabla_r(f_\tau(s)(1+\frac{\vert r\vert^2}{2})^\alpha) dr dxds \notag\\&= \int_0^t\int_{\mathbb{T}^d}\int_{\mathbb{R}^d}  f_\tau(s)  (\nabla u_L(s)r)\cdot\nabla_r(f_\tau(s)(1+\frac{\vert r\vert^2}{2})^\alpha) dr dxds.\notag
      \end{align}
      By using   \eqref{key-estimate-beta}, we obtain
            \begin{align}
          &\vert \int_0^t\int_{\mathbb{T}^d}\int_{\mathbb{R}^d}  f_\tau(s)  (\nabla u_L(s)r)\cdot\nabla_r(f_\tau(s)(1+\frac{\vert r\vert^2}{2})^\alpha) dr dxds\vert \notag\\
          &\leq \dfrac{\sqrt{\tau}}{2} \int_0^t\int_{\mathbb{T}^d}\int_{\mathbb{R}^d}  f_\tau^2(s)(1+\frac{\vert r\vert^2}{2})^\alpha dr dxds\notag\\
          &+\dfrac{1}{\sqrt{\tau}}\Vert \nabla u_L\Vert_\infty^2 \int_0^t\int_{\mathbb{T}^d}\int_{\mathbb{R}^d}    \vert r\vert^2\vert\nabla_r(f_\tau(s)(1+\frac{\vert r\vert^2}{2})^\alpha)\vert^2(1+\frac{\vert r\vert^2}{2})^{-\alpha} dr dxds\notag\\
          &\leq  \overline{\mathbf{K}}(\dfrac{T}{2}+\Vert \nabla u_L\Vert_\infty^2\zeta\alpha)\sqrt{\tau}:= \widetilde{\mathbf{C}}\sqrt{\tau}.
      \end{align}
         \item   Concerning the other terms, we have
      \begin{align*}
   J_\delta^M&:= \dfrac{1}{\zeta\alpha\tau}\int_0^t\int_{\mathbb{T}^d}\int_{\mathbb{R}^d} \alpha [\Div_r( g_\tau(s)r)]_\delta T_M[g_\tau(s)]_\delta dr dxds+\dfrac{1}{\zeta\alpha\tau}\int_0^t\langle [\Delta_rg_\tau(s) ]_\delta, T_M[g_\tau(s)]_\delta \rangle ds\notag\\ & +\dfrac{1}{\zeta\alpha\tau}\int_0^t\langle\dfrac{1}{2}[\Div_r(A_b(r) \nabla_rg_\tau(s))]_\delta, T_M[g_\tau(s)]_\delta \rangle ds\\
   &=-\dfrac{1}{\zeta\alpha\tau}\int_0^t\int_{\mathbb{T}^d}\int_{\mathbb{R}^d} \big(\alpha r [g_\tau(s)]_\delta +\nabla_r [g_\tau(s)]_\delta+\dfrac{1}{2}A_b(r) \nabla_r[g_\tau(s)]_\delta\big)\cdot\nabla_r(T_M[g_\tau(s)]_\delta) dr dxds.
\end{align*}
Notice that the smallest eigenvalue of $A_b(r)$ is $|r|^2$ and
by repeating the arguments presented in \autoref{subsection-uniform-beta}, we infer
\begin{align*}
    \liminf_{M\to +\infty}J^M_\delta
   &\geq \dfrac{1}{\zeta\alpha\tau}\int_0^t\int_{\mathbb{T}^d}\int_{\mathbb{R}^d}\vert \nabla_r (g_\tau(s)(1+\frac{\vert r\vert^2}{2})^\alpha)\vert^2(1+\frac{\vert r\vert^2}{2})^{1-\alpha} dr dxds.
\end{align*}
 Summarizing, we get
 \begin{align*}
    &\int_{\mathbb{T}^d}\int_{\mathbb{R}^d}g_\tau(t)^2(1+\frac{\vert r\vert^2}{2})^\alpha  drdx
    +\dfrac{1}{\zeta\alpha\tau}\int_0^t\int_{\mathbb{T}^d}\int_{\mathbb{R}^d}\vert \nabla_r (g_\tau(s)(1+\frac{\vert r\vert^2}{2})^\alpha)\vert^2(1+\frac{\vert r\vert^2}{2})^{1-\alpha} dr dxds\\
&\leq  \int_{\mathbb{T}^d}\int_{\mathbb{R}^d}(g_0^\tau)^2(1+\frac{\vert r\vert^2}{2})^\alpha  drdx+\widetilde{\mathbf{C}}\sqrt{\tau}.
\end{align*}
Hence \begin{align*}
    \sup_{t\in [0, T]}\int_{\mathbb{T}^d}\int_{\mathbb{R}^d}g_\tau(t)^2(1+\frac{\vert r\vert^2}{2})^\alpha  drdx
    &+\dfrac{1}{\zeta\alpha\tau}\int_0^T\int_{\mathbb{T}^d}\int_{\mathbb{R}^d}\vert \nabla_r (g_\tau(s)(1+\frac{\vert r\vert^2}{2})^\alpha)\vert^2(1+\frac{\vert r\vert^2}{2})^{1-\alpha} dr dxds\\
&\leq  \int_{\mathbb{T}^d}\int_{\mathbb{R}^d}(g_0^\tau)^2(1+\frac{\vert r\vert^2}{2})^\alpha  drdx+\widetilde{\mathbf{C}}\sqrt{\tau}.
\end{align*}
In other words, we obtain
\begin{align}
    &\sup_{t\in [0, T]}\int_{\mathbb{T}^d}\int_{\mathbb{R}^d}\vert f_\tau(t)-\rho(t)\otimes p_\alpha\vert^2(1+\frac{\vert r\vert^2}{2})^\alpha  drdx\notag\\
    &+\dfrac{1}{\zeta\alpha\tau}\int_0^T\int_{\mathbb{T}^d}\int_{\mathbb{R}^d}\vert \nabla_r (f_\tau-\rho \otimes p_\alpha)(s)(1+\frac{\vert r\vert^2}{2})^\alpha)\vert^2(1+\frac{\vert r\vert^2}{2})^{-\alpha+1} dr dxds \label{cv-inequality}\\
&\leq  \int_{\mathbb{T}^d}\int_{\mathbb{R}^d}\vert f_0^\tau-\rho_0\otimes p_\alpha\vert ^2(1+\frac{\vert r\vert^2}{2})^\alpha  drdx+\widetilde{\mathbf{C}}\sqrt{\tau}.\notag
\end{align}
      \end{itemize}
  Let us use \eqref{cv-inequality} to deduce the convergence results. Indeed, 
     \eqref{cv-inequality} implies
  \begin{align}
    &  \int_0^T\int_{\mathbb{T}^d}\int_{\mathbb{R}^d}\vert \nabla_r (f_\tau-\rho\otimes p_\alpha)(s)(1+\frac{\vert r\vert^2}{2})^\alpha)\vert^2(1+\frac{\vert r\vert^2}{2})^{-\alpha+1} dr dxds\label{strong-cv-1}\\
&\leq \zeta\alpha\tau( \int_{\mathbb{T}^d}\int_{\mathbb{R}^d}\vert f_0^\tau-\rho_0\otimes p_\alpha\vert ^2(1+\frac{\vert r\vert^2}{2})^\alpha  drdx+\widetilde{\mathbf{C}}\sqrt{\tau})\leq \overline{\mathbf{C}}\tau \to 0 \text{ as } \tau \to 0.\notag
\end{align}
Recall that for $\alpha>\dfrac{d}{2},$ $\nu_\alpha(dr)=\dfrac{1}{Z}(1+\frac{\vert r\vert^2}{2})^{-\alpha} dr$ is a Borel probability measure on $\mathbb{R}^d,$
set $$h_\alpha=(f_\tau-\rho \otimes p_\alpha)(1+\frac{\vert r\vert^2}{2})^\alpha,$$
by using  the weighted Poincar\'e inequality ( see \autoref{proposition-Poincare-weight}, Appendix B), there exists $C_\alpha>0$ such that
\begin{align*}
    \int_{\mathbb{R}^d} h_\alpha^2 d\nu_\alpha - (\int_{\mathbb{R}^d} h_\alpha d\nu_\alpha)^2 \leq C_\alpha\int_{\mathbb{R}^d} \vert \nabla h_\alpha\vert^2(1+\frac{\vert r\vert^2}{2}) d\nu_\alpha.
\end{align*}
On the other hand, notice that 
\begin{align*}
    \int_{\mathbb{R}^d} h_\alpha d\nu_\alpha= \int_{\mathbb{R}^d}(f_\tau-\rho\otimes p_\alpha)(1+\frac{\vert r\vert^2}{2})^\alpha d\nu_\alpha=\dfrac{1}{Z}(\int_{\mathbb{R}^d}f_\tau dr-\rho)=\dfrac{1}{Z}(\rho_\tau-\rho).
\end{align*}
Gathering the last inequalities to get 
\begin{align*}
    &  \int_0^T\int_{\mathbb{T}^d}\int_{\mathbb{R}^d}(f_\tau-\rho \otimes p_\alpha)^2(s)(1+\frac{\vert r\vert^2}{2})^\alpha dr dxds
\leq C_\alpha\overline{\mathbf{C}}\tau+\dfrac{1}{Z}\int_0^T\int_{\mathbb{T}^d} \vert \rho_\tau-\rho\vert^2 dxds.\notag
\end{align*}
By using \autoref{lem-cont-limit}, we have $\rho_\tau\to\rho$  in $L^2(\mathbb{T}^d\times (0,T))$ as $\tau \to 0$ and   the result holds. $\square$
\begin{remark}
    Since $\alpha>\dfrac{d}{2}$,  the last inequality implies the convergence in $L^1(\mathbb{T}^d\times \mathbb{R}^d\times (0,T)).$
\end{remark}
     \subsection{The limit as $\tau \to 0$ if  $0<\alpha\leq \dfrac{d}{2}$}
        We will discuss the following two cases.
        \begin{itemize}
            \item[i.]  $f^\tau_0 \in H_\alpha \cap H$ and there exists $\mathbf{\Lambda}>0$ independent of $\tau$  such that $\displaystyle\sup_{\tau>0}\Vert f^\tau_0 \Vert_{H_\alpha}^2\leq  \mathbf{\Lambda}. $
            \item[ii.]  $f^\tau_0 \in L^1(\mathbb{T}^d\times \mathbb{R}^d) \cap H$ and there exists $\mathbf{\Lambda}>0$ independent of $\tau$  such that $\displaystyle\sup_{\tau>0}\Vert f^\tau_0 \Vert_{L^1_{x,r}}\leq  \mathbf{\Lambda}. $
        \end{itemize}
        \subsubsection{The case $(f^\tau_0)_\tau$ is bounded in $H_\alpha$}
        Following
\autoref{propo-estimates},
there exists $\mathbf{K}=2\alpha\Vert \nabla_xu_L \Vert_\infty$  such that
\begin{align}
    \sup_{t\in [0, T]} \Vert f_\tau(t)\Vert_{H_\alpha}^2
    &+\dfrac{1}{\zeta\alpha\tau}\int_0^T\int_{\mathbb{T}^d}\int_{\mathbb{R}^d}\vert \nabla_r (f_\tau(s)P_{-\alpha})\vert^2(1+\frac{\vert r\vert^2}{2})^{1-\alpha} dr dxds\leq  e^{\mathbf{K}T}\mathbf{\Lambda}:= \overline{\mathbf{K}}. \label{key-estimate-beta-2case}
\end{align}
\begin{proposition}\label{prop-cv-to-0}
        Let $0<\alpha\leq \dfrac{d}{2}$ and assume that $(f_0^\tau)_\tau$ is bounded in $H_\alpha$. Then  there exists a subsequence of $(f_\tau)_\tau$ denoted by  $(f_{\tau_k})_{\tau_k}$ such that \begin{align}\label{weak-cv-info-prop}
    f_{\tau_k} \stackrel{\ast}{\rightharpoonup} 0 \text{ in } L^\infty(0,T;H_\alpha) \text{ as } \tau_k \to 0.
\end{align}
    \end{proposition}
    \begin{proof}
By using \eqref{key-estimate-beta-2case} and Banach–Alaoglu–Bourbaki theorem  (see \textit{e.g.} \cite[Theorem 3.16]{Brezis}),  one gets the existence of subsequence  of $(f_\tau)_\tau$ (denoted by $(f_{\tau_k})_{\tau_k}$ ) and $\overline{f}\in L^\infty(0,T;H_\alpha)$ such that
\begin{align}\label{weak-cv-info}
    f_{\tau_k} \stackrel{\ast}{\rightharpoonup} \overline{f} \text{ in } L^\infty(0,T;H_\alpha) \text{ as } \tau_k \to 0,
\end{align}
and 
\begin{align}\label{ineq-beta-0}\int_0^T\int_{\mathbb{T}^d}\int_{\mathbb{R}^d}\vert \nabla_r (f_{\tau_k}(s)P_{-\alpha})\vert^2(1+\frac{\vert r\vert^2}{2})^{-\alpha+1} dr dxds\leq   \overline{\mathbf{K}}\zeta \alpha\tau_k. 
\end{align}
From \eqref{ineq-beta-0}, we obtain 
\begin{align*}
    \vert \nabla_r (f_{\tau_k}(s)P_{-\alpha})\vert^2(1+\frac{\vert r\vert^2}{2})^{1-\alpha} \to 0 \text{ in } L^1([0,T]\times \mathbb{T}^d \times \mathbb{R}^d) \text{ as } \tau_k \to 0.
\end{align*}
Therefore, up to subsequence if necessary, we get 
\begin{align}\label{cv-dist-almost}
    \vert \nabla_r (f_{\tau_k}(s)P_{-\alpha})\vert^2 \to 0 \text{ a.e.  in } [0,T]\times \mathbb{T}^d \times \mathbb{R}^d \text{ as } \tau_k \to 0,
\end{align}
since $(1+\frac{\vert r\vert^2}{2})^{1-\alpha}>0.$ On the other hand,  thanks to  \eqref{weak-cv-info} and \eqref{cv-dist-almost}   we get
the following convergence in the sense of distributions
\begin{align*}
  \langle  \nabla_r (f_{\tau_k}(s)P_{-\alpha}), \Phi\rangle_{\mathcal{D}^\prime,\mathcal{D}}  \to  \langle \nabla_r (\overline{f}(s)P_{-\alpha}), \Phi\rangle_{\mathcal{D}^\prime,\mathcal{D}}=0,\quad \forall \Phi \in \mathcal{D}, 
\end{align*}
where $\mathcal{D}= C_c^\infty(]0,T[)\otimes C^\infty(\mathbb{T}^d)\otimes C^\infty_c(\mathbb{R}^d).$ Hence, we deduce the existence of distribution $h_{t,x}$ such that 
$$ \overline{f}P_{-\alpha} =h_{t,x} \text{ in  } \mathcal{D}^\prime \Leftrightarrow  \overline{f} =h_{t,x}\otimes P_{\alpha}(r)=h_{t,x}\otimes(1+\frac{\vert r\vert^2}{2})^{-\alpha} \text{ in  } \mathcal{D}^\prime.$$
We recall that $\int_{\mathbb{R}^d}(1+\frac{\vert r\vert^2}{2})^{-\alpha} dr$ diverges if $0<\alpha\leq \dfrac{d}{2}$. Since  
    $\overline{f} \in L^\infty(0,T;H_\alpha)$ from  \eqref{weak-cv-info}, we deduce that $h_{t,x} \equiv 0$.   
    \end{proof}
        \subsubsection{The $L^1-$setting} We prove the following result.
\begin{proposition}\label{prop-est-L1}\label{weak-cv-info-prop}
Assume  that   $f^\tau_0 \in L^1(\mathbb{T}^d\times \mathbb{R}^d) \cap H$ and there exists $\mathbf{\Lambda}>0$ (independent of $\tau$)  such that $\displaystyle\sup_{\tau>0}\Vert f^\tau_0 \Vert_{L^1}\leq  \mathbf{\Lambda}. $ Then
\begin{align}\label{est-L1-key}
    \sup_{t\in[0,T]}  \int_{\mathbb{T}^d}\int_{\mathbb{R}^d}\vert f_\tau(t)\vert  drdx \leq \int_{\mathbb{T}^d}\int_{\mathbb{R}^d}\vert f_0^\tau\vert drdx \leq \mathbf{\Lambda}.
\end{align}
Moreover, there exists a subsequence of $(f_\tau)_\tau$ (denoted by $(f_{\tau_k } )_k$)  and $\nu\in L^\infty(0,T;\mathcal{M}(\mathbb{T}^d\times \mathbb{R}^d))$ such that 
\begin{align}
      f_{\tau_k } \stackrel{\ast}{\rightharpoonup} \nu  \text{ in } L^\infty(0,T;\mathcal{M}(\mathbb{T}^d\times \mathbb{R}^d)) \text{ as } k \to +\infty,
\end{align}
and $\nu$ solves, in the sense of distributions, the following equation
\begin{align}
    \Div_r\big(\alpha r\nu+\nabla_r \nu+\dfrac{1}{2}A_b(r)\nabla_r \nu\big)=0 \text{ in } \mathcal{D}^\prime.
\end{align}
\end{proposition}
\begin{proof} 
Let $\epsilon>0$, we define  $\eta_\epsilon$ the odd continuous function on $\mathbb{R}$  as follows
\begin{align*}
    \eta_\epsilon(s)&=\begin{cases}
        1 & \textit{ if } \epsilon < s ,\\
  \dfrac{1}{2}\big[1+\sin{\big(\dfrac{\pi}{2\epsilon}(2s-\epsilon)\big)}\big], & \textit{ if } 0\leq s\leq \epsilon.
    \end{cases}
    \end{align*}
    Notice that $\eta_\epsilon\in C^1(\mathbb{R}),  \eta_\epsilon^{\prime\prime}\in L^\infty(\mathbb{R})$ and satisfies:
    \begin{itemize}
        \item $\vert  \eta_\epsilon (s)\vert \leq 1\wedge \dfrac{\pi s}{\epsilon}, \forall s    \in \mathbb{R};$ \qquad $\bullet$ $\eta_\epsilon^\prime \geq 0, \quad\text{supp} (\eta_\epsilon^\prime) \subset [-\epsilon,\epsilon]$ and $\vert  \eta_\epsilon^\prime(s)\vert \leq \dfrac{\pi}{2\epsilon},  \forall s    \in \mathbb{R}$,
        \item $\eta_\epsilon(s)$ converges pointwise to $\text{sign}(s)$   as $\epsilon \to 0.$
    \end{itemize}
    Let $\delta,\epsilon>0,$     note that $\eta_\epsilon([f_\tau(\cdot)]_\delta) \in \mathcal{Y}$. By using \eqref{eqn-regulari-test}, we get  for any $t\in [0,T]$
\begin{align}&\int_{\mathbb{T}^d}\int_{\mathbb{R}^d}[f_\tau(t)]_\delta\eta_\epsilon([f_\tau(t)]_\delta) drdx-\int_{\mathbb{T}^d}\int_{\mathbb{R}^d}[f_0^\tau]_\delta\eta_\epsilon([f_0^\tau]_\delta) drdx\notag\\&= -\int_0^t\int_{\mathbb{T}^d}\int_{\mathbb{R}^d}[u_L(s)\cdot \nabla_xf_\tau(s)]_\delta \eta_\epsilon([f_\tau(s)]_\delta)+ [(\nabla u_L(s)r) \cdot\nabla_rf_\tau(s)]_\delta\eta_\epsilon([f_\tau(s)]_\delta) dr dxds\notag\\	&+	\dfrac{1}{\zeta\alpha\tau}\int_0^t\int_{\mathbb{T}^d}\int_{\mathbb{R}^d} \alpha [\Div_r( f_\tau(s)r)]_\delta\eta_\epsilon([f_\tau(s)]_\delta) dr dxds\notag \\&+\dfrac{1}{\zeta\alpha\tau}\int_0^t\langle[\Delta_rf_\tau(s) ]_\delta,\eta_\epsilon([f_\tau(s)]_\delta)  \rangle+\dfrac{1}{2}\langle[\Div_r(A_b(r) \nabla_rf_\tau(s))]_\delta,\eta_\epsilon([f_\tau(s)]_\delta) \rangle ds.\notag\end{align}
First, note that 
\begin{align*}
    &\int_0^t\int_{\mathbb{T}^d}\int_{\mathbb{R}^d}[u_L(s)\cdot \nabla_xf_\tau(s)]_\delta \eta_\epsilon([f_\tau(s)]_\delta)drdxds\\&=   \int_0^t\int_{\mathbb{T}^d}\int_{\mathbb{R}^d}\big([u_L(s)\cdot \nabla_xf_\tau(s)]_\delta -   u_L(s)\cdot \nabla_x[f_\tau(s)]_\delta \big)\eta_\epsilon([f_\tau(s)]_\delta)drdxds\\&\quad+   \int_0^t\int_{\mathbb{T}^d}\int_{\mathbb{R}^d}u_L(s)\cdot \nabla_x[f_\tau(s)]_\delta \eta_\epsilon([f_\tau(s)]_\delta) drdxds=A_\delta+B_\delta \to 0 \text{ as } \delta \to 0.
\end{align*}
Concerning the term $A_\delta$,  recall that $\vert \eta_\epsilon([f_\tau(s)]_\delta)\vert \leq \dfrac{\pi}{\epsilon}\vert [f_\tau(s)]_\delta\vert, $ which ensures that $(\eta_\epsilon([f_\tau(s)]_\delta))_\delta$ is bounded in  $L^2(\mathbb{T}^d\times \mathbb{R}^d\times(0,T))$. Therefore, we can use the commutator estimate ( see \autoref{Lemma-tech-1}) to get $\displaystyle\lim_{\delta \to 0}A_\delta=0.$
On the other hand,  denote by  $N_\epsilon$  the primitive function of $\eta_\epsilon$, then  by using "divergence theorem" we get
\begin{align*}
  B_\delta=   \int_0^t\int_{\mathbb{T}^d}\int_{\mathbb{R}^d}u_L(s)\cdot \nabla_x[f_\tau(s)]_\delta \eta_\epsilon([f_\tau(s)]_\delta) drdxds=\int_0^t\int_{\mathbb{R}^d} \int_{\mathbb{T}^d}\Div_x\big( u_L(s) N_\epsilon([f_\tau(s)]_\delta)\big) dxdrds=0.
\end{align*}
Secondly, we have    $\displaystyle\lim_{\delta}  \int_0^t\int_{\mathbb{T}^d}\int_{\mathbb{R}^d}[\Div_r(\nabla u_L(s)r)f_\tau(s)]_{\delta} \eta_\epsilon([f_\tau(s)]_\delta)drdx   ds=0.$ Indeed, note that
\begin{align*}
    &\int_0^t\int_{\mathbb{T}^d}\int_{\mathbb{R}^d}[\Div_r(\nabla u_L(s)rf_\tau(s))]_{\delta} \eta_\epsilon([f_\tau(s)]_\delta)drdx   ds\\&=\int_0^t\int_{\mathbb{T}^d}\int_{\mathbb{R}^d}\big([\Div_r(\nabla u_L(s)r f_\tau(s))]_{\delta}-\Div_r(\nabla u_L(s)r[f_\tau(s)]_{\delta})\big)\eta_\epsilon([f_\tau(s)]_\delta)drdx   ds\\
    &+\int_0^t\int_{\mathbb{T}^d}\int_{\mathbb{R}^d}\Div_r(\nabla u_L(s)r[f_\tau(s)]_{\delta}) \eta_\epsilon([f_\tau(s)]_\delta)drdx   ds=C_\delta+D_\delta.
\end{align*}
Recall that    $\Div_r(\nabla u_L(s)r)=0$, by   using "divergence theorem" we get
\begin{align*}
    D_\delta&=\int_0^t\int_{\mathbb{T}^d}\int_{\mathbb{R}^d}\Div_r(\nabla u_L(s)r[f_\tau(s)]_{\delta} )\eta_\epsilon([f_\tau(s)]_\delta)drdx   ds\\&=\int_0^t\int_{\mathbb{T}^d}\int_{\mathbb{R}^d}\Div_r(\nabla u_L(s)rN_\epsilon([f_\tau(s)]_\delta)drdx   ds=0.
\end{align*}
Concerning $C_\delta$, note that
\begin{align*}
   & \int_0^t\int_{\mathbb{T}^d}\int_{\mathbb{R}^d}\big([\Div_r(\nabla u_L(s)rf_\tau(s))]_{\delta}-\Div_r(\nabla u_L(s)r[f_\tau(s)]_{\delta})\big)\eta_\epsilon([f_\tau(s)]_\delta)drdx   ds\\
&=\int_0^t\int_{\mathbb{T}^d}\int_{\mathbb{R}^d} \big(\Div_r(\Theta_\delta*(\nabla u_L(s)rf_\tau(s))-\Div_r(\nabla u_L(s)r\Theta_\delta *   f_\tau(s) )\big)\eta_\epsilon(\Theta_\delta *   f_\tau(s) )  drdxds\\
    &=-\int_0^t\int_{\mathbb{T}^d}\int_{\mathbb{R}^d} \big(\Theta_\delta*(\nabla u_L(s)rf_\tau(s)-(\nabla u_L(s)r\Theta_\delta *   f_\tau(s) \big)\Theta_\delta *   \nabla_rf_\tau(s)\eta_\epsilon^\prime(\Theta_\delta *   f_\tau(s) )drdx   ds.
\end{align*}
On  the other   hand,  for any $(s,x,r)\in [0,T]\times \T^d\times \R^d$  note    that
\begin{align*}
    &\big(\Theta_\delta*(\nabla u_L(\cdot)rf_\tau\big)(s,x,r)-\nabla u_L(s,x)r(\Theta_\delta *   f_\tau)(s,x,r)\\
    &=\int_{\R^d}[\nabla    u_L(x-y,s)-\nabla   u_L(x,s)]r\Theta_\delta(y)f_\tau(s,x-y,r)dy.
\end{align*}
By   mean-value theorem we  get
$    \vert   \nabla u_L(x-y,s)    -\nabla    u_L(x,s)\vert \leq        \vert       y\vert \Vert   u_L\Vert_{C^2}$
and 
\begin{align*}
    &\vert  \big(\Theta_\delta*(\nabla u_L(\cdot)rf_\tau\big)(s,x,r)-\nabla u_L(s,x)r(\Theta_\delta *   f_\tau)(s,x,r)\vert    \\
    &\leq  \Vert   u_L\Vert_{C^2} \int_{\R^d}\vert    y\vert  \vert   r\vert  \Theta_\delta(y)\vert   f_\tau(s,x-y,r)\vert       dy\leq  \delta\Vert   u_L\Vert_{C^2} \int_{\R^d}   \Theta_\delta(y)\vert   r\vert \vert   f_\tau(s,x-y,r)\vert   dy,
\end{align*}
since   $\text{supp }    \Theta\subset B(0,1).$    Therefore,  we  get
\begin{align*}
    &\int_0^t\int_{\mathbb{T}^d}\int_{\mathbb{R}^d} \vert \big(\Theta_\delta*(\nabla u_L(s)rf_\tau(s))-(\nabla u_L(s)r\Theta_\delta *   f_\tau(s) \big)\Theta_\delta *   \nabla_rf_\tau(s) \eta_\epsilon^\prime(\Theta_\delta *   f_\tau(s) )\vert  drdxds\\
    &\leq  \dfrac{\pi}{2\epsilon}  \delta\Vert   u_L\Vert_{C^2}    \int_0^t\Vert \int_{\R^2}   \Theta_\delta(y)\vert   r\vert \vert   f_\tau(s,x-y,r)\vert   dy \Vert_{L^2_{x,r}} \Vert\Theta_\delta *   \nabla_rf_\tau(s)\Vert_{L^2_{x,r}} ds\\
    &\leq   \dfrac{\pi}{2\epsilon} \delta\Vert   u_L\Vert_{C^2}    \int_0^t\Vert \Theta_\delta *\vert   r\vert \vert   f_\tau(s)\vert    \vert    \Vert_{L^2_{x,r}} \Vert\Theta_\delta *   \nabla_rf_\tau(s)\Vert_{L^2_{x,r}} ds\\
    &\leq   \dfrac{\pi}{2\epsilon} \delta\Vert   u_L\Vert_{C^2}    \int_0^t\Vert    f_\tau(s)        \Vert_H \Vert   \nabla_rf_\tau(s)\Vert_{L^2_{x,r}} ds\to    0   \text{  as    } \delta  \to 0.
\end{align*}
Thanks to the properties of the matrix $A_b$  and $\eta_\epsilon^\prime\geq 0,$ we deduce
\begin{align}
   & \dfrac{1}{\zeta\alpha\tau}\int_0^t\langle[\Delta_rf_\tau(s) ]_\delta,\eta_\epsilon([f_\tau(s)]_\delta)  \rangle+\dfrac{1}{2}\langle[\Div_r(A_b(r) \nabla_rf_\tau(s))]_\delta,\eta_\epsilon([f_\tau(s)]_\delta) \rangle ds\label{est-cancelation}\\&=-  \dfrac{1}{\zeta\alpha\tau}\int_0^t \int_{\mathbb{T}^d}\int_{\mathbb{R}^d}\big(\vert [\nabla_rf_\tau(s) ]_\delta\vert^2+\dfrac{1}{2}A_b(r) \nabla_r[f_\tau(s)]_\delta\cdot \nabla_r[f_\tau(s)]_\delta\big)\eta_\epsilon^\prime([f_\tau(s)]_\delta) drdx ds\notag\\
   &\leq -  \dfrac{1}{\zeta\alpha\tau}\int_0^t \int_{\mathbb{T}^d}\int_{\mathbb{R}^d}(1+\dfrac{\vert r\vert^2}{2})\vert [\nabla_rf_\tau(s) ]_\delta\vert^2\eta_\epsilon^\prime([f_\tau(s)]_\delta) drdx ds\notag.
\end{align}
 Finally, we have
\begin{align*}
    &\dfrac{1}{\tau}\int_0^t\int_{\mathbb{T}^d}\int_{\mathbb{R}^d}  [\Div_r( f_\tau(s)r)]_\delta\eta_\epsilon([f_\tau(s)]_\delta) dr dxds\\&=-   \dfrac{1}{\zeta\tau}\int_0^t\int_{\mathbb{T}^d}\int_{\mathbb{R}^d}  [ f_\tau(s)]_\delta r\cdot\nabla_r[f_\tau(s)]_\delta\eta_\epsilon^\prime([f_\tau(s)]_\delta) dr dxds\\
    &\leq  \dfrac{1}{\zeta\alpha\tau}\int_0^t \int_{\mathbb{T}^d}\int_{\mathbb{R}^d}\dfrac{\vert r\vert^2}{2}\vert [\nabla_rf_\tau(s) ]_\delta\vert^2\eta_\epsilon^\prime([f_\tau(s)]_\delta) drdx ds+\dfrac{\alpha}{2\zeta\tau} \int_0^t \int_{\mathbb{T}^d}\int_{\mathbb{R}^d} [f_\tau(s)]_\delta^2\eta_\epsilon^\prime([f_\tau(s)]_\delta) drdxds.
\end{align*}
By using \eqref{est-cancelation}, then letting $\delta \to 0$ we obtain
\begin{align*}
    \int_{\mathbb{T}^d}\int_{\mathbb{R}^d}f_\tau(t)\eta_\epsilon(f_\tau(t)) drdx &\leq \int_{\mathbb{T}^d}\int_{\mathbb{R}^d}f_0^\tau\eta_\epsilon(f_0^\tau) drdx+\dfrac{\alpha}{2\zeta\tau} \int_0^t \int_{\mathbb{T}^d}\int_{\mathbb{R}^d} f_\tau^2(s)\eta_\epsilon^\prime(f_\tau(s)) drdxds\\
   &\leq  \int_{\mathbb{T}^d}\int_{\mathbb{R}^d}\vert f_0^\tau\vert drdx+\dfrac{\alpha}{2\zeta\tau} \int_0^t \int_{\mathbb{T}^d}\int_{\mathbb{R}^d} \dfrac{\pi}{2\epsilon}f_\tau^2(s)1_{\{  \vert f_\tau \vert \leq \epsilon\}} drdxds \\&\leq \int_{\mathbb{T}^d}\int_{\mathbb{R}^d}\vert f_0^\tau\vert drdx+\dfrac{\alpha\pi}{4\zeta\tau} \int_0^t \int_{\mathbb{T}^d}\int_{\mathbb{R}^d} \vert f_\tau(s) \vert 1_{\{  \vert f_\tau \vert \leq \epsilon\}} drdxds.
    \end{align*}
By using Fatou's lemma (see \textit{e.g.} \cite[Theorem 4.2]{Brezis}), we get 
\begin{align*}
    \int_{\mathbb{T}^d}\int_{\mathbb{R}^d}\vert f_\tau(t)\vert  drdx \leq \liminf_{\epsilon \to 0}\int_{\mathbb{T}^d}\int_{\mathbb{R}^d}f_\tau(t)\eta_\epsilon(f_\tau(t)) drdx.
\end{align*}
In conclusion, we get
\begin{align}\label{est-L1-1}
     \int_{\mathbb{T}^d}\int_{\mathbb{R}^d}\vert f_\tau(t)\vert  drdx \leq \int_{\mathbb{T}^d}\int_{\mathbb{R}^d}\vert f_0^\tau\vert drdx+\dfrac{\alpha\pi}{4\zeta\tau} \int_0^t \int_{\mathbb{T}^d}\int_{\mathbb{R}^d} \vert f_\tau(s) \vert 1_{\{  \vert f_\tau \vert \leq \epsilon\}} drdxds. 
\end{align}
By using Gronwall inequality, we deduce
\begin{align}\label{esti-L1-2}
  \sup_{t\in[0,T]}   \int_{\mathbb{T}^d}\int_{\mathbb{R}^d}\vert f_\tau(t)\vert  drdx \leq  \exp{(\dfrac{\alpha\pi}{4\zeta\tau}T)}\int_{\mathbb{T}^d}\int_{\mathbb{R}^d}\vert f_0^\tau\vert drdx \leq \exp{(\dfrac{\alpha\pi}{4\zeta\tau}T)} \mathbf{\Lambda}.  
\end{align}
Since $\vert f_\tau \vert 1_{\{  \vert f_\tau \vert \leq \epsilon\}} \to 0$ a.e. as $\epsilon \to 0$ and $f_\tau \in L^1([0,T]\times \mathbb{T}^d\times \mathbb{R}^d)$ thanks to \eqref{esti-L1-2}, dominated convergence theorem (see \textit{e.g.} \cite[Theorem 4.2]{Brezis})  ensures 
$$\lim_{\epsilon\to 0}\int_0^t \int_{\mathbb{T}^d}\int_{\mathbb{R}^d} \vert f_\tau(s) \vert 1_{\{  \vert f_\tau \vert \leq \epsilon\}} drdxds=0.$$
Therefore,  after passing to the limit as $\delta\to 0 $ then $\epsilon\to 0$ in \eqref{est-L1-1}, we obtain \eqref{est-L1-key}.
\subsubsection*{The limit as $\tau \to 0$}
 We recall that Riesz–Markov–Kakutani representation theorem ensures  $\mathcal{M}(\mathbb{T}^d\times \mathbb{R}^d)\simeq (C_c(\mathbb{T}^d\times \mathbb{R}^d))^\prime,$ see \cite[Theorem 2.14]{Rudin}.  Since $L^1(\mathbb{T}^d\times \mathbb{R}^d) \hookrightarrow  \mathcal{M}(\mathbb{T}^d\times \mathbb{R}^d)$. By using  \eqref{est-L1-key} and Banach–Alaoglu–Bourbaki theorem  (see \textit{e.g.} \cite[Theorem 3.16]{Brezis}), we can extract   a subsequence of $(f_\tau)_\tau$ (denoted by $(f_{\tau_k })_{\tau_k }$)  and $\nu\in L^\infty(0,T;\mathcal{M}(\mathbb{T}^d\times \mathbb{R}^d))$ such that 
\begin{align}\label{cv-Radon1}
      f_{\tau_k } \stackrel{\ast}{\rightharpoonup} \nu  \text{ in } L^\infty(0,T;\mathcal{M}(\mathbb{T}^d\times \mathbb{R}^d)) \text{ as } k \to 0.
\end{align}
By using \eqref{equ-final-I},  we have  
\begin{align}&\langle\dfrac{d f_\tau(s)}{ds},\phi\rangle= \int_{\mathbb{T}^d}\int_{\mathbb{R}^d} G(f_\tau(s),\phi) drdx:=\int_{\mathbb{T}^d}\int_{\mathbb{R}^d}f_\tau(s)\left(u_L(s)\cdot \nabla_x \phi+ (\nabla u_L(s)r) \cdot\nabla_r\phi\right) dr dx\notag\\	&-	\dfrac{1}{\zeta\alpha\tau}\int_{\mathbb{T}^d}\int_{\mathbb{R}^d} \alpha f_\tau(s)r\cdot\nabla_r\phi+\nabla_rf_\tau(s) \cdot\nabla_r\phi  +\dfrac{1}{2}A_b(r) \nabla_rf_\tau(s)\cdot\nabla_r\phi dr dx, \forall \phi \in \mathcal{Y}. \notag\end{align} 
Let $\xi \in C^\infty_c(]0,T[)$, multiplying the last equality by $\xi$ and integrating by parts with respect to $s$, we obtain
\begin{align*}
  - \int_0^T \int_{\mathbb{T}^d}\int_{\mathbb{R}^d} f_\tau(s)\phi \partial_s\xi ds drdx=\int_0^T \int_{\mathbb{T}^d}\int_{\mathbb{R}^d} G(f_\tau(s),\phi)\xi drdxds.
\end{align*}
Now, set 
$\phi=\vartheta\otimes\varphi\in C^\infty(\mathbb{T}^d)\otimes C_c^\infty(\mathbb{R}^d)$ and  integrate by parts with respect to $r$ in the right hand side, we infer
\begin{align}&-\zeta\alpha\tau_k \int_0^T\hspace{-0.1cm}\int_{\mathbb{T}^d}\hspace{-0.1cm}\int_{\mathbb{R}^d}\hspace{-0.1cm}f_{\tau_k }(s)\partial_t\xi\otimes\vartheta\otimes\varphi +f_{\tau_k }(s)\left(u_L(s)\cdot  \xi\otimes \nabla_x\vartheta\otimes\varphi+ (\nabla u_L(s)r) \cdot\xi\otimes\vartheta\otimes\nabla_r\varphi\right) dr dxds\notag\\	&=\int_0^T\hspace{-0.1cm}\int_{\mathbb{T}^d}\hspace{-0.1cm}\int_{\mathbb{R}^d}\hspace{-0.2cm} -	\alpha f_{\tau_k }(s)r\cdot\xi\otimes\vartheta\otimes\nabla_r\varphi+f_{\tau_k }(s) \xi\otimes\vartheta\otimes\Delta_r\varphi  +\dfrac{1}{2} f_{\tau_k }(s)\Div_r(A_b(r)\xi\otimes\vartheta\otimes\nabla_r\varphi) dr dxds\notag\\
&=\int_0^T\langle f_{\tau_k }(s),  -\alpha r\cdot\xi\otimes\vartheta\otimes\nabla_r\varphi+ \xi\otimes\vartheta\otimes\Delta_r\varphi  +\dfrac{1}{2} \Div_r(A_b(r)\xi\otimes\vartheta\otimes\nabla_r\varphi)\rangle_{\mathcal{M},C_c} ds.\notag\end{align}
By using \eqref{cv-Radon1} and \eqref{est-L1-key}, we deduce as $\tau\to 0$
\begin{align}\label{passge-limit}
    \int_0^T\langle \nu(s),  -\alpha r\cdot\xi\otimes\vartheta\otimes\nabla_r\varphi+ \xi\otimes\vartheta\otimes\Delta_r\varphi  +\dfrac{1}{2} \Div_r(A_b(r)\xi\otimes\vartheta\otimes\nabla_r\varphi)\rangle_{\mathcal{M},C_c} ds=0.
\end{align}
In particular, $\nu$ is weak solution to the following  equation
\begin{align}\label{radon-2}
    \Div_r\big(\alpha r\nu+\nabla_r \nu+\dfrac{1}{2}A_b(r)\nabla_r \nu\big)=0 \text{ in } \mathcal{D}^\prime.
\end{align}
\end{proof}

\section{Discussion and outlook}\label{outlook}
\subsection{Summary}
We study the distribution of polymers in turbulent flow by  using  stochastic modeling of the  small-scales and the Fokker-Planck  equation  in the Hookean case.
We rigorously prove the emergence of the heavy-tail distribution of polymer elongation.
More precisely, the fluid velocity is decomposed into a smooth large-scale component and a rapidly fluctuating small-scale turbulent field. We send both the spatial and temporal scales of these turbulent fluctuations to zero, by using scaling and singular limits techniques. 
This allows to fill a gap between formal theory,   numerical simulations in polymer-turbulence literature and rigorous mathematical analysis. We distinguishes  two main regimes for  polymers advected  and  stretched by small-scale turbulence. Namely, we have
\begin{itemize}
\item \textit{The probability density  regime (\(\alpha > d/2\))}: The limit Fokker--Planck density factorizes as 
\( \rho(t,x)\otimes (1 + |r|^2/2)^{-\alpha} \), which is a power-law distribution in the end-to-end vector \(r\). $\alpha$  becomes smaller   in stronger turbulent flow and encodes the competition between polymer relaxation  and turbulent strain fluctuations. This corresponds to a well-defined probability density, representing  coil and moderate stretching states. 
Physically, polymer extension statistics remain bounded and close to equilibrium as  \(\alpha > d+\tfrac{1}{2}\), and polymers are in "\textit{coil-state}".
\\
\item \textit{The Radon measure  regime and coil-stretch transition (\(0 < \alpha \le d/2\))}: $\alpha=\frac{d}{2}$  is the \textit{coil-stretch transition} threshold and the limit object should be understood a Radon measure $h_{x}(t)\otimes(1 + |r|^2/2)^{-\alpha} dr$ exhibiting power-law behavior in the end-to-end vector.  This reflects the \textit{extreme stretching} of polymers induced by turbulence and  the possibility of polymers reaching infinite length, which indicates that  small-scale turbulence dominates. This corresponds to the loss of $L^1$-compactness from a mathematical point of view.
\end{itemize}

\subsection{FENE model}\label{discussion-FENE}		It is worth  making some comment on the \textit{Finitely Extensible Nonlinear Elastic (FENE)} model.  While the Hookean model assumes linear elasticity and infinite extensibility,
the FENE model introduces nonlinearity and a finite maximum extension. Namely, in the FENE model  we replace $\frac{1}{\beta}R_t$ in  \eqref{Intro1} by 
		$\frac{1}{\beta}\frac{R_t}{1-\vert R_t\vert^2/b},$
		where 
         $\sqrt{b}$ is    the upper bound of the length extension of the polymer chains \textit{i.e.} the   chains  cannot stretch beyond the length          $\sqrt{b}$. 
      From mathematical point of view, the analysis must be considered in a bounded domain with respect to $r$-variable with the appropriate boundary conditions. The rigorous justification of the  scaling and singular limits in the case of FENE model is the subject of the  ongoing work \cite{FENE-YAS-FED}. 
      However, for the convenience of the reader, let us write what do we expect for the limit density. 
In the case of FENE model, we expect to get on the right hand side of \eqref{sys-I} (after the  scaling limit)
\begin{align*}
    \dfrac{1}{\tau}\Div_r\Big(\dfrac{br}{b-\vert r\vert^2}f_\tau+\nabla_rf_\tau+\dfrac{C_d}{2} A_b(r)\nabla_r f_\tau\Big), \quad \vert r\vert^2 <b.
\end{align*}
Concerning the singular limit  \textit{i.e.} as $\tau \to 0$, we expect  $f_\tau \to \mathfrak{f}$ where $$\Div_r\Big(\dfrac{br}{b-\vert r\vert^2}\mathfrak{f}+\nabla_r\mathfrak{f}+\dfrac{C_d}{2} A_b(r)\nabla_r \mathfrak{f}\Big)=0, \quad \vert r\vert^2 <b,$$
with the interpretation that $C_d$ becomes larger as the turbulence becomes stronger. Therefore, let use  look for radial solution  to the following equation
\begin{align*}
    \dfrac{br}{b-\vert r\vert^2}\mathfrak{f}+	\nabla_r \mathfrak{f}+\dfrac{C_d}{2} A_b(r)\nabla_r \mathfrak{f}=0.
\end{align*}
Set $t=\vert r\vert$, we get the following ODE in $t:$
$\dfrac{b}{C_d}\dfrac{1}{b-t^2}\mathfrak{f}+	(\dfrac{2+C_d t^2}{2C_d t})\mathfrak{f}^\prime=0. $
By solving the  ODE, we get
\begin{align}\label{formula-stat}
    \mathfrak{f}( r)=\dfrac{1}{Z}\big(\frac{b-\vert r\vert^2}{b}\big)^{\gamma} \big(\frac{2+C_d\vert r\vert^2}{2}\big)^{-\gamma}, \gamma= \frac{b}{2+C_d b}>0 \text{ and } Z \text{ is a renormalized constant.}
\end{align}
Note that the last formula \eqref{formula-stat} exhibits power-law behavior, with  the term reflecting the fact that the domain is bounded. It is also similar to the stationary probability density  derived  in \cite{afonso2005nonlinear}.
Now, let us make some remarks on the FENE model in comparison with the Hookean one. First, note that $\mathfrak{f}$ is always a probability density function, unlike the Hookean case. Furthermore, the stretching  saturates at $\sqrt{b}$ whereas in the Hookean model, we could obtain an infinite length. Finally,
for intermediate extensions, $R_{min}<< \vert r\vert << \sqrt{b}$ , we have
\begin{align*}
    \mathfrak{f} \sim \vert r\vert^{d-1}(1+C_d\frac{\vert r\vert^2}{2})^{\frac{-b}{2+C_d b}} \sim  \vert r\vert^{-1+d-\frac{2b}{2+C_d b}}.
\end{align*}
In accordance with  \cite{Balk},  the coil-stretch holds  at $(2+C_d b)d=2b$. A more in-depth discussion of this transition and its link  to the Hookean case is part of the ongoing work \cite{FENE-YAS-FED}, see also \cite{PicardoLanceVinc} for some recent numerical simulation and discussions.

   \appendix 
\section{ Proof of \autoref{lemma-appendix}}\label{Appendix}
\begin{proof}[Proof of \autoref{lemma-appendix}]
 Let   us    pass    to  the limit  as $\delta \to 0$ then $M\to +\infty$   in \eqref{est-M-delta}.
   \begin{itemize}
       \item Recall that $f\in C_w([0,T],H)$, by using the properties of convolution and the lower semi-continuity of weak convergence, we obtain
       \begin{align*}
          \liminf_{M\to +\infty} \int_{\mathbb{T}^d}\int_{\mathbb{R}^d}f_\tau(t)^2 T_M drdx \leq \liminf_{M\to +\infty}\liminf_{\delta\to 0} \int_{\mathbb{T}^d}\int_{\mathbb{R}^d}[f_\tau(t)]_\delta^2 T_M drdx,
       \end{align*}
       by Fatou's lemma \yas{(see \textit{e.g.} \cite[Lemma 4.2]{Brezis})}, we deduce 
        \begin{align*}
      \int_{\mathbb{T}^d}\int_{\mathbb{R}^d}f_\tau(t)^2(1+\frac{\vert r\vert^2}{2})^\alpha  drdx\leq     \liminf_{M\to +\infty} \int_{\mathbb{T}^d}\int_{\mathbb{R}^d}f_\tau(t)^2 T_M drdx.
       \end{align*}
       \item Since $f_0^\tau\in H_\alpha$, we have $\displaystyle\lim_{\delta\to 0}\int_{\mathbb{T}^d}\int_{\mathbb{R}^d}[f_0^\tau]_\delta^2 T_M drdx=\int_{\mathbb{T}^d}\int_{\mathbb{R}^d}(f_0^\tau)^2 T_M drdx$ and  monotone convergence theorem ensures
       \begin{align*}
         \lim_{M\to +\infty}  \int_{\mathbb{T}^d}\int_{\mathbb{R}^d}(f_0^\tau)^2 T_M drdx=\int_{\mathbb{T}^d}\int_{\mathbb{R}^d}(f_0^\tau)^2(1+\frac{\vert r\vert^2}{2})^\alpha  drdx \leq \mathbf{\Lambda}.
       \end{align*}
       \item $\displaystyle\lim_{\delta\to 0}\int_0^t\int_{\mathbb{T}^d}\int_{\mathbb{R}^d}[u_L(s)\cdot \nabla_xf_\tau(s)]_\delta  T_M[f_\tau(s)]_\delta drdxds=0.$ Indeed, it holds by using commutators estimate, see \autoref{Lemma-tech-1} (see also \textit{e.g.}   \cite[Proof of (55)]{FlaTah24} for a similar arguments).\\
       
\item  The term $\displaystyle\int_0^t\int_{\mathbb{T}^d}\int_{\mathbb{R}^d} [(\nabla u_L(s)r) \cdot\nabla_rf_\tau(s)]_\delta T_M[f_\tau(s)]_\delta dr dxds$. By using  \eqref{regularity-f} and the properties of convolution one gets 
\begin{align*}
  &\lim_{\delta\to 0}  \int_0^t\int_{\mathbb{T}^d}\int_{\mathbb{R}^d} [(\nabla u_L(s)r) \cdot\nabla_rf_\tau(s)]_\delta T_M[f_\tau(s)]_\delta dr dxds\\&=\int_0^t\int_{\mathbb{T}^d}\int_{\mathbb{R}^d} (\nabla u_L(s)r) \cdot\nabla_rf_\tau(s) T_Mf_\tau(s) dr dxds.
\end{align*}
Now, by using that $\vert T_M^\prime\vert \leq 1$ we get
\begin{align*}
    \vert\int_0^t\int_{\mathbb{T}^d}\int_{\mathbb{R}^d}& (\nabla u_L(s)r) \cdot\nabla_rf_\tau(s) T_Mf_\tau(s) dr dxds\vert\\&\leq \alpha\vert\int_0^t\int_{\mathbb{T}^d}\int_{\mathbb{R}^d} (\nabla u_L(s)r) \cdot r f_\tau^2(s) (1+\frac{\vert r\vert^2}{2})^{\alpha-1} T_M^\prime dr dxds\vert\\
    &\leq 2\alpha\Vert \nabla_xu_L \Vert_\infty   \int_0^t\int_{\mathbb{T}^d}\int_{\mathbb{R}^d} \dfrac{\vert  r\vert^2}{2} f_\tau^2(s) (1+\frac{\vert r\vert^2}{2})^{\alpha-1}  dr dxds \\
    &\leq 2\alpha\Vert \nabla_xu_L \Vert_\infty   \int_0^t\int_{\mathbb{T}^d}\int_{\mathbb{R}^d}  f_\tau^2(s) (1+\frac{\vert r\vert^2}{2})^{\alpha}  dr dxds.
\end{align*}
 \item Concerning  the term $\displaystyle\dfrac{1}{\zeta\alpha\tau}\int_0^t\int_{\mathbb{T}^d}\int_{\mathbb{R}^d} \alpha [\Div_r( f_\tau(s)r)]_\delta T_M[f_\tau(s)]_\delta dr dxds$, we have
 \begin{align*}
      [\Div_r( f_\tau(s)r)]_\delta T_M[f_\tau(s)]_\delta=d[f_\tau(s)]_\delta^2T_M+r\cdot\nabla_r[f_\tau(s)]_\delta[f_\tau(s)]_\delta T_M,
 \end{align*}
 by using    \eqref{regularity-f} and the properties of convolution, it is possible to  let $\delta\to 0$ and obtain 
 \begin{align*}
    \lim_{\delta\to 0}& \int_0^t\int_{\mathbb{T}^d}\int_{\mathbb{R}^d} [\Div_r( f_\tau(s)r)]_\delta T_M[f_\tau(s)]_\delta dr dxds\\=d&\int_0^t\int_{\mathbb{T}^d}\int_{\mathbb{R}^d} f_\tau^2(s)T_M+r\cdot\nabla_rf_\tau(s)f_\tau(s) T_Mdr dxds=\int_0^t\int_{\mathbb{T}^d}\int_{\mathbb{R}^d} \Div_r( f_\tau(s)r) T_Mf_\tau(s) dr dxds\\
   =&\dfrac{d}{2}\int_0^t\int_{\mathbb{T}^d}\int_{\mathbb{R}^d} f_\tau^2(s)T_M drdxds-\dfrac{1}{2}\int_0^t\int_{\mathbb{T}^d}\int_{\mathbb{R}^d} f_\tau^2(s)r\cdot\nabla_rT_M drdxds\\
   =&\dfrac{d}{2}\int_0^t\int_{\mathbb{T}^d}\int_{\mathbb{R}^d} f_\tau^2(s)T_M drdxds-\dfrac{\alpha}{2}\int_0^t\int_{\mathbb{T}^d}\int_{\mathbb{R}^d} f_\tau^2(s)\vert r\vert^2 (1+\frac{\vert r\vert^2}{2})^{\alpha-1}T_M^\prime drdxds\\
   \leq & \dfrac{d}{2}\int_0^t\int_{\mathbb{T}^d}\int_{\mathbb{R}^d} f_\tau^2(s)(1+\frac{\vert r\vert^2}{2})^{\alpha} drdxds.
 \end{align*}
    \item The term $\displaystyle\dfrac{1}{\zeta\alpha\tau}\int_0^t\langle[\Delta_rf_\tau(s) ]_\delta, T_M[f_\tau(s)]_\delta\rangle ds.$ Note that 
    \begin{align*}
    \int_0^t\langle[\Delta_rf_\tau(s) ]_\delta, T_M[f_\tau(s)]_\delta\rangle ds= -\int_0^t\int_{\mathbb{T}^d}\int_{\mathbb{R}^d}   \vert [\nabla_r f_\tau(s)]_\delta\vert^2 T_M+[\nabla_r f_\tau(s)]_\delta\cdot \nabla_rT_M[f_\tau(s)]_\delta drdxds,
    \end{align*}
  by using \eqref{regularity-f} we can pass to the limit as $\delta\to 0$ and get
  \begin{align*}
  \lim_{\delta\to 0}  &\int_0^t\langle[\Delta_rf_\tau(s) ]_\delta, T_M[f_\tau(s)]_\delta\rangle ds= -\int_0^t\int_{\mathbb{T}^d}\int_{\mathbb{R}^d}   \vert \nabla_r f_\tau(s)\vert^2 T_M+\nabla_r f(s)\cdot \nabla_rT_Mf_\tau(s) drdxds\\
  &=-\int_0^t\int_{\mathbb{T}^d}\int_{\mathbb{R}^d}   \vert \nabla_r f_\tau(s)\vert^2 T_M+\alpha\nabla_r f_\tau(s)\cdot r (1+\frac{\vert r\vert^2}{2})^{\alpha-1}T_M^\prime f_\tau(s) drdxds.
    \end{align*}
    By using Young inequality we get
    \begin{align*}
     \vert \int_0^t\int_{\mathbb{T}^d}\int_{\mathbb{R}^d}   \alpha\nabla_r f_\tau(s)\cdot r (1+\frac{\vert r\vert^2}{2})^{\alpha-1}T_M^\prime f_\tau(s) drdxds\vert \leq   \alpha^2 \int_0^t\int_{\mathbb{T}^d}\int_{\mathbb{R}^d}f_\tau^2(s)(1+\frac{\vert r\vert^2}{2})^{\alpha} drdxds\\+\dfrac{1}{2}\int_0^t\int_{\mathbb{T}^d}\int_{\mathbb{R}^d} \vert \nabla_rf_\tau(s)\vert^2(1+\frac{\vert r\vert^2}{2})^{\alpha} drdxds.
    \end{align*}
    Thus, we obtain
    \begin{align*}
  \limsup_{M\to +\infty}\lim_{\delta\to 0}  \int_0^t\langle[\Delta_rf_\tau(s) ]_\delta, T_M[f_\tau(s)]_\delta\rangle ds& \leq  -\liminf_{M\to +\infty}\int_0^t\int_{\mathbb{T}^d}\int_{\mathbb{R}^d}   \vert \nabla_r f_\tau(s)\vert^2 T_M drdxds\\&+\dfrac{1}{2}\int_0^t\int_{\mathbb{T}^d}\int_{\mathbb{R}^d} \vert \nabla_rf_\tau(s)\vert^2(1+\frac{\vert r\vert^2}{2})^{\alpha} drdxds\\
  &+\alpha^2 \int_0^t\int_{\mathbb{T}^d}\int_{\mathbb{R}^d}f_\tau^2(s)(1+\frac{\vert r\vert^2}{2})^{\alpha} drdxds.
    \end{align*}
    \item  The term $\displaystyle\dfrac{1}{\zeta\alpha\tau}\int_0^t\langle\dfrac{1}{2}[\Div_r(A_b(r) \nabla_rf_\tau(s))]_\delta, T_M[f_\tau(s)]_\delta \rangle ds.$ We have 
    \begin{align*}
        &\int_0^t\langle[\Div_r(A_b(r) \nabla_rf_\tau(s))]_\delta, T_M[f_\tau(s)]_\delta \rangle ds\\&=-\int_0^t\int_{\mathbb{T}^d}\int_{\mathbb{R}^d} \Big((A_b(r) \nabla_r[f_\tau(s)]_\delta)\cdot  \nabla_r[f_\tau(s)]_\delta T_M\\&\qquad\qquad\qquad+\alpha(A_b(r) \nabla_r[f_\tau(s)]_\delta)\cdot  r (1+\frac{\vert r\vert^2}{2})^{\alpha-1}T_M^\prime [f_\tau(s)]_\delta \Big)drdxds.
    \end{align*}
    By using \eqref{regularity-f}, we can let $\delta \to 0$ in the last equality and obtain 
    \begin{align*}
  &    \lim_{\delta\to 0}  -\int_0^t\langle[\Div_r(A_b(r) \nabla_rf_\tau(s))]_\delta, T_M[f_\tau(s)]_\delta \rangle ds\\&=\int_0^t\int_{\mathbb{T}^d}\int_{\mathbb{R}^d} (A_b(r) \nabla_rf_\tau(s))\cdot  \nabla_rf_\tau(s) T_M+\alpha(A_b(r) \nabla_r f_\tau(s))\cdot  r (1+\frac{\vert r\vert^2}{2})^{\alpha-1}T_M^\prime f_\tau(s) drdxds\\
      &=\int_0^t\int_{\mathbb{T}^d}\int_{\mathbb{R}^d} (b+1)\vert r\vert^2\vert \nabla_r f_\tau(s)\vert^2T_M-b\vert r\cdot \nabla_rf_\tau(s)\vert^2 T_Mdrdxds\\&\quad+\alpha\int_0^t\int_{\mathbb{T}^d}\int_{\mathbb{R}^d}(A_b(r) \nabla_r f_\tau(s))\cdot  r (1+\frac{\vert r\vert^2}{2})^{\alpha-1}T_M^\prime f_\tau(s) drdxds\\
      &\geq \int_0^t\int_{\mathbb{T}^d}\int_{\mathbb{R}^d} \vert r\vert^2\vert \nabla_r f_\tau(s)\vert^2T_Mdrdxds\\&\quad+\alpha\int_0^t\int_{\mathbb{T}^d}\int_{\mathbb{R}^d}(A_b(r) \nabla_r f_\tau(s))\cdot  r (1+\frac{\vert r\vert^2}{2})^{\alpha-1}T_M^\prime f_\tau(s) drdxds.
    \end{align*}
On the other hand, we have
\begin{align*}
    &\vert (A_b(r) \nabla_r f_\tau(s))\cdot  r (1+\frac{\vert r\vert^2}{2})^{\alpha-1}T_M^\prime f_\tau(s)\vert\\&\leq 3\vert\vert r\vert^2 \nabla_r f_\tau(s))\cdot  rT_M^\prime f_\tau(s)(1+\frac{\vert r\vert^2}{2})^{\alpha-1}\vert\leq 6\vert \nabla_rf_\tau(s)\vert \vert r\vert f(s)(1+\frac{\vert r\vert^2}{2})^{\alpha}\\
    &\leq \dfrac{1}{2\alpha}\vert \nabla_rf_\tau(s)\vert^2 \vert r\vert^2 (1+\frac{\vert r\vert^2}{2})^{\alpha}+ 18\alpha f_\tau^2(s)(1+\frac{\vert r\vert^2}{2})^{\alpha},
\end{align*}
therefore, we get
  \begin{align*}
      \lim_{\delta\to 0}  -\int_0^t&\langle[\Div_r(A_b(r) \nabla_rf_\tau(s))]_\delta, T_M[f_\tau(s)]_\delta \rangle ds\\
      &\geq \int_0^t\int_{\mathbb{T}^d}\int_{\mathbb{R}^d} \vert r\vert^2\vert \nabla_r f_\tau(s)\vert^2T_Mdrdxds\\&-\int_0^t\int_{\mathbb{T}^d}\int_{\mathbb{R}^d}\dfrac{1}{2}\vert \nabla_rf_\tau(s)\vert^2 \vert r\vert^2 (1+\frac{\vert r\vert^2}{2})^{\alpha}+ 18\alpha^2 f_\tau^2(s)(1+\frac{\vert r\vert^2}{2})^{\alpha}
   drdxds.  \end{align*}
   Hence by using monotone convergence theorem, we get
   \begin{align*}
 &   \liminf_{M\to +\infty}  \lim_{\delta\to 0}  -\int_0^t\langle[\Div_r(A_b(r) \nabla_rf_\tau(s))]_\delta, T_M[f_\tau(s)]_\delta \rangle ds\\
      &\geq \int_0^t\int_{\mathbb{T}^d}\int_{\mathbb{R}^d} \dfrac{\vert r\vert^2}{2} \vert \nabla_r f_\tau(s)\vert^2(1+\frac{\vert r\vert^2}{2})^{\alpha}drdxds-18\alpha^2\int_0^t\int_{\mathbb{T}^d}\int_{\mathbb{R}^d}  f_\tau^2(s)(1+\frac{\vert r\vert^2}{2})^{\alpha}
   drdxds.  \end{align*}
\end{itemize}
By gathering the previous estimates we get
\begin{align*}
    &\int_{\mathbb{T}^d}\int_{\mathbb{R}^d}f_\tau(t)^2(1+\frac{\vert r\vert^2}{2})^\alpha  drdx+\dfrac{1}{2\zeta\tau\alpha}\int_0^t\int_{\mathbb{T}^d}\int_{\mathbb{R}^d}  \vert \nabla_r f_\tau(s)\vert^2(1+\frac{\vert r\vert^2}{2})^{\alpha+1}drdxds\\
&\leq  \int_{\mathbb{T}^d}\int_{\mathbb{R}^d}(f^\tau_0)^2(1+\frac{\vert r\vert^2}{2})^\alpha  drdx+(\dfrac{d/2+10\alpha}{\zeta\tau}+2\alpha\Vert \nabla_xu_L \Vert_\infty )\int_0^t\int_{\mathbb{T}^d}\int_{\mathbb{R}^d} f_\tau^2(s)(1+\frac{\vert r\vert^2}{2})^{\alpha} drdxds.
\end{align*}
Gronwall's inequality ensures the existence of $\mathbf{C}=\dfrac{d/2+10\alpha}{\zeta\tau}+2\alpha\Vert \nabla_xu_L \Vert_\infty$ such that 
\begin{align}
   \sup_{t\in [0, T]} \int_{\mathbb{T}^d}\int_{\mathbb{R}^d}f_\tau(t)^2(1+\frac{\vert r\vert^2}{2})^\alpha  drdx&+\dfrac{1}{2\zeta\tau\alpha}\int_0^T\int_{\mathbb{T}^d}\int_{\mathbb{R}^d}  \vert \nabla_r f_\tau(s)\vert^2(1+\frac{\vert r\vert^2}{2})^{\alpha+1}drdxds\notag\\
&\leq  e^{\mathbf{C}T}\int_{\mathbb{T}^d}\int_{\mathbb{R}^d}(f_0^\tau)^2(1+\frac{\vert r\vert^2}{2})^\alpha  drdx\leq  \mathbf{\Lambda}e^{\mathbf{C}T}.\end{align}
\end{proof}

\section{Weighted Poincar\'e inequality}\label{Appendix-B}
Let  $\vartheta_\sigma$ be  a family of the generalized Cauchy distributions on $\mathbb{R}^d,$ which have densities
\begin{align*}
    \dfrac{d\vartheta_\sigma(x)}{dx}=\dfrac{1}{A}(1+\vert x\vert^2)^{-\sigma}; \quad \sigma>\dfrac{d}{2} \text{ and a normalizing constant } A. 
\end{align*}
Recall that $
    \text{Var}_{\vartheta_\sigma}(g)= \int_{\mathbb{R}^d} g^2 d\vartheta_\sigma-(\int_{\mathbb{R}^d} g d\vartheta_\sigma)^2.$
We have the following result.
\begin{theorem}\label{thm1-appendixB}(\cite[Theorem 3.1]{bobkov09weighted}) The generalized Cauchy distribution $\vartheta_\sigma$ with $\sigma \geq d $ satisfies
the weighted Poincaré-type inequality
\begin{align}\label{ineq-B-1}
    \text{Var}_{\vartheta_\sigma}(g) \leq \mathbf{C}_\sigma\int_{\mathbb{R}^d} \vert \nabla g(x)\vert^2(1+\vert x\vert^2) d\vartheta_\sigma(x)
\end{align}
    for all bounded smooth functions $g$ on $\mathbb{R}^d$ with $\mathbf{C}_\sigma=\frac{1}{2(\sigma -1)} (\sqrt{1+\frac{2}{\sigma-1}}+\sqrt{\frac{2}{\sigma-1}})^2.$
\end{theorem}
The authors in \cite{bobkov09weighted} state  that  \autoref{thm1-appendixB} can involve the values $\frac{d}{2}<\sigma \leq d$ with constant depending  on the dimension $d$  (see \cite[page 415]{bobkov09weighted}). However,   general results on Riemannian manifold  \cite{bonnefont2016spectral}, based on the   spectral gap estimate, guarantee weighted Poincaré inequality with  explicit constants. Let us recall them (see \cite{bonnefont2016spectral}).
 Following \cite[Corollary 5.2 ($d= 3$) and Corollary 5.3  ($d=2$)]{bonnefont2016spectral}, the constant $C_\sigma$ in  \eqref{ineq-B-1} is given by
\begin{align}\label{constant-B-2}
    \mathbf{C}_\sigma=\dfrac{1}{(\sigma-\frac{d}{2})^2} \text{ if } \frac{d}{2}<\sigma \leq d;  \quad d=2,3.
\end{align}
On the other hand,
we recall that Poincaré  inequality is  understood in the following sense: if the right-hand side is finite, then the function is square-integrable, and the inequality holds true.\\
Therefore, based on \cite{bobkov09weighted,bonnefont2016spectral} briefly recalled above, we present the following result, which serves in our analysis. Namely, we consider
   $\nu_\sigma$,   a family of the generalized Cauchy distributions on $\mathbb{R}^d,$ which have densities
\begin{align*}
    \dfrac{d\nu_\sigma(y)}{dy}=\dfrac{1}{Z}(1+\dfrac{\vert y\vert^2}{2})^{-\sigma}; \quad \sigma>\dfrac{d}{2} \text{ and a normalizing constant } Z. 
\end{align*}
 \begin{proposition}\label{proposition-Poincare-weight}
    Let $d\in \{2,3\}$,  $\sigma > \dfrac{d}{2}$ and   $g:\mathbb{R}^d\to \mathbb{R}$. Assume that $g$ satisfies
    \begin{align*}  \int_{\mathbb{R}^d} \vert \nabla g(y)\vert^2(1+\dfrac{\vert y\vert^2}{2}) d\nu_\sigma(y) <+\infty.
    \end{align*}
    Then,  the weighted Poincaré-type inequality holds
    \begin{align}\label{ineq-B-2}
    \text{Var}_{\nu_\sigma}(g) \leq C_\sigma\int_{\mathbb{R}^d} \vert \nabla g(y)\vert^2(1+\dfrac{\vert y\vert^2}{2}) d\nu_\sigma(y),
\end{align}
where
\begin{align*}
    C_\sigma&=\begin{cases}
        \frac{1}{4(\sigma -1)} (\sqrt{1+\frac{2}{\sigma-1}}+\sqrt{\frac{2}{\sigma-1}})^2 & \text{ if } d \leq  \sigma ,\\
  \dfrac{1}{2(\sigma-\frac{d}{2})^2} &\text{ if } \frac{d}{2}<\sigma \leq d.
    \end{cases}
    \end{align*}
\end{proposition}
\begin{proof}
  By using \eqref{ineq-B-1}, \eqref{constant-B-2}, using the  change of variable $x=\frac{y}{\sqrt{2}}$ and noticing that 
   \begin{align*}
       Z=\int_{\mathbb{R}^d} (1+\dfrac{\vert y\vert^2}{2})^{-\sigma} dy=2^{\frac{d}{2}}\int_{\mathbb{R}^d} (1+\vert x\vert^2)^{-\sigma} dx=2^{\frac{d}{2}}A,
   \end{align*}
   we obtain \eqref{ineq-B-2}.
\end{proof}

\section{Commutator estimate}
\begin{lemma}\label{Lemma-tech-1}
    Let $ p\in [1,+\infty[$ and  $h\in L^p(\mathbb{T}^d\times \mathbb{R}^d\times(0,T))$.   Then, we have
    \begin{align*}
    \displaystyle\lim_{\delta \to 0}  \big([u_L(s)\cdot \nabla_xh(s)]_\delta -   u_L(s)\cdot \nabla_x[h(s)]_\delta \big)=0 \text{ in } L^p(\mathbb{T}^d\times \mathbb{R}^d\times(0,T)).
    \end{align*}
    Similarly, if $h\in L^p(\mathbb{T}^d\times(0,T))$ only.   Then, we have
    \begin{align*}
    \displaystyle\lim_{\delta \to 0}  \big([u_L(s)\cdot \nabla_xh(s)]_\delta -   u_L(s)\cdot \nabla_x[h(s)]_\delta \big)=0 \text{ in } L^p(\mathbb{T}^d\times(0,T)).
    \end{align*}
\end{lemma}
\begin{proof} 
Set $r_\delta(s):=[u_L(s)\cdot \nabla_xh(s)]_\delta -   u_L(s)\cdot \nabla_x[h(s)]_\delta,$
we have
   a.e.  $ s\in[0,T]$
\begin{align}
&
\left\Vert r_{\delta}(s)\right\Vert _{L^{p}(\T^d\times\R^d)}\leq C\left\Vert \nabla u_L(s)\right\Vert
_{\infty}\left\Vert h(s)\right\Vert _{L^{p}(\T^d\times\R^d)},   \label{lem-17-eqn1-2-*}%
\\
&\lim_{\delta}r_{\delta}=0\text{ in    } L^{p}(\T^d\times\R^d\times (0,T))  \label{lem-17-eqn2-2-*},%
\end{align}
where   $C>0$   independent of  $\delta$.   Indeed, let us show \eqref{lem-17-eqn1-2-*}, note    that
\[
r_{\delta}\left(s,  x,r\right)  =-\int_{\R^d}\left(  u_L\left(  x,s\right)  -u_L\left(  y,s\right)
\right)  \cdot\nabla_{x}\rho_{\delta}\left(  x-y\right)  h\left(s,  y,r\right)  dy.
\]%
Consider    the following   change  of  variables   $z=\dfrac{x-y}{\delta}$  to  get 
\begin{align*}
    r_{\delta}\left(s,  x,r\right)  &=-\int_{\R^2}  \dfrac{u_L\left( x ,s\right)  -u_L\left(  x-\delta  z,s\right)
}{\delta}  \cdot\nabla\rho\left(  z\right)  h\left(s,   x-\delta  z,r\right)  dz\\
&=-\int_{\R^d}  \int_0^1 \nabla u_L(x-\gamma\delta z)  z d\gamma \cdot\nabla\rho\left(  z\right)  h\left(s,  x-\delta  z,r\right)  dz.
\end{align*}
where we used mean value theorem \textit{i.e.}
$   u_L\left( x ,s\right)  -u_L\left(  x-\delta  z,s\right)=\delta\int_0^1 \nabla u_L(x-\gamma\delta z)  z d\gamma.$
Recall that    $\text{supp}\rho\subset    B(0,1)$, thus  
\begin{align*}
  \vert  r_{\delta}\left(s,  x,r\right)\vert&\leq \Vert    \nabla   u_L(s)\Vert_{\infty}\int_{\text{supp}\rho}\vert z\vert  \vert\nabla\rho\left(  z\right)\vert  \vert    h\left(s,  x-\gamma\delta z,r\right)\vert  dz\\
  &\leq    \Vert \nabla   u_L(s)\Vert_{\infty} \Vert\nabla\rho \Vert_{C^1_b}\int_{\text{supp}\rho}    \vert    h\left(s,  x-\gamma\delta z,r\right)\vert  dz
\end{align*}
  and  by H\"older’s inequality and  Fubini's theorem, we get
\begin{align*}
    \Vert   r_{\delta}(s)\Vert_{L^{p}(\T^d\times\R^d)}^p&\leq \Vert    \nabla   u_L(s)\Vert_{\infty}^p\Vert\rho \Vert_{C^1_b}^p\int_{\R^d}\int_{\T^d}  (\int_{\text{supp}\rho} \vert    h\left(s,  x-\gamma\delta z,r\right)\vert  dz)^pdxdr\\
    &\leq \Vert    \nabla   u_L(s)\Vert_{\infty}^p\Vert\rho \Vert_{C^1_b}^p \int_{\text{supp}\rho}\int_{\R^d}\int_{\T^d}   \vert    h\left(s,  x-\gamma\delta z,r\right)\vert^p  dxdrdz\\
    &\leq \Vert    \nabla   u_L(s)\Vert_{\infty}^p\Vert\rho \Vert_{C^1_b}^p  \Vert    h\left(s\right)\Vert_{L^{p}(\T^d\times\R^d)}^p, 
 \end{align*}
 by using the  periodicity with respect to the $x$-variable.
 Concerning    \eqref{lem-17-eqn2-2-*},  we  have
\begin{equation*}
L^p(\T^d\times\R^d\times  (0,T))\mbox{-}\lim_{\delta\rightarrow0}r_\delta(s)=-h(s,\cdot)\left(\int_{\R^2}\nabla    u_L(s,\cdot)z\cdot\nabla\rho(z)dz\right).
\end{equation*} Indeed,  by using  that $z\in \text{supp}\rho:=\mathbf{F} \subset B(0,1)$, we  get
\begin{align*}
&\int_0^T\int_{\R^d}\int_{\T^d}\vert \int_{\R^d}  \int_0^1 \nabla u_L(x-\gamma\delta z)  z d\gamma \cdot\nabla\rho\left(  z\right)  h\left(s,  x-\delta  z,r\right)  dz\\&\qquad\qquad\qquad\qquad-    h(s,x,r)\left(\int_{\R^d}\nabla  u_L(s,x)z\cdot\nabla\rho(z)dz\right)\vert^pdxdr ds\\
&\leq\int_0^T\int_{\R^d}\int_{\T^d} \int_{\mathbf{F}}  \int_0^1\vert   \nabla u_L(x-\gamma\delta z)  z  \cdot\nabla\rho\left(  z\right)  h\left(s,  x-\delta  z,r\right)  -h(s,x,r)\nabla u_L(s,x)z\cdot\nabla\rho(z)\vert^2d\gamma  dzdxdrds\\
&\leq \int_0^T\int_{\mathbf{F}}  \int_0^1\int_{\R^d}\int_{\T^d} \vert   [h\left(s,  x-\delta  z,r\right)\nabla u_L(x-\gamma\delta z)      -h(s,x,r)\nabla u_L(s,x)]z  \cdot\nabla\rho\left(  z\right)\vert^pdxdrd\gamma  dz ds
\\
&\leq  \Vert\rho \Vert_{C^1_b}^p\int_0^T\int_{\mathbf{F}}  \int_0^1\int_{\R^d}\int_{\T^d} \vert  h\left(s,  x-\delta  z,r\right)\nabla u_L(x-\gamma\delta z)      -h(s,x,r)\nabla u_L(s,x) \vert^pdxdrd\gamma  dz ds\\
&\leq  2^{p-1}\Vert\rho \Vert_{C^1_b}^p\int_0^T\int_{\mathbf{F}}  \int_0^1\int_{\R^d}\int_{\T^d} \vert  h\left(s,  x-\delta  z,r\right)[\nabla u_L(x-\gamma\delta z)-\nabla u_L(x-\delta z)] \vert^pdxdrd\gamma  dz ds\\
&+2^{p-1}\Vert\rho \Vert_{C^1_b}^p\int_0^T\int_{\mathbf{F}}  \int_0^1\int_{\R^d}\int_{\T^d} \vert  h\left(s,  x-\delta  z,r\right)\nabla u_L(x-\delta z)      -h(s,x,r)\nabla u_L(s,x) \vert^pdxdrd\gamma  dz ds\\
&:=I^1_\delta+I^2_\delta.
\end{align*}
Using the continuity of translations in $L^2(\T^d\times \R^d \times (0,T))$ for the function $h\nabla u_L$, we get  $\displaystyle\limsup_{\delta\to 0}        I^2_\delta=0.$
Concerning  $I^1_\delta$,    by   mean-value theorem we  get    
\begin{align*}
   I_\delta^1= &2^{p-1}\Vert\rho \Vert_{C^1_b}^p\int_0^T\int_{\mathbf{F}}  \int_0^1\int_{\R^d}\int_{\T^d} \vert  h\left(s,  x-\delta  z,r\right)[\nabla u_L(x-\gamma\delta z)-\nabla u_L(x-\delta z)] \vert^pdxdrd\gamma  dz ds\\ 
   & \leq  2^{p-1}\Vert\rho \Vert_{C^1_b}^p\delta^p\Vert u_L\Vert_{C^2_b}^p\int_0^T\int_{\mathbf{F}}  \int_{\R^d}\int_{\T^d} \vert  h\left(s,  x-\delta  z,r\right)\vert^pdxdr  dz ds\\
  & \leq 2^{p-1}\Vert\rho \Vert_{C^1_b}^p\delta^p\Vert u_L\Vert_{C^2_b}^p \Vert    h\left(s\right)\Vert_{L^{p}(\T^d\times\R^d\times (0,T))}^p \to 0 \text{ as } \delta \to 0.
\end{align*}
Finally,    since   $\rho$  is  radially    symmetric,  we  get
  $\int_{\R^d}z_i\partial_j\rho(z)dz=-\delta_{ij}$  and so 
$$\int_{\R^d}\nabla u_L(s,\cdot)z\cdot\nabla\rho(z)dz=-\Div u_L=0,$$
which   gives   \eqref{lem-17-eqn2-2-*}. Then, the first claim of  \autoref{Lemma-tech-1} follows by integrating from $0$ to $T$ the above estimates. For the second part, it is clear that it follows, by neglecting integration with respect to the variable $r$ in the above calculation.
\end{proof}
\begin{acknowledgements}  The research of  the author is
funded by the European Union (ERC, NoisyFluid, No. 101053472). Views and opinions expressed are however those of the author only and do not necessarily reflect those of the European Union or the European Research Council. Neither the European Union nor the granting authority can be held responsible for them. The author  would like to thank  Franco Flandoli for  many valuable discussions. 
\end{acknowledgements}

\bibliographystyle{plain} 
\bibliography{High_turbulent-limit}{}

\begin{thebibliography}{10}

\bibitem{afonso2005nonlinear}
M~M. Afonso and D.~Vincenzi.
\newblock Nonlinear elastic polymers in random flow.
\newblock {\em Journal of Fluid Mechanics}, 540:99--108, 2005.

\bibitem{Balk}
E~Balkovsky, A~Fouxon, and V~Lebedev.
\newblock Turbulent dynamics of polymer solutions.
\newblock {\em Physical review letters}, 84(20):4765--4768, 2000.

\bibitem{Bensoussan1978asymptotic}
A.~Bensoussan, J-L Lions, and G.~Papanicolaou.
\newblock {\em Asymptotic analysis for periodic structures}, volume 374.
\newblock American Mathematical Soc., 1978.

\bibitem{bobkov09weighted}
S.~G. Bobkov and M.~Ledoux.
\newblock Weighted {P}oincar{\'e}-type inequalities for {C}auchy and other
  convex measures.
\newblock {\em The Annals of Probability}, 37(2):403--427, 2009.

\bibitem{bonnefont2016spectral}
M.~Bonnefont, A.~Joulin, and Y.~Ma.
\newblock Spectral gap for spherically symmetric log-concave probability
  measures, and beyond.
\newblock {\em Journal of Functional Analysis}, 270(7):2456--2482, 2016.

\bibitem{Brezis}
H.~Brezis.
\newblock {\em Functional analysis, {S}obolev spaces and partial differential
  equations}, volume~2.
\newblock Springer, 2011.

\bibitem{BriTL09}
C.~L. Bris and T.~Lelièvre.
\newblock Multiscale modelling of complex fluids: a mathematical initiation.
\newblock {\em Multiscale modeling and simulation in science}, pages 49--137,
  2009.

\bibitem{ButFlaLuo}
F.~Butori, F.~Flandoli, and E.~Luongo.
\newblock On the {I}t\^{o}-{S}tratonovich {D}iffusion {L}imit for the
  {M}agnetic {F}ield in a 3{D} {T}hin {D}omain.
\newblock {\em arXiv preprint arXiv:2401.15701}, 2024.

\bibitem{BFLT2024}
F.~Butori, F.~Flandoli, E.~Luongo, and Y.~Tahraoui.
\newblock Background {V}lasov equations and {Y}oung measures for passive scalar
  and vector advection equations under special stochastic scaling limits.
\newblock {\em arXiv preprint arXiv:2407.10594}, 2024.

\bibitem{FENE-YAS-FED}
F.~Butori and Y.~Tahraoui.
\newblock {FENE} dumbbell model and turbulence via stochastic scaling and
  singular limits.
\newblock {\em preprint}.

\bibitem{diperna1989ordinary}
R.~J. DiPerna and P.L. Lions.
\newblock Ordinary differential equations, transport theory and {S}obolev
  spaces.
\newblock {\em Inventiones mathematicae}, 98(3):511--547, 1989.

\bibitem{Edwards}
R.~E. Edwards.
\newblock {\em Functional analysis}.
\newblock New York: Dover Publications Inc., 1995.

\bibitem{Masmoudi-Vlasov}
N.~El~Ghani and N.~Masmoudi.
\newblock Diffusion limit of the {V}lasov-{P}oisson-{F}okker-{P}lanck system.
\newblock {\em Commun. Math. Sci.}, 8(2):463--479, 2010.

\bibitem{evans2022partial}
Lawrence~C Evans.
\newblock {\em Partial Differential Equations}, volume~19.
\newblock American Mathematical Soc., 2022.

\bibitem{fetecau2015}
RC~Fetecau and W.~Sun.
\newblock First-order aggregation models and zero inertia limits.
\newblock {\em Journal of Differential Equations}, 259(11):6774--6802, 2015.

\bibitem{Flandoli2021delayed}
F.~Flandoli, L.~Galeati, and D.~Luo.
\newblock Delayed blow-up by transport noise.
\newblock {\em Communications in Partial Differential Equations},
  46(9):1757--1788, 2021.

\bibitem{Flandoli2021scaling}
F.~Flandoli, L.~Galeati, and D.~Luo.
\newblock Scaling limit of stochastic 2d {E}uler equations with transport
  noises to the deterministic {N}avier--{S}tokes equations.
\newblock {\em Journal of Evolution Equations}, 21(1):567--600, 2021.

\bibitem{Flandoli2024quantitative}
F.~Flandoli, L.~Galeati, and D.~Luo.
\newblock Quantitative convergence rates for scaling limit of spdes with
  transport noise.
\newblock {\em Journal of Differential Equations}, 394:237--277, 2024.

\bibitem{FL1}
F.~Flandoli and D.~Luo.
\newblock High mode transport noise improves vorticity blow-up control in 3{D}
  {N}avier--{S}tokes equations.
\newblock {\em Probability Theory and Related Fields}, 180:309--363, 2021.

\bibitem{FlandoliLuo2024}
F.~Flandoli and D.~Luo.
\newblock On the {B}oussinesq hypothesis for a stochastic {P}roudman-{T}aylor
  model.
\newblock {\em SIAM J. Math. Anal}, 56(3):3886–3923, 2024.

\bibitem{FlaElibook}
F.~Flandoli and E.~Luongo.
\newblock {\em Stochastic partial differential equations in fluid mechanics},
  volume 2330.
\newblock Springer Nature, 2023.

\bibitem{FlaTah24}
F.~Flandoli and Y.~Tahraoui.
\newblock Stretching of polymers and turbulence: Fokker {P}lanck equation,
  special stochastic scaling limit and stationary law.
\newblock {\em Journal of Differential Equations}, 452:113789, 2026.

\bibitem{Galeati}
L.~Galeati.
\newblock On the convergence of stochastic transport equations to a
  deterministic parabolic one.
\newblock {\em Stochastics and Partial Differential Equations: Analysis and
  Computations}, 8(4):833--868, 2020.

\bibitem{Gerashen}
S.~Gerashchenko, C.~Chevallard, and V.~Steinberg.
\newblock Single-polymer dynamics: Coil-stretch transition in a random flow.
\newblock {\em Europhysics Letters}, 71(2):221--227, 2005.

\bibitem{Ghattassi}
M.~Ghattassi, X.~Huo, and N.~Masmoudi.
\newblock On the diffusive limits of radiative heat transfer system {I}:
  Well-prepared initial and boundary conditions.
\newblock {\em SIAM Journal on Mathematical Analysis}, 54(5):5335--5387, 2022.

\bibitem{goudon2005}
T.~Goudon.
\newblock Hydrodynamic limit for the {V}lasov--{P}oisson--{F}okker--{P}lanck
  system: Analysis of the two-dimensional case.
\newblock {\em Mathematical Models and Methods in Applied Sciences},
  15(05):737--752, 2005.

\bibitem{Herda2018}
M.~Herda and L.~M. Rodrigues.
\newblock Large-time behavior of solutions to
  {V}lasov-{P}oisson-{F}okker-{P}lanck equations: from evanescent collisions to
  diffusive limit.
\newblock {\em Journal of Statistical Physics}, 170(5):895--931, 2018.

\bibitem{jabin2000}
P-E. Jabin.
\newblock Macroscopic limit of {V}lasov type equations with friction.
\newblock In {\em Annales de l'Institut Henri Poincar{\'e} C, Analyse non
  lin{\'e}aire}, volume~17, pages 651--672. Elsevier, 2000.

\bibitem{Jikov2012homogenization}
V.~V. Jikov, S.~M Kozlov, and O.~A. Oleinik.
\newblock {\em Homogenization of differential operators and integral
  functionals}.
\newblock Springer Science \& Business Media, 2012.

\bibitem{luo2025enhanced}
D.~Luo.
\newblock Enhanced dissipation for stochastic {N}avier--{S}tokes equations with
  transport noise.
\newblock {\em Journal of Dynamics and Differential Equations}, 37(1):859--894,
  2025.

\bibitem{MK}
A.~J Majda and P.~R Kramer.
\newblock Simplified models for turbulent diffusion: theory, numerical
  modelling, and physical phenomena.
\newblock {\em Physics reports}, 314(4-5):237--574, 1999.

\bibitem{Nieto2001high}
J.~Nieto, F.~Poupaud, and J.~Soler.
\newblock High-field limit for the {V}lasov-{P}oisson-{F}okker-{P}lanck system.
\newblock {\em Archive for rational mechanics and analysis}, 158:29--59, 2001.

\bibitem{Papini}
A.~Papini, F.~Flandoli, and R.~Huang.
\newblock Turbulence enhancement of coagulation: The role of eddy diffusion in
  velocity.
\newblock {\em Physica D: Nonlinear Phenomena}, 448:133726, 2023.

\bibitem{Pavliotis2008multiscale}
G.~Pavliotis and A.~Stuart.
\newblock {\em Multiscale methods: averaging and homogenization}.
\newblock Springer Science and Business Media, 2008.

\bibitem{PicardoLanceVinc}
J.~R. Picardo, L.C. VI M.~Plan Emmanuel, and D.~Vincenzi.
\newblock Polymers in turbulence: stretching statistics and the role of extreme
  strain rate fluctuations.
\newblock {\em Journal of Fluid Mechanics}, 969:A24, 2023.

\bibitem{Poupaud1992}
F.~Poupaud.
\newblock Runaway phenomena and fluid approximation under high fields in
  semiconductor kinetic theory.
\newblock {\em ZAMM-Journal of Applied Mathematics and Mechanics},
  72(8):359--372, 1992.

\bibitem{Poupaud2000}
F.~Poupaud and J.~Soler.
\newblock Parabolic limit and stability of the {V}lasov--{F}okker--{P}lanck
  system.
\newblock {\em Mathematical Models and Methods in Applied Sciences},
  10(07):1027--1045, 2000.

\bibitem{ReedSimon75}
M.~Reed and B.~Simon.
\newblock {\em Methods of modern mathematical physics II: Fourier Analysis,
  Self-Adjointness}.
\newblock 1975.

\bibitem{Rudin}
W.~Rudin.
\newblock Real and complex analysis.
\newblock 1974.

\end{thebibliography}

\end{document}